\documentclass[titlepage,11pt]{article}
\usepackage{latexsym}
\usepackage{amsfonts}
\usepackage{amsmath}
\usepackage{amsthm}
\usepackage{comment}
\usepackage{float}
\usepackage{tikz}
\usetikzlibrary{decorations.pathmorphing}
\usetikzlibrary{positioning,shapes,arrows,decorations.markings}
\usetikzlibrary{decorations.pathreplacing,calligraphy}
\usetikzlibrary{positioning,arrows.meta}
\usetikzlibrary{fit}
\usepackage{etoolbox}
\usepackage{hyperref}

\newcommand*{\halfway}{0.5*\pgfdecoratedpathlength+.5*3pt}

\def\Stealtharrow{{\arrow[xshift=2pt+3.2\pgflinewidth]{Stealth[scale=1.3]}}}

\tikzset{
    midarrow/.style={decoration={markings,mark=at position #1 with {\Stealtharrow}},postaction={decorate}},
    midarrow/.default=0.5
}

\tikzset{->-/.style 2 args={decoration={
  markings,
  mark=at position \halfway with {\arrow[rotate=#2]{Stealth[scale=1.3]}}},postaction={decorate}},
  ->-/.default = {0.55}{0}
}
\tikzset{
dot/.style = {circle, fill, minimum size=#1,
              inner sep=0pt, outer sep=0pt},
dot/.default = 5pt
}

\newcommand{\mac}{\mathcal}
\def\ins{\operatorname{ins}}
\newcommand{\vtx}[3][]{\node[dot, label={#1}] (#2) at (#3) {};}
\newcommand{\dedge}[3][]{\draw[midarrow] (#2) to[#1] (#3);}
\newcommand{\tdedge}[3][]{\draw[midarrow, line width=1.5] (#2) to[#1] (#3);}
\newcommand{\hdedge}[3][]{\draw[midarrow, very thick] (#2) to[#1] (#3);}
\newcommand{\sline}[3][]{\draw[very thick, decorate, decoration={snake,amplitude=1,segment length=10, post length=0,pre length=0}] (#2) to[#1] (#3);}

\newcommand\vset[3]{\node[very thick, draw, ellipse, inner sep=0, fit={+(-0.707,-0.353) +(0.707,0.353)}, label={center:#2}] (#1) at (#3) {};}
\newcommand\smallvset[3]{\node[very thick, draw, ellipse, inner sep=0, fit={+(-0.471,-0.235) +(0.471,0.235)}, label={center:#2}] (#1) at (#3) {};}

\newcommand{\from}{\leftarrow}
\def\DD{\text{-}}

\def\cupcup{\cup\dots\cup}

\newtheorem{thm}{}[section]
\newcounter{proof}
\newcounter{step}[proof]
\newenvironment{step}{\refstepcounter{step}\bigskip\noindent(\thestep) \em}{\bigskip}
\AtBeginEnvironment{proof}{\refstepcounter{proof}}

\usepackage{xpatch}
\newcommand{\prooffont}{\bfseries}
\xpatchcmd{\proof}{\itshape}{\prooffont}{}{}

\title{When all directed cycles have the same weight}
\author{Eli Berger\\
University of Haifa, Haifa, Israel
\and
    Daniel Carter\thanks{Supported by NSF grant DGE-2444107.}\\
Princeton University, Princeton, USA
\and
Paul Seymour\thanks{Supported by AFOSR grant
FA9550-22-1-0234, and NSF grant  DMS-2154169.}\\
Princeton University, Princeton, USA}

\date{October 17, 2025; revised \today}

\begin{document}
\hypersetup{pageanchor=false}
\maketitle
\begin{abstract}
A digraph $G$ is \emph{weightable} if its edges can be weighted with real numbers such that the total weight in each directed cycle 
equals 1. There are several equivalent conditions: that $G$ admits a $0/1$-weighting with the same property, or that
$G$ contains no subdivided ``double-cycle'' as a subdigraph, or that for every triple of vertices, all directed cycles containing 
all three pass through them in the same cyclic order. And there is quite a rich supply of such digraphs: for instance, any digraph 
drawn in the plane such that each of its directed cycles rotates clockwise around the origin is weightable (let us call such digraphs ``circular''), and there are weightable
planar digraphs with much more complicated structure than this. 

Until now the general structure of weightable digraphs was not known, and that is our objective in this paper. We will show that:
\begin{itemize}
\item there is a construction that builds every planar weightable digraph from circular digraphs; and
\item there is a (different) construction that builds every weightable digraph from planar ones.
\end{itemize}
We derive a poly-time algorithm
to test if a digraph is weightable.
 
\end{abstract}
\hypersetup{pageanchor=true}

\section{Introduction}
Graphs and digraphs in this paper are finite, and may have loops or parallel edges. 
Let $G$ be a digraph drawn in the plane (without crossings), where the origin belongs to one of the regions. 
Each edge $e$ of $G$ subtends an angle at the origin (a $w(e)$
fraction of a full rotation, say). If the drawing has the property that $w(e)>0$ for every edge, then every directed cycle clearly
must wind around the origin. But more than that; every directed cycle must wind around the origin exactly once, because curves that
wind more than once intersect themselves. Let us call such a drawing \emph{circular}.

This is an intriguing property of digraphs. It is related to a theorem of Thomassen~\cite{acyclic, thomassen}, that says in one form:
\begin{thm}\label{thomassenthm}
Let $G$ be a digraph with no loops or parallel edges such that every vertex has in-degree and out-degree at least two, and suppose there are $a,b\in V(G)$ such that every directed cycle contains at least one of $a,b$.
Then there is no directed cycle containing both $a,b$ if and only if 
$G$ admits a circular drawing.
\end{thm}
This version of Thomassen's result looks like it ought to be made somehow more general, and that brings us to the question that was the starting point of the research in this paper:
is there a theorem that says ``every appropriately connected digraph $G$ contains no thing of type X if and only if $G$ 
admits a circular drawing''? 

Let us say a \emph{weighting} of a digraph $G$ is a real-valued function $w:E(G)\to\mathbb{R}$, such that $w(C)=1$
for every directed cycle $C$, where $w(C)$ means $\sum_{e\in E(C)}w(e)$, and if $G$ admits a weighting, we say $G$ is 
\emph{weightable}. 
In a circular drawing, the function $w$ defined earlier is a weighting, and so 
digraphs with circular drawings are weightable.
For some time we hoped that for a converse, at least for sufficiently well-connected digraphs,
that, say, every strongly 2-connected, weakly 3-connected weightable digraph admits a circular drawing (we will define these terms later).
But this turns out to be false -- see Figure~\ref{fig:carter}.

\begin{figure}[ht]

\centering

\begin{tikzpicture}[auto=left]

\vtx{a1}{-3,1};
\vtx{a2}{-1,1};
\vtx{a3}{1,1};
\vtx{a4}{3,1};
\vtx{b1}{-3,-1};
\vtx{b2}{-1,-1};
\vtx{b3}{1,-1};
\vtx{b4}{3,-1};
\vtx{c1}{-2,2};
\vtx{c2}{0,2};
\vtx{c3}{2,2};
\vtx{c4}{4,0};
\vtx{c5}{2,-2};
\vtx{c6}{0,-2};
\vtx{c7}{-2,-2};
\vtx{c8}{-4,0};

\foreach \from/\to in {a1/a2,a2/a3,a3/a4,b2/b1,b3/b2,b4/b3,a2/b2,a4/b4, c1/c2,c2/c3,c5/c6,c6/c7,a1/c1,c1/a2,a2/c2,c2/a3,a3/c3,c3/a4,b4/c5,c5/b3,b3/c6,c6/b2,b2/c7,c7/b1,a4/c4,c4/b4,b1/c8}
\dedge{\from}{\to};
\tdedge{c8}{a1};
\tdedge{b1}{a1};
\tdedge{b3}{a3};

\tdedge[bend left = 45]{c8}{c1};
\dedge[bend left = 45]{c7}{c8};
\dedge[bend left = 45]{c4}{c5};
\dedge[bend left = 45]{c3}{c4};

\end{tikzpicture}

\caption{A weightable digraph with no circular drawing. The four thick edges have weight one, and the others have weight zero.} \label{fig:carter}
\end{figure}
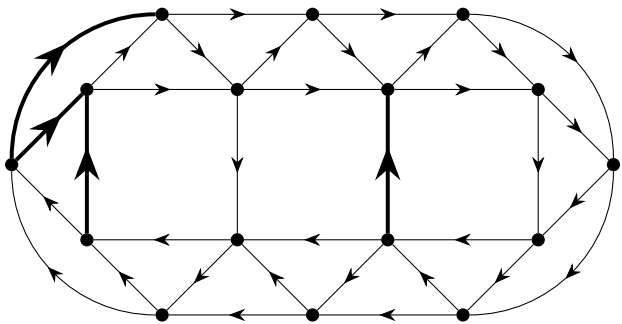

We still have not come up with the characterization we hoped for of the digraphs with circular drawings, but we 
now know which 
digraphs are weightable, and that is what we will explain in this paper.

An answer to the question ``which digraphs are weightable?'' could mean several things:
\begin{enumerate}
\item A characterization of the minimal digraphs (under subdigraph containment) that are not weightable.
\item A poly-time algorithm to test whether an input digraph is weightable.
\item A poly-time algorithm to find a weighting if there is one.
\item A poly-time algorithm to test if a given function on $E(G)$ is a weighting.
\item A method of construction that will build all (ond only) weightable digraphs by piecing together well-understood ones 
in prescribed ways.
\end{enumerate}
The first of these was already known~\cite{tricycles}, but all the others are new and obtained in this paper. 
We will show in the next section that the third and fourth are both consequences of the second. 
The second is a consequence of the fifth, the construction method, which is our main theorem.  We discuss this algorithm at the end of the paper.

Curiously, the construction method breaks into two parts: roughly, we will show how to build all planar weightable digraphs, by
piecing together those with circular embeddings; and then show how to build general weightable digraphs, by piecing together planar ones
(with a different construction).

Combining all of our constructions together yields the following theorem (we use some terms defined at the start of the next section):
\begin{thm}\label{summary}
Let $G$ be a 1-strong weightable digraph. Then at least one of the following is true:
\begin{enumerate}
    \item $G$ is not strongly 2-connected, and is obtained from two smaller 1-strong weightable digraphs by the construction described in~\ref{1reduce} and depicted in Figure~\ref{fig:1reduce}.
    \item $G$ is not 3-weak, and is obtained from two smaller 1-strong weightable digraphs by the construction described in~\ref{2weakreduce} and depicted in Figure~\ref{fig:2weakreduce}.
    \item $G$ is planar and admits a circular drawing.
    \item $G$ is planar, does not admit a circular drawing, and is obtained from two smaller 1-strong planar weightable digraphs by the construction described in~\ref{buildplanar} and depicted in Figure~\ref{fig:buildplanar}.
    \item $G$ is nonplanar, and there is a subset $Y\subseteq V(G)$ with $|Y|\in\{3,4\}$ such that $G\setminus Y$ has at least two weak 
components, and $G$ is obtained from two or three smaller 1-strong weightable digraphs by one of the constructions described 
in~\ref{firstcon}, \ref{secondcon}, or~\ref{thirdcon} and depicted in Figures~\ref{fig:firstcon}, \ref{fig:secondcon}, or~\ref{fig:thirdcon}, respectively.
\end{enumerate}
\end{thm}

There is another way to describe the 1-strong planar weightable digraphs, as planar digraphs that admit a kind of 
tree-decomposition called a ``bond carving'' of ``diwidth two''. This is discussed in Section~\ref{sec:carving}.
\section{Preliminaries}
First, let us state some standard definitions.
A graph $G$ is \emph{$k$-connected} if it has at least $k+1$ vertices and $G\setminus X$ is connected for every $X\subseteq V(G)$ with $|X|<k$.
The \emph{underlying graph} $G^-$ of a digraph $G$ is the graph obtained for forgetting the direction of all edges.
A digraph is \emph{weakly connected} if its underlying graph is connected, and \emph{weakly $k$-connected} or \emph{$k$-weak} if its underlying graph is
$k$-connected. (We usually write ``1-weak'' for ``weakly connected''.) A \emph{weak component} of a digraph is a maximal 1-weak subdigraph.

A \emph{dipath} is a directed path, and a \emph{dicycle} is a directed cycle.
A digraph $G$ is
\emph{strongly connected} if for every two vertices $u,v$ there is a dipath from $u$ to $v$. (This is
equivalent to saying that $G$ is 1-weak and every edge is in a dicycle.) A digraph $G$ is
\emph{strongly $k$-connected} or \emph{$k$-strong} if $G\setminus X$ is strongly connected  for every $X\subseteq V(G)$ with $|X|<k$.

Let us say a drawing of a digraph $G$ (in a plane or a 2-sphere) is \emph{diplanar} if
for every vertex $v\in V(G)$, the edges of $G$ with
head $v$ form an interval in the cyclic ordering of edges incident with $v$ determined by the drawing (and a digraph is \emph{diplanar} 
if it admits a diplanar drawing).
For instance, the digraph in Figure~\ref{fig:carter} is diplanar. (This was called ``strongly planar'' in~\cite{pfaffians}, but this 
seems a confusing name since we also need to talk about 1-strong digraphs in the sense of being strongly connected.)

We need ``ears''. Let $H$ be a 1-strong subdigraph of a 1-strong digraph $G$. If $H\ne G$, there is an \emph{ear} for $H$ in $G$, that is, 
either
\begin{itemize}
\item  a dipath of $G$ with length at least one, with both ends in $V(H)$ and with no edge or internal vertex in $H$; or 
\item a dicycle of $G$ with
exactly one vertex in $H$. 
\end{itemize}
(This is a standard, elementary result.) The point of ears is that if $P$ is an ear as above, then $H\cup P$ is also 1-strong, and 
either $H\cup P=G$ or we can choose an ear for $H\cup P$ in $G$, and so on. Thus every 1-strong subdigraph of a 1-strong digraph $G$ can be 
grown by adding ears one at a time until it becomes $G$. More exactly, let $H$ be a 1-strong subdigraph of a 1-strong digraph $G$, and
let $P_1,\dots, P_n$ be a sequence of subdigraphs of $G$ such that $H\cup P_1\cupcup P_n=G$, and for $1\le i\le n$, $P_i$ is an ear for 
$H\cup P_1\cupcup P_{i-1}$ in $G$. We call the sequence $P_1,\dots, P_n$ an \emph{ear sequence} for $H$ in $G$. Then there is an ear sequence for every 1-strong subdigraph of a 1-strong digraph. (We mention that if we were working with 1-strong, 2-weak digraphs, we could change the definition of ``ear'' to exclude the dicycle case in the second bullet, and the same result is true. This is proved in~\cite{tricycles}.)

If $C$ is a (not necessarily directed) cycle of a digraph $G$, and we select one of the two cyclic orientations of $C$, let $\mathbf{c}$ be the map from 
$E(G)$ to $\mathbb{R}$ defined by $\mathbf{c}(e) = 1$ if $e\in E(C)$ in the direction of the chosen orientation, 
$\mathbf{c}(e) = -1$ if $e\in E(C)$ is in the other direction, and  $\mathbf{c}(e) = 0$ if $e\notin E(C)$.
We call 
$\mathbf{c}$
a \emph{characteristic vector} of $C$. (Thus, $C$ has two characteristic vectors, negations of each other.)
If $C$ is a dicycle, it has a non-negative characteristic vector.
Let $C_0$ be a dicycle of a 1-strong digraph $G$, and let $P_1,\dots, P_n$ be an ear
sequence for $C_0$ in $G$. For $1\le i\le n$, there is a dicycle $C_i$ consisting of $P_i$ together with a dipath of
$C_0\cup P_1\cupcup P_{i-1}$ between the ends of $P_i$. We call the sequence of nonnegative characteristic vectors of $C_0,\dots, C_n$
an \emph{ear-basis}.
Thus, if $G$ is 1-strong with $E(G)\ne \emptyset$, then it has an 
ear-basis. It is easy to see (and is proved in~\cite{evendigraph})  that for every cycle $C$, its 
characteristic vectors are integer linear combinations of the members of any ear-basis. 
This has two consequences that we will need later:
\begin{thm}\label{cycleint}
Let $w$ be a weighting of a 1-strong digraph $G$, and let $C$ be a cycle of $G$ (not necessarily directed). If $\mathbf{c}$ denotes 
a characteristic vector of $C$, then the scalar product $w\cdot\mathbf{c}$ is an integer.
\end{thm}
\begin{proof} Let $\mathbf{c}_0,\dots, \mathbf{c}_n$ be an ear-basis. Then there are integers $\lambda_0,\dots, \lambda_n$ such that
$\sum_{0\le i\le n}\lambda_i\mathbf{c}_i=\mathbf{c}$. But since $w\cdot \mathbf{c}_i=1$ for $0\le i\le n$, it follows that 
$w\cdot \mathbf{c}$ is an integer. This proves~\ref{cycleint}.\end{proof}
\begin{thm}\label{dirsum}
Let $G$ be a weightable 1-strong digraph, and let $w:E(G)\to \mathbb{R}$ be some function. Let $\mathbf{c}_0,\dots, \mathbf{c}_n$ be an ear-basis. If $w\cdot \mathbf{c}_i=1$ for $0\le i\le n$ then $w$ is a weighting.
\end{thm}
\begin{proof} Let $w'$ be a weighting of $G$, and let $C$ be a dicycle of $G$, with characteristic vector $\mathbf{c}$. Then there are integers $\lambda_0,\dots, \lambda_n$ such that
$\sum_{0\le i\le n}\lambda_i\mathbf{c}_i=\mathbf{c}$. Since $(w-w')\cdot \mathbf{c}_i=0$ for $0\le i\le n$, it
follows that $(w-w')\cdot \mathbf{c}=0$, and so $w(C)=w'(C)=1$. This proves~\ref{dirsum}.\end{proof}

At the end of the previous section we listed five possible meanings of ``which digraphs are weightable?'' Let us prove that 
solving the second would solve the third and fourth.
We claim that: 
\begin{thm}\label{findweighting}
There are poly-time algorithms that, given a weightable digraph $G$ as input:
\begin{itemize}
\item finds a weighting of $G$, and
\item tests whether a given function on the edge set of $G$ is a weighting.
\end{itemize}
\end{thm}
\begin{proof}
In both cases, we may assume that the input digraph $G$ is 1-strong.
First, 
choose an ear-basis $\mathbf{c}_0,\dots, \mathbf{c}_n$. 

To find a weighting, just find a function $w:E(G)\to \mathbb{R}$ that satisfies
$w\cdot\mathbf{c}_i=1$ for $0\le i\le n$. (This can be done in poly-time in general by linear programming, but for an ear-basis it is 
particularly easy, since the corresponding matrix is upper triangular.) By~\ref{dirsum}, $w$ is a weighting. 

To test if a given function $w$ is a weighting, just test that $w\cdot\mathbf{c}_i=1$ for $0\le i\le n$.
If so then $w$ is a weighting by~\ref{dirsum}, and if not then it is not. This proves~\ref{findweighting}.\end{proof}

\section{More examples}

As we said, the digraph of Figure~\ref{fig:carter} disproved our original conjecture about the structure of 2-strong weightable
digraphs, so what could be the true structure? All dicycles of the drawing in Figure~\ref{fig:carter} are clockwise
(it is important that we are talking about plane drawings here, not drawings in the sphere, so that ``clockwise''
makes sense. We hoped this was a clue, because
one can show if a digraph admits a diplanar drawing in the plane in which every dicycle is
clockwise, then it is weightable.  So our next hope was the conjecture that every 2-strong, 3-weak weightable digraph admits a
diplanar drawing in which every dicycle is clockwise. But this was disproved by the digraph in Figure
\ref{fig:notclockwise}.

\begin{figure}[htb!]
\centering
\begin{tikzpicture}[auto=left]

\begin{scope}[xscale = 2, yscale=1.7, rotate = -90]
\vtx{1}{0.5,-0.5}
\vtx{2}{1.5,-0.5}
\vtx{3}{2.5,-0.5}
\vtx{4}{0,0}
\vtx{5}{1,0}
\vtx{6}{2,0}
\vtx{7}{3,0}
\vtx{8}{-0.5,0.5}
\vtx{9}{3.5,0.5}
\vtx{10}{0,1}
\vtx{11}{1,1}
\vtx{12}{2,1}
\vtx{13}{3,1}
\vtx{14}{0.5,1.5}
\vtx{15}{1.5,1.5}
\vtx{16}{2.5,1.5}

\dedge{1}{2}
\dedge{1}{5}
\dedge{2}{3}
\dedge{2}{6}
\dedge{3}{7}
\dedge[out=0, in=270]{3}{9}
\dedge{4}{1}
\dedge{4}{5}
\dedge{5}{2}
\dedge{5}{6}
\dedge{5}{11}
\dedge{6}{3}
\dedge{6}{7}
\dedge{7}{9}
\dedge{7}{13}
\dedge[out=270, in=180]{8}{1}
\dedge{8}{4}
\dedge{9}{13}
\dedge[out=90, in=0]{9}{16}
\tdedge{10}{4}
\tdedge{10}{8}
\dedge{11}{10}
\dedge{11}{14}
\tdedge{12}{6}
\dedge{12}{11}
\dedge{12}{15}
\dedge{13}{12}
\dedge{13}{16}
\tdedge[out=180, in=90]{14}{8}
\dedge{14}{10}
\dedge{15}{11}
\dedge{15}{14}
\dedge{16}{12}
\dedge{16}{15}

\vtx{a1}{0.5,-3}
\vtx{a2}{1.5,-3}
\vtx{a3}{2.5,-3}
\vtx{a4}{0,-2.5}
\vtx{a5}{1,-2.5}
\vtx{a6}{2,-2.5}
\vtx{a7}{3,-2.5}
\vtx{a8}{-0.5,-2}
\vtx{a9}{3.5,-2}
\vtx{a10}{0,-1.5}
\vtx{a11}{1,-1.5}
\vtx{a12}{2,-1.5}
\vtx{a13}{3,-1.5}
\vtx{a14}{0.5,-1}
\vtx{a15}{1.5,-1}
\vtx{a16}{2.5,-1}

\dedge{a2}{a1}
\dedge{a5}{a1}
\dedge{a3}{a2}
\dedge{a6}{a2}
\dedge{a7}{a3}
\dedge[out=270, in=0]{a9}{a3}
\dedge{a1}{a4}
\dedge{a5}{a4}
\dedge{a2}{a5}
\dedge{a6}{a5}
\dedge{a11}{a5}
\dedge{a3}{a6}
\dedge{a7}{a6}
\dedge{a9}{a7}
\dedge{a13}{a7}
\dedge[out=180, in=270]{a1}{a8}
\dedge{a4}{a8}
\dedge{a13}{a9}
\dedge[out=0, in=90]{a16}{a9}
\tdedge{a4}{a10}
\tdedge{a8}{a10}
\dedge{a10}{a11}
\dedge{a14}{a11}
\tdedge{a6}{a12}
\dedge{a11}{a12}
\dedge{a15}{a12}
\dedge{a12}{a13}
\dedge{a16}{a13}
\tdedge[out=90, in=180]{a8}{a14}
\dedge{a10}{a14}
\dedge{a11}{a15}
\dedge{a14}{a15}
\dedge{a12}{a16}
\dedge{a15}{a16}

\dedge{a14}{1}
\dedge{3}{a16}
\tdedge[out=135, in=225]{a8}{8}
\dedge[out=315, in=45]{9}{a9}
\end{scope}
\end{tikzpicture}
\caption{A diplanar, 2-strong, weightable digraph, that cannot be drawn in the plane such that all its dicycles are 
clockwise. The thick edges have weight one.} \label{fig:notclockwise}
\end{figure}
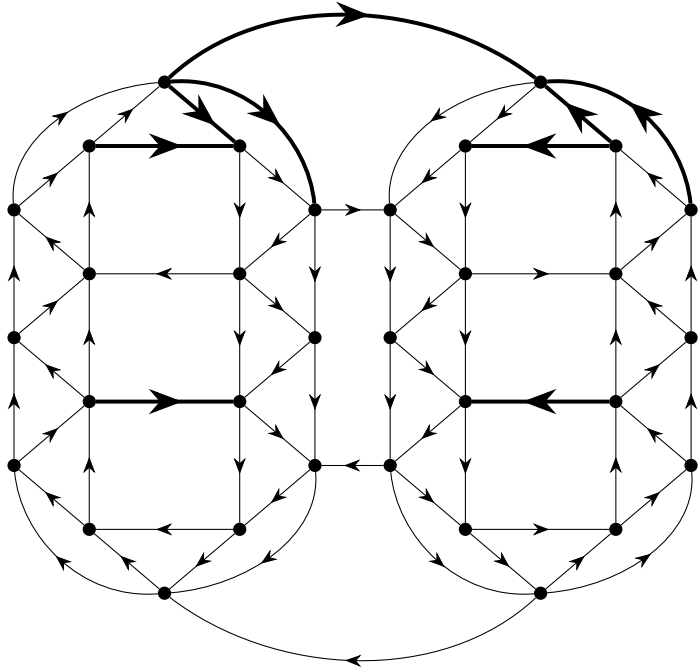

So both our conjectures were false, and we fell back to the question: is it at least true that every 2-strong, 3-weak 
weightable digraph is planar? But that too is false, disproved by the digraph in Figure~\ref{fig:nonplanar}. 

\begin{figure}[htb!]
\centering
\begin{tikzpicture}[auto=left]

\begin{scope}[scale = .7]
\vtx{a1}{-2,-2};
\vtx{a2}{2,2};
\vtx{a3}{-2,2};
\vtx{a4}{2,-2};
\vtx{a5}{-5,0};
\vtx{a6}{5,0};
\vtx{a7}{0,-1};
\vtx{a8}{0,1};

\foreach \from/\to in {a1/a7,a2/a8,a3/a1,a3/a5, a4/a2,a4/a6,a5/a1,a6/a2,a7/a4}
\dedge{\from}{\to};
\tdedge{a8}{a3};
\tdedge[in=90, out=120]{a2}{a5};
\tdedge[out=90, in=60]{a6}{a3};
\dedge[out=300, in =270]{a1}{a6};
\dedge[out=-90, in= -120]{a5}{a4};
\tdedge[out = 240, in = 120]{a8}{a7};
\dedge[out =60, in = 300]{a7}{a8};
\end{scope}

\end{tikzpicture}
\caption{A nonplanar, 2-strong, weightable digraph. Again, the thick edges have weight one.} \label{fig:nonplanar}
\end{figure}
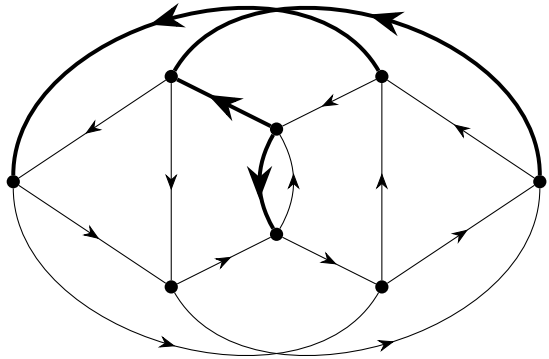

\section{Known results}

Let us observe first that, 
to understand the weightable digraphs $G$, we may assume that $G$ is 1-weak (since weightings for different weak components may be considered separately), and every edge is in a dicycle (since edges not in a dicycle may be given any weight and deleted). It follows that $G$ is 1-strong.
We will often assume that the digraphs we are considering are 1-strong, without further explanation.

There are some results about weightable digraphs that were already known. 
It was shown in~\cite{tricycles} that:
\begin{thm}\label{get01}
If $G$ admits a weighting then it admits a 0/1-valued weighting.
\end{thm}
\begin{proof} 
We may assume that $G$ is 1-strong. 
For every set $X$
of vertices, since $w$ is a weighting, so is $w+c_X$, where $c_X(e)=1$ for edges $e$ from $X$ to $V(G)\setminus X$, 
$c_X(e)=-1$ for edges from $V(G)\setminus X$ to $X$, and $c_X(e)=0$ otherwise. By adding multiples of the functions $c_X$ to $w$ for appropriate choices of 
$X$, we can obtain a weighting which is zero on every edge of some spanning tree $T$.
For $e\in E(G)\setminus E(T)$, let $C$ be the cycle that contains $e$ and is otherwise in $T$, and let $\mathbf{c}$ be its characteristic 
vector. By~\ref{cycleint}, $w\cdot\mathbf{c}$ is an integer and so $w(e)$ is an integer.
Consequently this weighting is integer-valued on every edge of $G$.

Having obtained an integer weighting, now let us choose an integer weighting $w$ that minimizes the sum of $|w(e)|$ over all edges $e$ with $w(e)<0$. Suppose there
is an edge $e=uv$ with $w(e)<0$. Let $X$ be the set of all vertices $x$ such that there is
a dipath $P$ of $G$ from $v$ to $x$ where $w(e)\le 0$ for all edges $e\in E(P)$. Then $w-c_X$ is a better choice than $w$,
a contradiction. So $w(e)\ge 0$ for all edges $e$, and the result follows.\end{proof}

Henceforth in the paper we will only work with 0/1-valued weightings.
Let $k\ge 3$. A \emph{weak $k$-double-cycle} is a digraph formed by the union of $k$ dicycles $C_1,\dots, C_k$, where
each vertex belongs to at most two of $C_1,\dots, C_k$, and $C_i\cap C_{i+1}$ is a dipath for $1\le i\le k$ (reading subscripts
modulo $k$), and $C_i$ is vertex-disjoint from $C_j$
if $j\not\equiv i-1,i,i+1\bmod k$.  (In some earlier papers $k=2$ is permitted, but we do not need that here.) An example is shown in Figure~\ref{fig:doublecycle}.

\begin{figure}[htb!]
    \centering
    \begin{tikzpicture}
        \vtx{1}{0,3};
        \vtx{2}{1,2};
        \vtx{3}{3,3};
        \vtx{4}{2,2};
        \vtx{5}{3,0};
        \vtx{6}{2,1};
        \vtx{7}{0.5,0.5};
        \vtx{8}{1.5,3};
        \vtx{9}{2.5,2.5};
        \foreach \from/\to in {1/2,2/4,4/9,9/3,3/8,8/1,3/5,5/6,6/4,6/7,7/5,7/1,2/7}
        \dedge{\from}{\to};
    \end{tikzpicture}
    \caption{An example of a weak 4-double cycle.}
    \label{fig:doublecycle}
\end{figure}
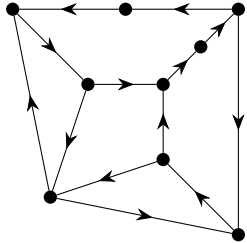

It was shown in~\cite{tricycles} that:
\begin{thm}\label{doublecycle}
A digraph $G$ is weightable if and only if for all $k\ge 3$, no subdigraph is a weak $k$-double-cycle.
\end{thm}
A digraph $G$ is \emph{odd-weightable} if there is a function $w:E(G)\to \{0,1\}$ such that $w(C)$ is odd for every dicycle of $G$.
Thus, by~\ref{get01}, every weightable 
digraph is odd-weightable.
The proof given in~\cite{tricycles} for~\ref{doublecycle} was obtained from a similar proof in~\cite{evendigraph}, where it was shown that:
\begin{thm}\label{eventhm}
A digraph $G$ is odd-weightable
if and only if for all odd $k\ge 3$, no subdigraph
is a weak $k$-double-cycle.
\end{thm}

There is another set of older results we need, not really about weightable digraphs, but relevant. Let $H$ be a graph with a bipartition $(A,B)$, and let $M$ be a perfect matching of 
$H$; we call the pair
$(H,M)$ a \emph{bisource}. Direct 
all the edges of $H$ from $A$ to $B$, except for the edges in $M$, and then contract all the edges in $M$. This produces some 
digraph, called a \emph{collapse} of $(H,M)$. (It also depends on the choice of the bipartition $(A,B)$, so if $H$ is connected, 
$(H,M)$ has two collapses, one obtained from the other by reversing all edges.) 
Conversely, every digraph is a collapse of some (unique) bisource. There is 
a remarkable theorem:
\begin{thm}\label{invariant}
Let $H$ be a bipartite graph, let $M, M'$ be perfect matchings of $H$, and let $G,G'$ be corresponding collapses. Then $G$ is 
odd-weightable if and only if $G'$ is odd-weightable.
\end{thm}
This is proved for ``two-extendible'' bipartite graphs in~\cite{pfaffians}, although it seems to be implicit in earlier papers. 
We do not prove it here because the result is just for background. But the moral is that to  
understand odd-weightable digraphs, it is better to understand the bipartite graphs of the corresponding bisouces, because the choice of 
perfect matching is irrelevant.
It would have been nice if the same simplification held for the property of being weightable, but it is not;
we shall see that whether a digraph is weightable or not depends on both terms of its bisource.

A graph is \emph{$k$-extendible} if every matching of size at most $k$ can be extended to a perfect matching; and 
a \emph{brace} is a connected 2-extendible bipartite graph.
A bipartite graph $H$ with a perfect matching is a brace if and only if the collapse of $(H,M)$ is 2-strong for some (or equivalently,
every) choice of a perfect matching $M$.
Let us say a bipartite graph $H$ is \emph{Pfaffian} if there is a perfect matching $M$ of $H$ such that the collapse of
$(H,M)$ is odd-weightable. 
(These are precisely the bipartite graphs that admit ``Pfaffian orientations'',
a topic of interest in theoretical physics and other areas that we do not define here.) In~\cite{pfaffians}, Robertson, Seymour and Thomas
gave a construction for all Pfaffian bipartite graphs. Essentially, the problem can be reduced to constructing the
Pfaffian braces; and they showed that a brace $H$ is Pfaffian if and only if either $H$ is planar, or $H$ is the Heawood graph,
or $H$ admits a decomposition into three smaller Pfaffian braces that we will discuss in more detail later. And reversing this decomposition gives a way to piece together three Pfaffian braces to make a larger Pfaffian brace. 

For our problem, we can reduce it to studying the 2-strong weightable digraphs, and such digraphs $G$ are collapses of
bisourses $(H,M)$ for which $H$ is a brace. Since, as we saw, every such $G$ is odd-weightable, and therefore $H$ is Pfaffian, we can apply
the decomposition theorem of~\cite{pfaffians} to our problem, and deduce that either $H$ is planar or admits some useful
decomposition into three parts, and therefore the same applies to $G$. The problem is that the corresponding composition operation of gluing three Pfaffian braces together to make one larger Pfaffian brace does not preserve the property that the collapse is weightable,
so this by itself does not reduce our problem to the planar case, and we will need to look carefully at the decomposition given 
by~\cite{pfaffians}. 
To illustrate: the graph Rotunda, shown in Figure~\ref{fig:rotunda}, was fundamental in the result of~\cite{pfaffians}. It 
has only three perfect matchings that are not equivalent to one 
another under symmetries of the graph, and hence it only gives rise to three nonisomorphic collapses. All three are odd-weightable;
but one is nonplanar and weightable (the digraph in Figure~\ref{fig:nonplanar}), one is nonplanar and not weightable, and one
is planar and not weightable.

Nevertheless, by refining the Pfaffian brace decomposition theorem, we will able to reduce our problem to the planar case.
And we can get a little more from it. When $G$ is the collapse of $(H,M)$, if $H$ is planar, 
then $G$ is not only planar but diplanar; and if $H$ is not planar, then it is so far from planar that if $G$ is weightable then it 
is also nonplanar.
This will imply the convenient fact that if $G$ is 2-strong, 3-weak, planar and weightable then it is diplanar. 

\section{Basic lemmas}

If $C$ is a dicycle of a digraph $G$, and $u,v,w\in V(C)$ are distinct,
then $C$ passes through these three vertices in some order, one of the two cyclic orders of the triple $\{u,v,w\}$. Let us say 
the ordered triple $(u,v,w)$ is \emph{in order} in $C$ if the dipath of $C$ from $u$ to $v$ does not pass through $w$. 
Thus $(u,v,w)$ is in order in $C$ if and only if $(v,w,u)$ is in order in $C$. A triple $\{u,v,w\}$ of three distinct vertices is \emph{bad}
in a digraph $G$ if  there exist 
dicycles $C,C'$ of $G$, both containing $u,v,w$, such that $(u,v,w)$ is in order in $C$ and $(w, v,u)$ is in order in $C'$.
We say such cycles $C,C'$ \emph{disagree} on $\{u,v,w\}$.

Here is a result that we will use very frequently:
\begin{thm}\label{triple}
Let $G$ be a digraph; then $G$ is weightable if and only if there is no bad triple.
\end{thm}
\begin{proof}
First, we assume that $G$ is weightable, and suppose that $\{u,v,w\}$ is a bad triple. By~\ref{get01} we may choose a 0/1-valued weighting $w$. 
Let $F$ be the set of edges $e$ 
with $w(e) = 1$. So every dicycle has exactly one edge in $F$. Now choose dicycles $C,C'$ of $G$,
such that $(u,v,w)$ is in order in $C$ and $(w,v,u)$ is in order in $C'$. Let $C(uv)$
be the subpath of $C$ from $u$ to $v$, and define $C(vw), C(wu), C'(uw), C'(wv), C'(vu)$ similarly. Since $|E(C)\cap F|=1$
we may assume that $C(vw), C(wu)$ both have no edges in $F$, and similarly two of $C'(uw), C'(wv), C'(vu)$ have no edges in $F$.
But $C(vw)\cup C'(wv)$ includes a dicycle, which has an edge in $F$, and so $F\cap C'(wv)\ne \emptyset$; and similarly
$C(wu)\cup C'(uw)$ includes a dicycle and hence $C'(uw)$ has an edge in $F$, a contradiction. 

For the converse, now we assume that $G$ is not weightable, and therefore includes a weak $k$-double-cycle for some $k\ge 3$;
let $C_1,\dots, C_k$ be as in the definition of weak $k$-double-cycle. Choose $v_1\in V(C_k\cap C_1)$, and $v_2\in V(C_1 \cap C_2)$, and
$v_3\in V(C_2\cap C_3)$. Then there is a dipath $P_1$ of $C_1$ from $v_1$ to $v_2$, and a dipath $P_2$ of $C_1$ from 
$v_2$ to $v_3$, and a dipath $P_3$ of $C_3\cupcup C_k$ from $v_3$ to $v_1$, and the union of these three paths is a dicycle in which $(v_1,v_2,v_3)$ is in order. But similarly, there is a dipath $Q_1$ of $C_1$ from $v_2$ to $v_1$, and 
 a dipath $Q_2$ of $C_2$ from $v_3$ to $v_2$, and a dipath $Q_3$ of $C_3\cupcup C_k$ from $v_1$ to $v_3$,
giving a dicycle $C'$ in which $(v_3,v_2,v_1)$ is in order. So $\{v_1,v_2,v_3\}$ is a bad triple. 
This proves~\ref{triple}.\end{proof}

We need a lemma which allows us to to convert a 0/1-valued weighting to one that is more convenient:
\begin{thm}\label{pickroot}
Let $G$ be a weightable digraph and let $u\in V(G)$. Then there is a 0/1-valued weighting $w$ of $G$ such that $w(e) = 1$
for every edge $e$ with head $u$ and $w(e) = 0$ for every edge with tail $u$.
\end{thm}
\begin{proof} We may assume that every edge of $G$ is in a dicycle (since edges not in a cycle can be given any weight we want). We can also assume that $G$ is 1-weak, and hence 1-strong.
By~\ref{get01}, there is a 0/1-valued weighting $w$. Let $X_w$ be the set of all vertices $v$ such that there is a dipath $P$
of $G$ from $u$ to $v$ where $w(e) = 0$ for each edge $e\in E(P)$, and choose $w$ with $X_w$ maximal. 
Let $D^+$ be the set of edges $ab$ with $a\in X$ and $b\notin X$, and let $D^-$ be the set of edges $ab$ with $a\notin X$ and $b\in X$.
From the definition of $X_w$, it follows that $w(e) = 1$ for each $e\in D^+$. Moreover, we claim that
$w(e) = 0$ for each edge $e\in D^-$. To see this, observe that $e$ is in a dicycle $C$, because $G$ is 1-strong, and so $C$ contains an edge in $D^+$; and since $w(C)=1$, it follows that $w(e) = 0$, as we claimed.

Define $w'$ by:
\[
w'(e) = \begin{cases}
    w(e)-1 & \text{if $e\in D^+$,} \\
    w(e)+1 & \text{if $e\in D^-$,} \\
    w(e) & \text{otherwise}.
\end{cases}
\]
Then $w'$ is a 
0/1-valued weighting. But $X_w\subseteq X_{w'}$, and so $X_{w'}=X_w$ from the choice of $w$. Consequently $D^+=\emptyset$, and 
therefore $X_w=V(G)$ since $G$ is 1-strong. We deduce that $w(f) = 1$ for every $f=vu$ with head $u$ since there is a dipath $P$
from $u$ to $v$ with $w(e) = 0$ for each edge $e\in E(P)$, and adding $f$ to $P$ makes a dicycle. Moreover, for every edge $f$ 
with tail $u$, since $f$ belongs to a dicycle that contains an edge with head $u$, it follows that $w(f) = 0$. This proves 
\ref{pickroot}.\end{proof}

Similarly, we can obtain a 0/1-valued weighting such that $w(e) = 0$
for every edge $e$ with head $u$, and $w(e) = 1$ for every edge with tail $u$. It is easy to convert the proof above
to a poly-time algorithm that, given 
as input a 1-strong digraph, a 0/1-valued weighting, and a vertex $u$, outputs a weighting as in~\ref{pickroot}.

If $e$ is an edge of a digraph $G$, we denote by $G/e$ the digraph obtained by contracting $e$.
If $G$ is a digraph, we say an edge $e=uv$ is a \emph{singular edge} of $G$ if $u\ne v$ and either no edge different from $e$ has head $v$ or no
edge different from $e$ has tail $u$, and the operation of contracting this edge is called \emph{singular contraction}. 
If $H$ can be obtained from a subdigraph $G'$ of a digraph $G$ by repeated singular contraction, $H$ is said to be a \emph{butterfly minor} of $G$.
\begin{thm}\label{butterfly}
If $e$ is a singular edge of $G$, then $G$ is weightable if and only if $G/e$ is weightable. Consequently, 
if $G$ is weightable and $H$ is a butterfly minor of $G$, then $H$ is weightable. 
\end{thm}
\begin{proof}
Let $e=uv$ be a singular edge of $G$. We may assume that no edge different from $e$ has tail $u$ (the other case is the same, by reversing the direction of all edges). Suppose that $G$ is weightable. By~\ref{pickroot}, there is a 0/1-valued weighting $w$
of $G$ with $w(e) = 0$. But then the restriction of $w$ to $E(G')$ is a weighting of $G/e$, since for every dicycle $C'$ of $G'$,
either $C'$ is a dicycle of $G$, or there is a dicycle $C$ containing $e$ with $C/e = C'$.

Conversely, suppose that $w'$ is a weighting of $G/e$, and define $w(e) = 0$ and $w(f) =w'(f)$ for each edge $f\ne e$ of $G$. 
For every dicycle $C$ of $G$, either $C$ is a dicycle of $G/e$, or $e\in E(C)$ and $C/e$ is a dicycle of 
$G/e$, and it follows that $w$ is a weighting of $G$. This proves~\ref{butterfly}.\end{proof}

\section{Some easy reductions}\label{sec:easy}
Let us say a digraph is \emph{simple} if it has no loops or parallel edges (it might have directed cycles of length two).
As we said before, to understand the weightable digraphs, it suffices to understand those that are 1-strong, and we can also assume that $G$ is simple.
It is slightly less clear that we may assume that $G$ is 2-strong, so let us prove that. Suppose that $G$
is a 1-strong digraph, and $c\in V(G)$, and there is a partition $(A,B)$ of $V(G)\setminus \{c\}$ into two nonempty sets $A,B$, such that 
no edge of $G$ is from $B$ to $A$. Let $G_1$ be the digraph made from $G[A\cup \{c\}]$ by adding an edge from $a\in A$ to $c$ for each $ab\in E(G)$ with $b\in B$,
and similarly let $G_2$ be obtained from $G[B\cup \{c\}]$ by adding an edge from $c$ to $b\in B$ for each such $ab\in E(G)$. Then remove any parallel edges in $G_1$ and $G_2$.

See Figure~\ref{fig:1reduce}. There will be similar drawings for other constructions, so let us explain our conventions for 
drawing constructions/decompositions. Ovals indicate sets of vertices. Thick arrows indicate (possibly empty) \emph{sets} of edges, all going in the direction of the arrow, and squiggly lines indicate sets of edges that may go in either direction. When the position of one end of an arrow or squiggly line matches in the drawing of $G$ and $G_i$ for some $i\in\{1,2\}$, it indicates that those edges of $G$ are to be ``rerouted'' in $G_i$ in the way depicted; for instance, in Figure~\ref{fig:1reduce}, all edges from $A$ to $B$ in $G$ get rerouted so they now go from $A$ to $c$ in $G_1$ and from $c$ to $B$ in $G_2$. There are a few further conventions that we explain later.

\begin{figure}[htb!]
    \centering
    \begin{tikzpicture}[scale=1.3]
        \vset{A}{$A$}{0,2};
        \vset{B}{$B$}{0,0};

        \hdedge{A.315}{B.45};

        \vtx[left:$c$]{c}{-0.5, 1};
        \sline{c}{A.225};
        \sline{c}{B.135};

        \node (G) at (0,-0.75) {$G$};

        \draw[<->, thick] (1.5,1) -- (2.5,1) node[midway, above] {\ref{1reduce}};

        \vset{A}{$A$}{4,2.5};
        
        \vtx[below:$c$]{c}{4, 1.5};
        \sline[out=180, in=270]{c}{A.225};
        \hdedge[out=270, in=0]{A.315}{c};
        \node (G1) at (4, 3.25) {$G_1$};

        \vset{B}{$B$}{4,-0.5};
        
        \vtx[above:$c$]{c}{4, 0.5};
        \sline[out=180, in=90]{c}{B.135};
        \hdedge[out=0, in=90]{c}{B.45};
        \node (G2) at (4, -1.25) {$G_2$};
    \end{tikzpicture}
    \caption{Construction for building non-2-strong weightable digraphs.}
    \label{fig:1reduce}
\end{figure}
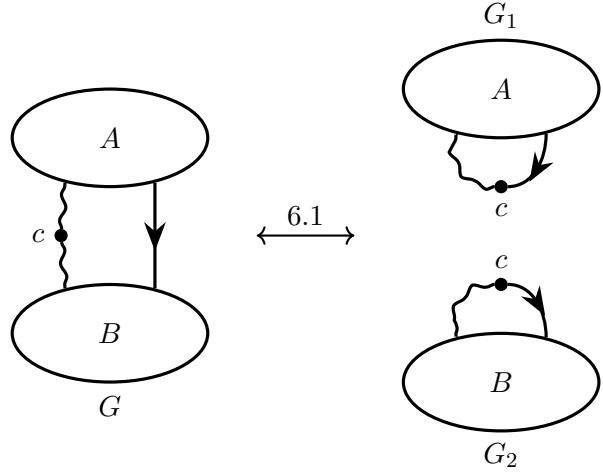

\begin{thm}\label{1reduce}
If $G, c, A,B,G_1,G_2$ are as above (or in Figure~\ref{fig:1reduce}), then $G$ is weightable if and only if $G_1,G_2$ are both weightable. 
\end{thm}
\begin{proof}
Since $G$ is 1-strong, it follows that for each $a\in A$ there is a dipath from $c$ to $a$, and this path is a subpath of $G[A\cup \{w\}]$ since there is no edge from $B$ to $A$. Consequently $G_1$ is 1-strong, and similarly so is $G_2$. 

Suppose first that $w_i$ is a 0/1-valued weighting of $G_i$ for $i = 1,2$. By~\ref{pickroot}, we may assume that $w_1(e) = 1$
for every edge of $G_1$ with head $c$, and $w_2(e) = 1$ for every edge of $G_2$ with tail $c$. For each $e\in E(G)$, define
$w(e)$ by: $w(e) = w_i(e)$ if $e\in E(G_i)$ for $i = 1,2$, and $w(e) = 1$ if $e$ is from $A$ to $B$. We claim $w$ is a weighting 
of $G$. To see this, it suffices to check that $w(C) = 1$ for every dicycle $C$ of $G$ that is not a subdigraph of either of 
$G[A\cup \{c\}], G[B\cup \{c\}]$. Such a cycle $C$ contains an edge $ab$ from $A$ to $B$, and since there is no edge from 
$B$ to $A$, it follows that $c\in C$, and $C$ consists of the union of a dipath $P_1$ of $G[A\cup \{c\}]$ from $c$ to $a$, the edge
$ab$, and a dipath $P_2$ of $G[B\cup \{c\}]$ from $b$ to $c$. But adding the edge $bc$ to $P_1$ makes a dicycle of $G_1$,
and since $w_1(bc)=1$, we deduce that $w_1(P_1) = 0$. Similarly $w_2(P_2) = 0$, and so $w(C)=1$ as desired. Thus $w$ is a weighting
of $G$.

Conversely, suppose that one of $G_1,G_2$, say $G_1$, is not weightable, and let 
dicycles $C,C'$ of $G_1$ disagree on some bad triple $\{u,v,w\}$ of $G_1$.
Define
a cycle $D$ of $G$ as follows. If every edge of $C$ belongs to $G$ let $D=C$. Otherwise, exactly one edge of $C$ is not an edge of 
$G$, and any such edge is from some $a\in A$ to $c$ where there is an edge $ab\in E(G)$ for some $b\in B$. Choose a dipath $P$ of $G[B]$ from $b$ to $c$, and let $D$ be the dicycle of $G$
made by the union of the dipath of $C$ from $c$ to $a$, the edge $ab$, and $P$. Define $D'$ similarly, starting from $C'$.
Then $(u,v,w)$ is in order in $D$ and $(w,v,u)$ is in order in $D'$, and so $G$ is not weightable. This proves~\ref{1reduce}.\end{proof}

In view of~\ref{1reduce}, it suffices to understand the 2-strong weightable digraphs. 
Let us see also that we can assume $G$ is 3-weak. Suppose then that $G$ is 2-strong and $|V(G)|\ge 4$, and so $G$
is 2-weak, and assume $G$ is not 3-weak. Choose distinct
$c,d\in V(G)$ and a partition $(A,B)$ of $V(G)\setminus \{c,d\}$ such that there is no edge of $G$ between $A,B$ in either direction.
Let $G_1$ be obtained from $G[A\cup \{c,d\}$ by adding an edge $cd$ and an edge $dc$ if they are not already present, and define $G_2$
similarly from $G[B\cup \{c,d\}]$. (See Figure~\ref{fig:2weakreduce}. In this figure and future (de)compositions, thin arrows 
between vertices represent single edges. Additionally, there may be edges between vertices outside of the labeled vertex sets that 
are not shown in the figure.  For example, there may be an edge from $c$ to $d$ in $G$.)

\begin{figure}[htb!]
    \centering
    \begin{tikzpicture}[scale=1.3]
        \vset{A}{$A$}{0,2};
        \vset{B}{$B$}{0,0};

        \vtx[left:$c$]{c}{-0.5, 1};
        \vtx[right:$d$]{d}{0.5, 1};
        \sline{c}{A.225};
        \sline{c}{B.135};
        \sline{d}{A.315};
        \sline{d}{B.45};

        \node (G) at (0,-0.75) {$G$};

        \draw[<->, thick] (1.5,1) -- (2.5,1) node[midway, above] {\ref{2weakreduce}};

        \vset{A}{$A$}{4,2.5};
        
        \vtx[left:$c$]{c}{3.5, 1.5};
        \vtx[right:$d$]{d}{4.5, 1.5};
        \sline{c}{A.225};
        \sline{d}{A.315};
        \dedge[bend left=45]{c}{d};
        \dedge[bend left=45]{d}{c};
        \node (G1) at (4, 3.25) {$G_1$};

        \vset{B}{$B$}{4,-0.5};
        
        \vtx[left:$c$]{c}{3.5, 0.5};
        \vtx[right:$d$]{d}{4.5, 0.5};
        \sline{c}{B.135};
        \sline{d}{B.45};
        \dedge[bend left=45]{c}{d};
        \dedge[bend left=45]{d}{c};
        \node (G2) at (4, -1.25) {$G_2$};
    \end{tikzpicture}
    \caption{Construction for building non-3-weak weightable digraphs.}
    \label{fig:2weakreduce}
\end{figure}
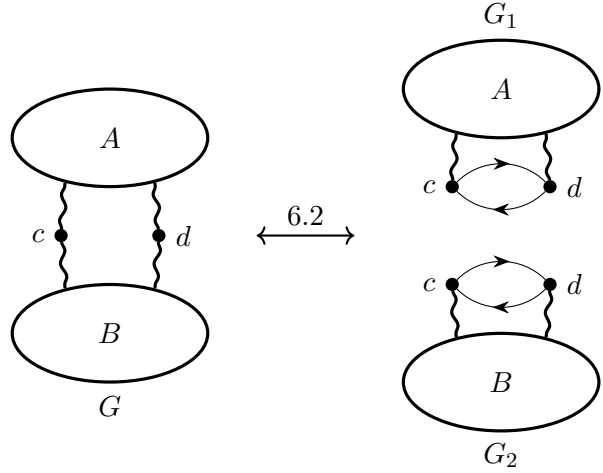

\begin{thm}\label{2weakreduce}
With $G, c,d,A,B,G_1,G_2$ as above (or in Figure~\ref{fig:2weakreduce}), then $G$ is weightable if and only if $G_1,G_2$ are weightable. 
\end{thm}
\begin{proof}
Suppose first that $w_i$ is a 0/1-valued weighting of $G_i$ for $i = 1,2$. By~\ref{pickroot}, we may assume that $w_1(cd) = w_2(cd)$.
Since $w_i(cd) = 1-w_i(dc)$ for $i = 1,2$, it follows that $w_1(dc) = w_2(dc)$. 
For each edge $e\in E(G)$, define $w(e) = w_i(e)$ where $e\in E(G_i)$. If $C$ is a dicycle of $G$ that is not a cycle of either
of $G_1,G_2$, then, by exchanging $c,d$ if necessary, we may assume that 
$C$ consists of a dipath $P_1$ of $G_1$ from $c$ to $d$ and a dipath $P_2$ of $G_2$ from $d$ to $c$. Since adding $dc$
to $P_1$ makes a dicycle of $G_1$, it follows that $w(P_1) = w_1(P_1) = 1-w_1(dc)$, and similarly $w(P_2) = 1-w_2(cd)$.
Since $w_1(dc)+w_2(cd)=w_1(dc)+w_1(cd) = 1$, it follows that $w(C)=1$, as desired. Thus $w$ is a weighting
of $G$.

Conversely, suppose that one of $G_1,G_2$, say $G_1$, is not weightable, and let $C,C'$ disagree on $\{u,v,w\}$ in $G_1$. 
We claim that  there is a 
dipath $P_2$ of $G[B\cup \{c,d\}]$ from $c$ to $d$. To see this, let $b\in B$. Since $G$ is 2-strong, there is a dipath 
from $c$ to $b$ in $G\setminus \{d\}$, and this is therefore a path of $G[B\cup \{c,d\}]$. Similarly there is a dipath
of $G[B\cup \{c,d\}]$ from $b$ to $d$, and the union of these paths includes a dipath of $G[B\cup \{c,d\}]$ from $c$ to $d$,
as claimed. Let $P(cd)$ be such a path, and similarly let $P(dc)$ be a dipath of $G[B\cup \{c,d\}]$ from $d$ to $c$.
If $C$ contains one of $cd,dc$, say $cd$, let $D$ be the cycle consisting of the path of $C$ from $d$ to $c$ together with $P(cd)$,
and define $D$ similarly if $dc\in E(C)$. (Not both $cd,dc\in E(C)$, so this is well-defined.) If $cd,dc\notin E(C)$ let $D=C$. 
Define $D'$ similarly, starting from $C'$. Then $D,D'$ are dicycles of $G$ that disagree on $\{u,v,w\}$, and so $G$ is not 
weightable. This proves~\ref{2weakreduce}.\end{proof}

In view of~\ref{2weakreduce}, it suffices to understand the simple, 2-strong, 3-weak, weightable digraphs. More exactly, we have 
shown so far that: 
\begin{itemize}
\item every weightable digraph $G$ can be built from  simple, 2-strong, 3-weak, weightable digraphs by operations
that preserve being weightable; and
\item if we have a poly-time algorithm to decide whether any 2-strong, 3-weak digraph is weightable, then in poly-time 
we can decide 
whether a general digraph is weightable, and if so, find a weighting of it.
\end{itemize}
\section{The planar decomposition}

Let $G$ be a digraph drawn in a plane or 2-sphere; then each edge $e$ is an open line segment, and we speak of \emph{points} of 
$e$ to refer to points in this line segment.  
Let $F$ be a simple closed curve $F$,
such that $F$ passes through no vertex of $G$, and passes through at most one point of the interior of each edge, and crosses each edge
that it intersects. Let us call such a curve $F$ a \emph{cut-curve}. Let us say a \emph{gap} of $F$ is a line segment in $F$ with both ends 
in the drawing and no internal point in the drawing. Thus, its ends necessarily belong to the interiors of distinct edges.
A \emph{change} in $F$ is a gap with ends in two edges $e,f$, such that exactly one of $e,f$ has head inside the disc bounded by $F$ 
(that is, $e,f$ cross $F$ in opposite directions). It follows that there is an even number of changes, and we call this number 
the \emph{change number} of $F$. We are interested in cut-curves with change number two; they will give the construction we need to 
build all weightable diplanar digraphs. 

Some terminology: if $G$ is a digraph drawn in a 2-sphere, and $F$ is a cut-curve, and $A$ is the 
set of vertices drawn within one of the two discs defined by $F$, we want to consider the digraph and drawing obtained by \emph{squishing} $A$, that is,
deleting all edges with both 
ends in $A$ and then identifying all the vertices in $A$ into one vertex, forming a digraph $G_1$. (For the moment, $G[A]$ might not be 
1-weak, so this is not the same as contracting the edges of $G[A]$.) This operation might introduce parallel edges, but not loops.
Thus, the edges of $G$ after squishing are the edges of $G$ before squishing that have at least 
one end not in $A$, but the incidence
relation between edges and vertices has changed. In particular, if in $G$ an edge has head in $A$ and tail in $V(G)\setminus A$, then 
in $G_1$, its head is the new vertex $a$ and its tail is the same as before. So, for clarity, we speak of the \emph{$G$-head} or 
\emph{$G_1$-head} of edges, and similarly speak of \emph{$G$-tail} and \emph{$G_1$-tail}. 

We observe, first:
\begin{thm}\label{buildplanar}
Let $G$ be a digraph drawn in a 2-sphere, and let $F$ be a cut-curve with change number two. Let $A$ be the set of vertices of $G$
inside one of the discs bounded by $F$, and let $B=V(G)\setminus A$. Let $G_1$ be obtained from $G$ by squishing $B$ into a vertex $b$,
and define $G_2, a$ similarly. (See Figure~\ref{fig:buildplanar}.) If $G_1, G_2$ are weightable and 1-strong then $G$ is weightable.
\end{thm}

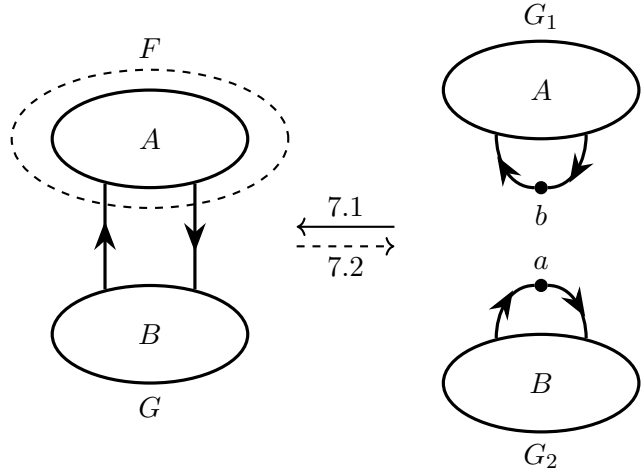
\begin{figure}[htb!]
    \centering
    \begin{tikzpicture}[scale=1.3]
        \node[thick, dashed, draw, ellipse, inner sep=0, fit={+(-1,-0.5) +(1,0.5)}, label={above:$F$}] (F) at (0,2) {};
    
        \vset{A}{$A$}{0,2};
        \vset{B}{$B$}{0,0};

        \hdedge{B.135}{A.225};
        \hdedge{A.315}{B.45};

        \node (G) at (0,-0.75) {$G$};

        \draw[<-, thick] (1.5,1.1) -- (2.5,1.1) node[midway, above] {\ref{buildplanar}};
        \draw[->, thick, dashed] (1.5,0.9) -- (2.5,0.9) node[midway, below] {\ref{goodcut}};

        \vset{A}{$A$}{4,2.5};
        
        \vtx[below:$b$]{b}{4, 1.5};
        \hdedge[out=180,in=270]{b}{A.225};
        \hdedge[out=270,in=0]{A.315}{b};
        \node (G1) at (4, 3.25) {$G_1$};

        \vset{B}{$B$}{4,-0.5};
        
        \vtx[above:$a$]{a}{4, 0.5};
        \hdedge[out=90,in=180]{B.135}{a};
        \hdedge[out=0,in=90]{a}{B.45};
        \node (G2) at (4, -1.25) {$G_2$};
    \end{tikzpicture}
    \caption{Construction for building planar weightable digraphs. Here, the pictures represent drawings of the relevant graphs in the plane or 2-sphere. It is important that a closed curve $F$ separating $A$ and $B$ has change number two, as depicted. Also note that this construction admits only a partial converse, described in~\ref{goodcut}.}
    \label{fig:buildplanar}
\end{figure}

\begin{proof} Suppose $G_1,G_2$ are weightable and 1-strong. By~\ref{pickroot}, there is a 0/1-valued weighting $w_1$ of $G_1$
such that $w_1(e) = 1$ for every edge of $G_1$ with head $b$ and $w_1(e)=0$ for every edge with tail $b$. Similarly,
there is a 0/1-valued weighting $w_2$ of $G_2$
such that $w_2(e) = 1$ for every edge of $G_2$ with tail $a$ and $w_1(e)=0$ for every edge with head $a$. 
For each edge $e\in E(G)$, choose $i\in \{1,2\}$ with 
$e\in E(G_i)$ and let $w(e) = w_i(e)$ (if $e$ belongs to both $G_1,G_2$ then it crosses $F$ and $w_1(e) =w_2(e)$, so this is well-defined). We claim that $w$ is a weighting of $G$. Let $C$ be a dicycle of $G$, we may assume that $C$ is not a cycle of $G_1$ or 
of $G_2$, so $C$ crosses $F$ at least twice. But if we enumerate the edges of $C$ that cross $F$ in their cyclic order in $F$,
then every consecutive pair cross $F$ in opposite directions, and since $F$ has change number two, it follows that $C$ crosses $F$
exactly twice. Hence there are two edges $e,f$ of $C$ that cross $F$ such that the component $P$ of $C\setminus \{e,f\}$
from the head of $e$ to the tail of $f$ is a path of $G_1$ and the component $Q$ of $C\setminus \{e,f\}$
from the head of $f$ to the tail of $e$ is a path of $G_2$. The edges of $P$, together with $e,f$, make a dicycle of $G_1$,
and since $w_1(e) = 0$ and $w_1(f) = 1$, if follows that $w_1(P) = w(P) = 0$. Similarly $w(Q) = 0$, and so 
\[ w(C)=w(P)+w(Q)+w(e)+w(f) = 1 \]
as required. This proves~\ref{buildplanar}.\end{proof}

Thus, to give a construction for all planar 1-strong weightable digraphs, it would suffice to show that each such digraph (except some small
ones that we would consider ``building blocks'') admits a cut-curve with change number two such that the digraphs $G_1,G_2$
as in~\ref{buildplanar} are weightable and smaller than $G$.  But this needs some care. It is not enough to find a cut-curve with 
change number two such that $|A|,|B|\ge 2$ (where $A,B$ are as in~\ref{buildplanar}), because the digraphs $G_1, G_2$ might not be 
weightable. Indeed,  even if $G$ is diplanar and $G[A],G[B]$ are 1-weak, the digraphs $G_1,G_2$ still
might not be weightable. This is shown by the digraph in Figure~\ref{fig:badcut}.

\begin{figure}[htb!]
\centering
\begin{tikzpicture}

\vtx{a}{2,2};
\vtx{b}{2,0};
\vtx{c}{0,2};
\vtx{d}{0,0};
\vtx{f}{-2,2};
\vtx{e}{-2,0};

\foreach \from/\to in {b/a, d/c, b/d}
\dedge{\from}{\to};

\tdedge[bend left = 30]{c}{a};
\dedge[bend left = 30]{a}{c};
\tdedge[bend right = 30]{d}{e};
\dedge[bend right = 30]{e}{d};
\tdedge[bend right = 30]{c}{f};
\dedge[bend right = 30]{f}{c};
\dedge[bend right = 60]{e}{b};
\tdedge[bend left = 30]{e}{f};
\dedge[bend left = 30]{f}{e};

\draw[dashed, thick] (2,1) ellipse (0.75 and 1.5);

\end{tikzpicture}

\caption{A cut-curve with change number two in a 1-strong diplanar weightable digraph; contracting the edge inside the cut-curve makes the digraph not weightable.} \label{fig:badcut}
\end{figure}
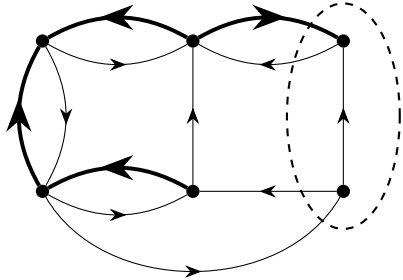

So what condition on $G$ and $F$ do we need to ensure that $G_1,G_2$ are weightable? The following is the most general we have found.
(If $P$ is a dipath and $u,v\in V(P)$, and $u$ is earlier than $v$ in $P$, then
$P[u,v]$ denotes the subpath of $P$ from $u$ to $v$.)

\begin{thm}\label{goodcut}
Let $G$ be a digraph drawn in a 2-sphere, and let $F$ be a cut-curve with change number two. Suppose that for every two edges $e,f$
crossing $F$ in opposite directions, there is a dicycle of $G$ that contains both $e,f$. Then
$G_1,G_2$ (defined as in~\ref{buildplanar}) are weightable.
\end{thm}
\begin{proof} Let $A,B,a,b$ be as in~\ref{buildplanar}. We will show that $G_2$ is weightable. Suppose not; then are two dicycles $C_1,C_2$
of $G_2$ that disagree on some triple $\{x,y,z\}$ of $G_2$. For $i = 1,2$, if $a\notin V(C_i)$ let $C'_i=C_i$.
If $a\in C_i$, let $e_i,f_i$ be the edges of $C_i$ that have $G_2$-head $a$ and $G_2$-tail $a$ respectively, let $P_i$ be a
dipath of $G[A]$ from the $G$-head of $e_i$ to the $G$-tail of $f_i$, let $Q_i = C_i\setminus a$,
and let $C_i'$ be the dicycle of $G$ formed by the union of $P_i, Q_i, e_i$ and $f_i$.  
(Such a path $P_i$ exists since 
there is a dicycle of $G$ containing $e_i,f_i$, and since $F$ has change number two, this cycle only crosses $F$ twice.)
Since $G$ is weightable, $C'_1,C_2'$ do not disagree on $\{x,y,z\}$ in $G$, and so one of $x,y,z$ equals $a$, say $z$. 
If some vertex $w\in A$ belongs to both $V(P_1), V(P_2)$, then the order of $x,y,z$ in $C_i$ is the same as the order of $x,y,w$ 
in $C_i'$, for $i = 1,2$, so  $C'_1,C_2'$ disagree on $\{x,y,w\}$, a contradiction. Thus, 
$P_1,P_2$ are vertex-disjoint, and so we may assume that $e_1,e_2,f_2,f_1$ appear in this order in the cyclic order of the edges 
that cross $F$. 

Now $x,y\in V(Q_1)\cap V(Q_2)$, and since $C_1,C_2$ disagree on $\{x,y,a\}$ in $G_2$, we may assume that $x$ is before $y$ in $Q_1$ and
$y$ is before $x$ in $Q_2$. Since $y$ is before $x$ in $Q_2$, there is a minimal dipath $Q_2[r,s]$ of $Q_2$, such that      
$r,s\in V(Q_1)$, and $s$ is strictly before $r$ in $Q_1$. It follows that no edge or internal vertex of $Q_2[r,s]$ belongs to $Q_1$ 
(because if some $t$
belongs to $V(Q_1)$ and the interior of $Q_2[r,s]$, then either $t$ is after $s$ or before $r$ in $Q_1$, and in either case the minimality of 
$Q_2[r,s]$ is contradicted). Let $R_2$ be the subpath of $Q_2$ from the $G$-head of $f_2$ to $r$, and let $S_2$ be the subpath from $s$ to the 
$G$-tail of $e_2$. Thus $Q_2$ is the concatenation of $R_2, Q_2[r,s]$, and $S_2$. Let $S_1$ be the subpath of $Q_1$ from the $G$-head of $f_1$
to $s$, and let $R_1$ be the subpath from $r$ to the $G$-tail of $e_1$. Thus $Q_1$ is the concatenation of $S_1$, $Q_1[s,r]$,
and $R_1$. (See Figure~\ref{fig:goodcutpf}.)

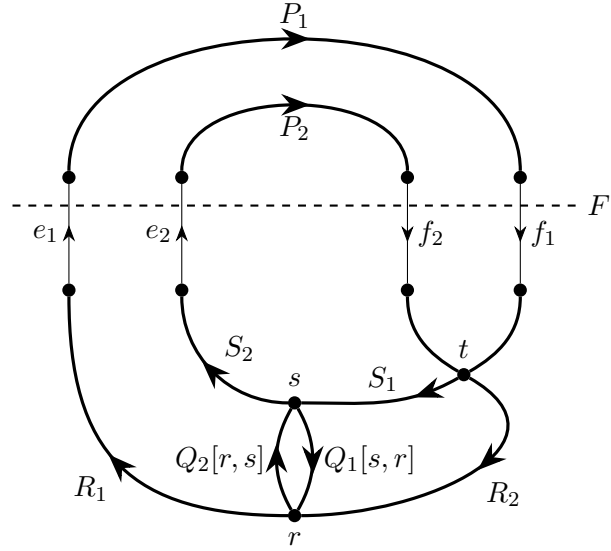
\begin{figure}[htb!]
    \centering
    \begin{tikzpicture}[scale=1.5]
        \vtx{e1h}{0,1};
        \vtx{e1t}{0,0};
        \vtx{e2h}{1,1};
        \vtx{e2t}{1,0};
        \vtx{f2t}{3,1};
        \vtx{f2h}{3,0}:
        \vtx{f1t}{4,1};
        \vtx{f1h}{4,0};
        \vtx[below:$r$]{r}{2, -2};
        \vtx[above:$s$]{s}{2, -1};
        \vtx[above:$t$]{t}{3.5, -0.75};

        \draw[dashed, thick] (-0.5, 0.75) -- (4.5, 0.75) node[right] {$F$};

        \draw[midarrow] (e1t) -- (e1h) node[midway, left] {$e_1$};
        \draw[midarrow] (e2t) -- (e2h) node[midway, left] {$e_2$};
        \draw[midarrow] (f1t) -- (f1h) node[midway, right] {$f_1$};
        \draw[midarrow] (f2t) -- (f2h) node[midway, right] {$f_2$};

        \draw[very thick, midarrow] (e1h) to[out=90, in=90] node[midway, above] {$P_1$} (f1t);
        \draw[very thick, midarrow] (e2h) to[out=90, in=90] node[midway, below] {$P_2$} (f2t);
        
        \draw[very thick, midarrow] (s) to[out=180, in=270] node[midway, above right] {$S_2$} (e2t);
        \draw[very thick, midarrow] (f2h) to[out=270, in=150] (t) to[out=330, in=0, looseness=1.5] node[midway, below right] {$R_2$} (r);

        \draw[very thick, midarrow] (r) to[out=180, in=270, looseness=1.3] node[midway, below left] {$R_1$} (e1t);
        \draw[very thick, midarrow] (f1h) to[out=270, in=30] (t) to[out=210, in=0] node[midway, above] {$S_1$} (s);
        \draw[very thick, midarrow] (r) to[out=120, in=240] node[midway, left] {$Q_2[r,s]$} (s);
        \draw[very thick, midarrow] (s) to[out=300, in=60] node[midway, right] {$Q_1[s,r]$} (r);
    \end{tikzpicture}
    \caption{One of the cases in the proof of~\ref{goodcut}, where $R_2$ and $S_1$ intersect. Here, the arrows, with the exceptions of $e_1,e_2,f_1,f_2$, represent paths (possibly of length 0), not edges. There may be other intersections between some of the paths; the simplest possibility is shown. The other cases correspond to other ways that $R_1\cup R_2$ and $S_1\cup S_2$ might intersect. (There must be some intersection since $e_1,e_2,f_2,f_1$ appear in that cyclic order on the cut $F$.)}
    \label{fig:goodcutpf}
\end{figure}

If some vertex $t\in V(R_2)$ also belongs to $R_1$ with $t\ne r$, then $t\ne s$, and $t,r,s$ appear in this order in $Q_2$, 
and in the order $s,r,t$ in $Q_1$, so
$C_1', C_2'$ disagree on $\{r,s,t\}$, a contradiction. If there is some $t\in V(R_2) \cap V(S_1)$ then $t\ne r,s$ and again
$C_1', C_2'$ disagree on $\{r,s,t\}$. (This is the case depicted in Figure~\ref{fig:goodcutpf}.) Similarly, $V(S_1\cap S_2)=\{s\}$ and $V(S_2\cap R_1)=\emptyset$. Consequently $R_1\cup R_2$ and $S_1\cup S_2$ are vertex-disjoint dipaths, which contradicts that $e_1,e_2,f_2,f_1$ appear in this order in the cyclic order of the edges
that cross $F$.

Hence there are no such $C_1,C_2$, and so $G_2$ is weightable, and similarly $G_1$ is weightable. This proves~\ref{goodcut}.\end{proof}

In the same notation, we say $w\in A$ is a \emph{centre} for $A$ if for each edge $uv$ with $u\in B$ and $v\in A$ there is a dipath of 
$G[A]$ from $v$ to $w$, and for  
each edge $uv$ with $u\in A$ and $v\in B$ there is a dipath of $G[A]$ from $w$ to $u$. We define a centre for $B$ similarly.
An easy way to arrange that a cut-curve with change number two has the property of~\ref{goodcut} is to ensure that $A,B$ have centres.
A \emph{central cycle} for $A$ is a dicycle of $G[A]$ such that all its vertices are centres for $A$.

In view of~\ref{goodcut} and~\ref{buildplanar}, we might now look for a theorem that if $G$ is 1-strong, weightable, and drawn in the plane, then it admits a cut-curve as in~\ref{goodcut}, unless 
$G$ is already suffiently simple to be understood. But we don't really need that, because we already reduced the general problem of constructing all weightable digraphs to constructing those that are 
2-strong and 3-weak; and we will show later that if $G$ is 2-strong, 3-weak, planar and weightable
then it is diplanar. So
we could confine ourselves to finding cut-curves in 2-strong 3-weak diplanar weightable digraphs if we wanted, and this extra information will be helpful.

But that turns out to be \emph{too much}
information. We can build such digraphs from smaller digraphs, but the smaller digraphs need not be 2-strong, even if the 
digraph we are building is 2-strong. That looks like a difficulty; but, fortunately, the same building method also serves to build all
strong diplanar weightable digraphs rather than just those that are 2-strong.
Thus, to get a result saying that we can build all digraphs with property X from smaller digraphs with property X, 
we will take property X to be ``strong, diplanar, and weightable''. 

Let $S$ be a 2-sphere, and let $C_1,C_2$ be dicycles, drawn
in $S$, and bounding closed discs with disjoint interiors. Fix an orientation ``clockwise'' of the 2-sphere.
The rotation of $C_1$ around its disc defines an orientation of the 2-sphere, and if $C_1,C_2$ define the same orientations of the 2-sphere,
we say they are {\em similarly-oriented}.
A great merit of working with diplanar drawings is that, 
if $C_1,C_2$ are similarly-oriented cycles that bound discs in the 2-sphere 
with disjoint interiors, 
then $C_1,C_2$ are vertex-disjoint, as is easily seen.

Let us assign an orientation ``clockwise'' to the plane.  In a digraph
drawn in the plane, each dicycle rotates clockwise or counterclockwise in the natural sense; a \emph{clockwise cycle} means a dicycle that
rotates clockwise, and a \emph{counterclockwise cycle} is defined similarly. (It is important that we are working with drawings
in the plane rather than in a 2-sphere.)

We extend the ``similarly-oriented'' terminology to planar drawings in the natural way; that is, we say two dicycles
in a planar drawing are similarly-oriented if they are similarly-oriented in the 2-sphere obtained by one-point compactification of the plane.
Thus, if $C_1,C_2$ are similarly-oriented vertex-disjoint dicycles in a planar
drawing, then each bounds a unique disc in the plane, and these discs might be disjoint, or one might contain the other. In the first
case, either both $C_1,C_2$ are clockwise in the plane, or both are counterclockwise. In the second case, one is clockwise in the plane and
the other is counterclockwise (which might seem paradoxical at first sight, since we called them ``similarly-oriented''). 
We will prove:
\begin{thm}\label{planarcut} 
Let $G$ be a 1-strong, weightable digraph with a diplanar drawing in the 2-sphere. Suppose that there are two vertex-disjoint similarly-oriented  cycles $D_1,D_2$ 
in $G$, and choose $D_1,D_2$ such that the annulus between them is minimal. Then $G$ admits a cut-curve $F$ with change number two, such that, in the usual notation, $D_1$ is 
a central cycle for $A$ and $D_2$ is a central cycle for $B$.
\end{thm}
\begin{proof} With a given planar drawing, 
for each cycle $C$ of $G$, let $\ins(C)$ be the closed disc in the plane bounded by $C$.
Since $G$ admits a diplanar drawing in a 2-sphere in which $D_1,D_2$ bound disjoint discs and $D_1,D_2$ are similarly-oriented,
$G$ also admits a 
diplanar drawing in the plane such that $D_1$ is clockwise, and $D_2$ is counterclockwise, and $\ins(D_1)\subseteq \ins(D_2)$.
Fix such a drawing. 
Let $\Sigma$ be the annulus in the plane between $D_1,D_2$ (including $D_1,D_2$). 
The choice of $D_1,D_2$ implies that:

\begin{step}\label{step1a}
There is no dicycle $C$ in $\Sigma$ different from $D_1,D_2$ such that $\ins(D_1)\subseteq \ins(C)\subseteq \ins(D_2)$.
\end{step}

Let us say a set of dicycles of $G$ is \emph{free} if its members bound discs in the plane with disjoint interiors, and each of them is 
drawn in $\ins(D_2)$ and vertex-disjoint from $D_2$. 
If $\mac C$ is a free set of dicycles,
let $U(\mac C)$ denote the union of the members of $\mac C$, and let $I(\mac C)$ be the subdigraph of $G$ consisting of all vertices and edges that are drawn in or inside some $C\in \mac C$. 

Choose a free set $\mac C$ of dicycles with $D_1\in \mac C$ such that
$U(\mac C)$ is 1-strong and, subject to that, with $I(\mac C)$ maximal. 
Thus $I(\mac C)$ is 1-strong.

\begin{step}\label{step2a}
Every ear for $I(\mac C)$ has a vertex in $D_2$.
\end{step}

Let $P$ be an ear for $I(\mac C)$, and suppose that $V(P\cap D_2) = \emptyset$. Consequently $P$ is drawn in $\Sigma$, and either 
$P$ is a dipath with both ends in $U(\mac C)$ and no internal vertex or edge in $I(\mac C)$, or $P$ is a dicycle with one 
vertex in $U(\mac C)$ and with no other vertex or edge in $I(\mac C)$.
Since $U(\mac C)$ is 1-strong, there is a dipath $Q$ 
(possibly of length zero)
of $U(\mac C)$ such that $P\cup Q$ is a dicycle. Since $Q$ intersects $I(\mac C)$ only in $U(\mac C)$, it follows that for 
each $C\in \mac C$, either $\ins(C)\subseteq \ins(P\cup Q)$, or the interiors of $\ins(C), \ins(P\cup Q)$ are disjoint. Let $\mac C'$
be the set consisting of $P\cup Q$ and all $C\in \mac C$ such that  $\ins(C), \ins(P\cup Q)$ are disjoint. From~\eqref{step1a}, $D_1\in \mac C'$,
and $\mac C'$ is free, and $I(\mac C)$ is a proper subset of $I(\mac C')$, a contradiction. This proves~\eqref{step2a}.

\begin{step}\label{step3a}
For every edge $e$ with one end in $V(U(\mac C))$ that does not belong to $I(\mac C)$, there is a dipath containing $e$
with one end in $V(U(\mac C))$, the other end in $V(D_1)$, and with no internal vertex in either set.
\end{step}

Let $e=ab$, say. From the symmetry, we may assume that $a\in V(U(\mac C))$. 
Since $G$ is 1-strong, $e$ belongs to a dicycle $C$ 
of $G$. Let $P$ be the minimal subpath of $C$ from $b$ to $V(U(\mac C))\cup V(D_2)$ (this exists since $a\in V(U(\mac C))$. 
Let the ends of $P$ be $b,p$. If $p\in V(U(\mac C))$ then the union of $P$ and $e$ is an ear violating \eqref{step2a}, so $p\in V(D_2)$, and the union
of $P$ and $e$ satisfies \eqref{step3a}. This proves \eqref{step3a}.

\bigskip

By \eqref{step2a}, every edge with both ends in $V(U(\mac C))$ belongs to $I(\mac C)$. Consequently, 
there is a cut-curve $F$, obtained by closely following the outer boundary of $U(\mac C)$, such that the edges that cross $F$
are precisely the edges $e$ with exactly one end in $V(U(\mac C))$. If $F$ has change number two, we are done, so we assume for a 
contradiction that $F$ has change number at least four. Hence there are edges $e_1,e_2,e_3,e_4$, each crossing $F$ and numbered 
according to the clockwise
cyclic order defined by $F$, such that $e_1,e_3$ have tail in $V(U(\mac C))$ and $e_2,e_4$ have head in $V(U(\mac C))$. By \eqref{step2a},
there are dipaths $P_1,\dots, P_4$ such that $e_i$ is an edge of $P_i$ for $1\le i\le 4$, and such that $P_1,P_3$ are from
the tail of $e_i$ to $D_2$ and $P_2,P_4$ are from $D_2$ to the head of $e_i$, and for $1\le i\le 4$,
no internal vertex of $P_i$ belongs to $V(U(\mac C))\cup V(D_2)$. Let $P_i$ be from $a_i$ to $b_i$ for $1\le i\le 4$. 

\begin{step}\label{step4a}
For $1\le i\le 4$, $P_i, P_{i+1}$ are internally disjoint. (Here $P_5$ means $P_1$.)
\end{step}

If $w$ belongs to the interiors of $P_i, P_{i+1}$, then $P_i\cup P_{i+1}$ includes an ear for $I(\mac C)$ violating \eqref{step2a}.
This proves \eqref{step4a}.

\bigskip

Please refer to Figure~\ref{fig:planarcutpf} for the end of the proof.

\begin{figure}[htb!]
    \centering
    \begin{tikzpicture}[scale=2]
        \vtx[{[label distance=-3pt]225:$a_3$}]{a3}{0, -1};
        \vtx[{[label distance=-3pt]225:$y$}]{y}{-1, 0};
        \vtx[{[label distance=-3pt]above left:$b_3$}]{b3}{0, -3};
        \vtx{1}{0,1};
        \vtx{2}{0.707,-0.707};
        \vtx[above left:$x$]{x}{1, 0};
        \vtx[left:$b_2$]{b2}{1.5, -0.5};
        \vtx[{[label distance=-3pt]above right:$a_2$}]{a2}{3, 0};
        \vtx[{[label distance=-3pt]above right:$b_1$}]{b1}{0, 3};
        \vtx[{[label distance=-3pt]below right:$a_4$}]{a4}{-3, 0};
        \vtx[{[label distance=-3pt]135:$b_4$}]{b4}{-1.8, 0.5};
        \vtx{3}{-1.5,1.2};
        \vtx{4}{-0.5,1.5};
        \vtx[{[label distance=-3pt]135:$a_1$}]{a1}{-1, 2};

        \hdedge[out=180, in=270]{a3}{y};
        \hdedge[out=90, in=180]{y}{1};
        \hdedge[out=0, in=90]{1}{x};
        \hdedge[out=270, in=45]{x}{2};
        \hdedge[out=225, in=0]{2}{a3};
        \hdedge[out=330, in=255]{2}{b2};
        \hdedge[out=75, in=0]{b2}{x};
        \dedge[out=0, in=270]{b3}{a2};
        \dedge[out=90, in=0]{a2}{b1};
        \dedge[out=180, in=90]{b1}{a4};
        \dedge[out=270, in=180]{a4}{b3};
        \hdedge[out=270, in=180]{b4}{y};
        \hdedge[out=210, in=90]{3}{b4};
        \hdedge[out=180, in=30]{4}{3};
        \hdedge[out=90, in=0]{1}{4};
        \hdedge[out=0, in=90]{a1}{4};
        \hdedge[in=180, out=120]{3}{a1};
        \draw[midarrow] (a3) -- (b3) node[midway, right] {$P_3$};
        \draw[midarrow] (a2) to[out=180, in=0] node[midway, above left] {$P_2$} (b2);
        \draw[midarrow] (a4) to[out=0, in=180] node[midway, below right] {$P_4$} (b4);
        \draw[midarrow] (a1) to[out=90, in=270] node[midway, below right] {$P_1$} (b1);

        \draw[dashed, very thick] ([shift={(0.1,0)}]b2.center) to[out=75, in=0] ([shift={(0,0.1)}]x.center) to[out=180, in=90] ([shift={(-0.1,0)}]x.center) to[out=270, in=0] ([shift={(0,0.1)}]a3.center) to[out=180, in=270] ([shift={(0.1,0)}]y.center) to[out=90, in=180] ([shift={(0,-0.1)}]1.center) to[out=0, in=270] ([shift={(0.1,0)}]1.center) to[out=90, in=0] node[midway, label={[label distance=-10pt]45:$Q_{2,1}$}] {} ([shift={(0,0.1)}]4.center) to[out=180, in=30] ([shift={(0.04,0.12)}]3.center) to[out=120, in=180] ([shift={(0,-0.1)}]a1.center);

        \draw[dashed, very thick] ([shift={(0,0.1)}]b1.center) to[out=180, in=90] node[midway, label={[label distance=-10pt]135:$R_{1,2}$}] {} ([shift={(-0.1,0)}]a4.center) to[out=270, in=180] ([shift={(0,-0.1)}]b3.center) to[out=0, in=270] ([shift={(0.1,0)}]a2.center);

        \draw[dotted, very thick] ([shift={(0.1,0)}]b4.center) to[out=270, in=180] ([shift={(-0.1,0.1)}]y.center) to[out=90, in=180] ([shift={(0,0.1)}]1.center) to[out=0, in=90] node[midway, label={[label distance=-10pt]45:$Q_{4,3}$}] {} ([shift={(0.1,0)}]x.center) to[out=270, in=0] ([shift={(0,-0.1)}]a3.center);

        \draw[dotted, very thick] ([shift={(0,0.1)}]b3.center) to[out=0, in=270] ([shift={(-0.1,0)}]a2.center) to[out=90, in=0] node[midway, label={[label distance=-10pt]225:$R_{3,4}$}] {}([shift={(0,-0.1)}]b1.center) to[out=180, in=90] ([shift={(0.1,0)}]a4.center);

        \node at (0.55,0.55) {$D_1$};
        \node at (2.3,2.3) {$D_2$};
    \end{tikzpicture}
    \caption{Part of the proof of~\ref{planarcut}. The innermost (clockwise) cycle is $D_1$, and the outermost (counterclockwise) cycle is $D_2$. The union of the thicker arrows in the middle of the figure is $U(\mac C)$. The dashed lines mark $R_{1,2}$ and $Q_{2,1}$, and the dotted lines mark $R_{3,4}$ and $Q_{4,3}$.}
    \label{fig:planarcutpf}
\end{figure}
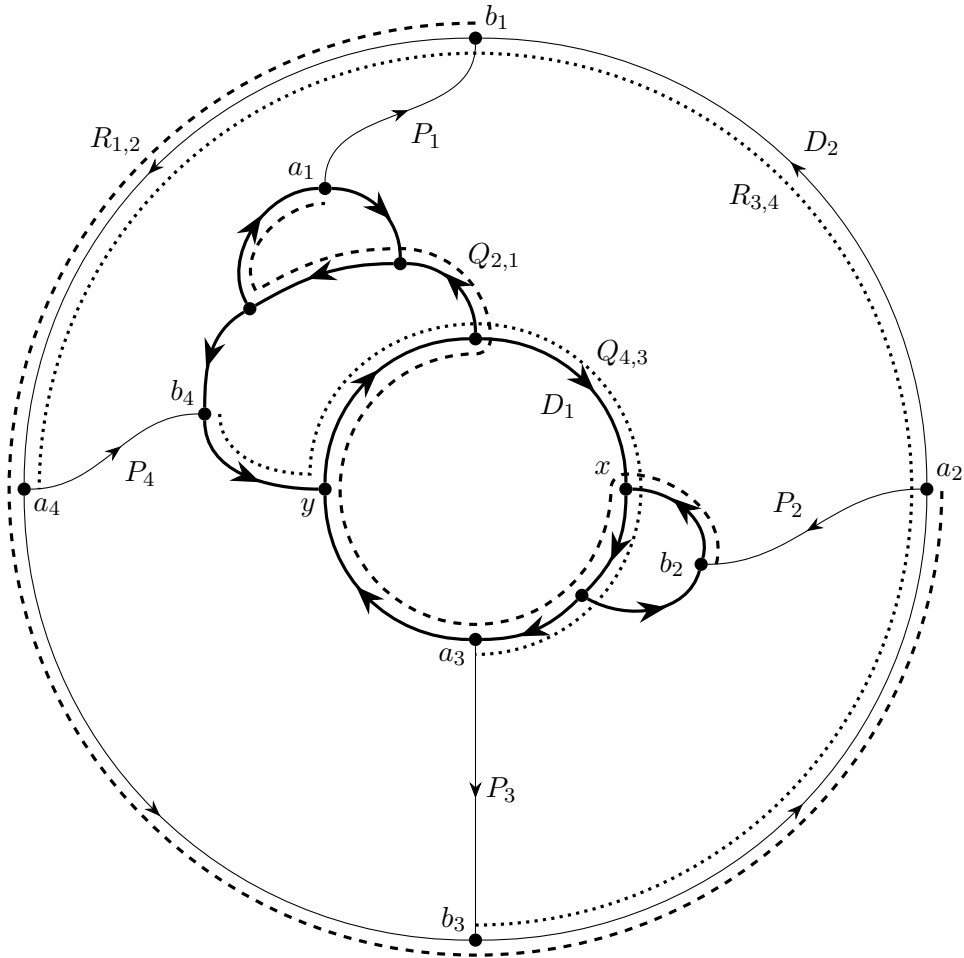

Let $Q_{2,1}$ be a path of $U(\mac C)$ from the last vertex of $P_2$ to the first vertex of $P_1$.
Suppose that $b_1 = a_2$. Since the drawing of $G$ is diplanar and $D_2$ is a counterclockwise circuit that intersects 
the dicycle  $P_1\cup P_2\cup Q_{2,1}$, it follows that $P_1\cup P_2\cup Q_{2,1}$ is counterclockwise, and 
since $e_1,e_2,e_3,e_4$
are in clockwise order in $F$, it follows 
that $b_1,a_2,b_3,a_4$ are all equal, contradicting that the drawing is diplanar. Thus $b_1\ne a_2$, and similarly $b_3\ne a_4$. 
Let $R_{1,2}$ be the subpath of 
$D_2$ from the last vertex of $P_1$ to the first vertex of $P_2$. Then $P_2\cup Q_{2,1}\cup P_1\cup R_{1,2}$ is a counterclockwise 
dicycle, and by \eqref{step1a}, $\ins(D_1)$ is not a subset of $\ins(P_2\cup Q_{2,1}\cup P_1\cup R_{1,2})$. So $\ins(D_1)$ is a subset 
of $\ins(P_2\cup Q_{2,1}\cup P_1\cup R_{2,1})$, where $R_{2,1}$ is the dipath of $D_2$ from $a_2$ to $b_1$. (This cycle is
not a dicycle.) Similarly, $\ins(D_1)$ is a subset of $\ins(P_1\cup Q_{4,3}\cup P_3\cup R_{4,3})$, where $Q_{4,3}, R_{4,3}$
are respectively a dipath of $U(\mac C)$ from $b_4$ to $a_3$ and a dipath of $D_2$ from $a_4$ to $b_3$. 
Consequently, $P_2\cup Q_{2,1}\cup P_1$ and $P_4\cup Q_{4,3}\cup P_3$ are not vertex-disjoint.
Let $x,y$ be the first and last vertices of the path $P_2\cup Q_{2,1}\cup P_1$ that belong to $P_4\cup Q_{4,3}\cup P_3$. 
Since the drawing is diplanar it follows that $x\ne y$, and hence from planarity, $y$ is strictly earlier than $x$ in the path  
$P_4\cup Q_{4,3}\cup P_3$. But then the dicycles $P_2\cup Q_{2,1}\cup P_1\cup R_{1,2}$ and $P_4\cup Q_{4,3}\cup P_3\cup R_{3,4}$
disagree on $\{x,y, b_1\}$ (where $R_{4,3}$ is the dipath of $D_2$ from $a_4$ to $b_3$), a contradiction. This proves~\ref{planarcut}.
\end{proof}

A drawing in a 2-sphere can be converted to a drawing in the plane by removing from the 2-sphere one point in some region of
the drawing. Let us call this \emph{puncturing} the drawing. 
Let us say a drawing in the plane (without crossings) of a digraph $G$ is \emph{circular} if the origin belongs to one of the regions, 
and each edge of $G$ is drawn as a curve that moves monotonically in a clockwise direction around the origin. (This is a more exact 
restatement of the definition of ``circular'' given in the introduction.) Theorem~4.2 of~\cite{auer} implies that if 
a 1-strong digraph $G$ admits a diplanar drawing in a 2-sphere with no two vertex-disjoint similarly-oriented
cycles, then $G$ admits a circular drawing in the plane. Here is a slight strengthening (we omit its proof): one can obtain a circular
drawing by puncturing the 2-sphere drawing at some point inside a region bounded by a directed cycle. 

Thus, from~\ref{planarcut} we deduce our first main result:
\begin{thm}\label{planarcon}
Let $G$ be a 1-strong weightable digraph with a diplanar drawing in the 2-sphere. Then either:
\begin{itemize}
\item the drawing of $G$ can be constructed by the construction of~\ref{buildplanar} from diplanar drawings of two smaller 1-strong 
weightable digraphs, in such a way that the sets $A,B$ of~\ref{buildplanar} have similarly-oriented central cycles, or
\item by puncturing the drawing one can obtain a circular drawing in the plane.
\end{itemize}
\end{thm}

\section{Carvings of planar digraphs}
\label{sec:carving}

There is another interesting way to view 1-strong diplanar weightable digraphs, following an approach in~\cite{ratcatcher} for 
decomposing planar graphs. 
Let $V$ be a finite set with $|V| \ge 2$. 
Two subsets $A, B\subseteq V$ \emph{cross} if $A\cap B$, $A\setminus B$, $B\setminus A$ and $V\setminus (A\cup B)$ are all nonempty.
A \emph{carving} in $V$ is a set $\mac C$ of subsets of $V$, such that:
\begin{itemize}
\item $\emptyset, V\ne \mac C$;
\item no two members of $\mac C$ cross; and
\item $\mac C$ is maximal with this property.
\end{itemize}
It follows that if $A\in \mac C$ then $V\setminus A\in \mac C$, and $\{v\}\in \mac C$ for each $v\in V$.

One can view a carving as arising from a tree, as follows.
(The leaves of a tree are its vertices of degree 1.)
\begin{thm}\label{carvings}
Let $V$ be a finite set with $|V|\ge 2$, let $T$ be a tree in which every vertex has
degree 1 or 3, and let $\tau$ be a bijection from $V$ onto the set of leaves of $T$. For each
edge $e$ of $T$ let $T_1(e), T_2(e)$ be the two components of $T\setminus e$, and let
\[\mac C = \{\{v \in  V : \tau(v)\in V(T_i(e))\}:  e\in  E(T), i = 1, 2\}.\]
Then $\tau$ is a carving in $V$. Conversely, every carving in $V$ arises from some tree $T$
and bijection $\tau$ in this way. 
\end{thm}
(This is theorem 1.1 of~\cite{ratcatcher}.)

The main result of~\cite{ratcatcher} is a poly-time algorithm that, given as input some planar graph $G$ with $|V(G)|\ge 2$,
finds a carving $\mac C$ of $V(G)$ such that $\max_{C\in \mac C}|\delta(C)|$ is as small as possible, where $\delta(C)$ denotes the 
set of edges with an end in $C$ and an end in $V(G)\setminus C$. (Its running time was
$O((|V(G)|+|E(G)|)^2)$, where the multiplicative constant was reasonable.) But now we want to use carvings for 
planar digraphs, and instead of minimizing $\max_{C\in \mac C}|\delta(C)|$, we want to minimize something else, related to change number.

This can be done in a few different ways, but the neatest is only possible if we assume that the digraph $G$ is 1-strong, 2-weak, and loopless.
If $G$ is drawn in the plane and $A\subseteq V(G)$ such that $G[A]$ and $G[V(G)\setminus A]$ are both nonnull and 1-weak, 
the set of edges between $C$ and $V(G)\setminus C$ is a ``bond'' of 
the graph underlying $G$, that is, a minimal edge-cutset, and so it corresponds to a cycle of the dual graph. 
Hence there is a cut-curve $F$ separating $C,V(G)\setminus C$, and the edges crossing $F$ are the edges in $\delta(C)$. Let us say the
\emph{change number} of $C$ is the change number of $F$.

Let us say a \emph{bond carving} of a planar digraph $G$ is a carving $\mac C$ of $G$ such that $G[C]$ is 1-weak 
for all $C\in \mac C$. 
The \emph{diwidth} of $\mac C$ is the 
maximum over all $C\in \mac C$ of the change number of $C$. 
We will prove:
\begin{thm}\label{widthweight}
Let $G$ be a 1-strong, 2-weak, loopless digraph with a diplanar drawing in the plane. Then $G$ is weightable if and only if $G$ admits
a bond carving of diwidth two.
\end{thm}
To prove this we need a couple of lemmas. First:
\begin{thm}\label{usecut}
Let $G$ be a 1-strong, 2-weak, loopless digraph drawn in the plane, and let $A,B$ be a partition of $V(G)$ into two nonnull subsets such that $G[A],G[B]$ are both 
1-weak. Let $G_1,G_2$ be the drawings obtained by contracting all the edges of $G[A]$ (respectively, all edges of $G[B]$). Then 
both $G_1,G_2$ are 1-strong, 2-weak, and loopless, and if they both admit bond carvings of diwidth two then so does $G$.
\end{thm}
\begin{proof} Clearly $G_1,G_2$ are 1-strong and loopless. Let $a$ be the vertex made by contracting $G[A]$ into a vertex, and define $b$ similarly.
To see that $G_1,G_2$ are 2-weak, note that $G_1\setminus v$ is 1-weak for all $v\ne a$
since $G$ is 2-weak, and $G_1\setminus a=G[B]$ is 1-weak. So $G_1$ is 2-weak, and similarly so is $G_2$. 

Suppose that for $i = 1,2$, $\mac C_i$ is a bond carving of $G_i$ with diwidth two. Thus $\{a\}\in \mac C_1$ and $\{b\}\in \mac C_2$.
Let 
\[ \mac C_1'=\{C\in \mac C_1:a\notin C\}\cup \{A\cup (C\setminus \{a\}):C\in \mac C_1\text{ with } a\in C\}, \]
and define $\mac C_2'$ similarly. Thus $A,B$ belong to both $\mac C_1, \mac C_2$. Let $\mac C = \mac C_1'\cup \mac C_2'$; then 
$\mac C$ is a bond carving of $G$ with diwidth two. This proves~\ref{usecut}.\end{proof}

Second, we need:
\begin{thm}\label{breakcircular}
Let $G$ be 1-strong, 2-weak, and loopless, and suppose $G$ admits a diplanar drawing in the plane such that every dicycle is clockwise and 
bounds a disc including the origin. Then $G$ admits a  bond carving of diwidth two.
\end{thm}
\begin{proof}
We proceed by induction on $|V(G)|+|E(G)|$. Thus, we may assume that $G$ is simple. The boundary of the infinite region is a cycle
(because $G$ is a 2-weak graph drawn in the plane), and it is a clockwise dicycle (because each of its edges is in a directed cycle, 
since $G$ is 1-strong, and this directed cycle is clockwise by hypothesis). Let $D$ be this cycle. 

\begin{step}\label{step1b}
We may assume that no vertex $v$ of $D$ has outdegree one and indegree one.
\end{step}

Suppose that $v\in V(D)$ has outdegree one and indegree one. Let $uv$ be the edge with head $v$. If $D$ has length two, then 
either $u$ is a 1-vertex cutset or $|V(G)|=2$,
in either case contradicting that $G$ is 2-weak. So $D$ has length at least three. If $|V(G)|=3$, then $G=D$ and the theorem is true; 
so we assume that $|V(G)|\ge 4$. Hence the digraph $G_1$ obtained by contracting the edge $uv$ is 1-strong and 2-weak.
From the inductive hypothesis, $G_1$ admits a bond carving $\mac C_1$ with diwidth two. Let $w$ be the vertex made by identifying 
$u,v$ under contraction. Let 
\[\mac C=\{C\in \mac C_1:w\notin C\}\cup \{\{u,v\}\cup (C\setminus \{w\}):C\in \mac C_1, w\in C\}\cup \{\{u\}, \{v\}\}.\]
Then $\mac C$ is a bond carving of $G$ of diwidth two, as required. This proves~\eqref{step1b}.

\begin{step}\label{step2b}
We may assume that $G\setminus \{e,f\}$ is 1-weak for all distinct $e,f\in E(D)$.
\end{step}

Suppose not, and let $e = a_1b_1$ and $f=b_2a_2$. Since $G$ is 2-weak and by~\eqref{step1b}, it follows that $a_1,b_1,b_2,a_2$ are all distinct. 
Since $G\setminus \{e,f\}$ is not weakly connected, it has exactly two weak components, one (say $A$) including the dipath of $D$ 
from $a_2$ to $a_1$, and the other ($B$)  including the dipath from $b_1$ to $b_2$.  Let $G_1$ be obtained from $G$ by contracting 
all edges  in $G[A]$, and $G_2$ by contracting the edges of $G[B]$. Since $A,B$ have change number two, and hence $G_1,G_2$ admit
diplanar drawings that satisfy the hypothesis of the theorem, we deduce from the inductive hypothesis that $G_1,G_2$ both
admit bound carvings of diwidth two. But then so does $G$, by~\ref{usecut}. This proves~\eqref{step2b}.

\bigskip

From~\eqref{step2b} and theorem 2.1 of~\cite{dirSP} (or just by choosing an edge $e$ of $D$ in as few directed cycles as possible), we deduce that
there is an edge $e\in E(D)$ such that $G\setminus e$ is 1-strong. Let $e = ab$.

\begin{step}\label{step3b}
$G\setminus e$ is 2-weak.
\end{step}

Suppose not; then there is a vertex $w$ of the path $D\setminus e$ with $w\ne a,b$, such that $a,b$ belong to different weak components (say, $A,B$ respectively) of $(G\setminus e)\setminus w$. Since $G\setminus w$ is weakly connected, $(G\setminus e)\setminus w$ has at most two weak components, and so 
$A\cup B= (G\setminus e)\setminus w$. Since $G\setminus e$ is 1-strong, there is a directed path in $G\setminus e$ from $a$ to $w$, which is 
therefore a dipath of $A$,  
and similarly there is a dipath of $B$ from $w$ to $b$. Consequently there is a directed cycle of $A$ that contains the edge of $D$
with tail $w$, and a directed cycle of $B$ that contains the edge of $D$ with head $w$. These two cycles are both clockwise, by hypothesis, 
and share exactly one vertex, and both contain an edge incident with the infinite region of $G$, which is impossible since 
each bounds an open disc including the origin.

\bigskip

From~\eqref{step3b}, the boundary of $G\setminus e$ is a directed cycle. The edge $e$ is incident with two regions of the drawing of $G$, one 
the infinite region outside $D$, and the other, $r$ say, inside $D$. The cycle of $G$ that forms the boundary of $r$ consists of $e$
and a path $P$ joining the ends of $e$, and $P$ is part of the boundary of the infinite region of $G\setminus e$. Hence $P$ is a 
directed path from $a$ to $b$. From the inductive hypothesis and~\eqref{step3b}, $G\setminus e$ admits a bond carving $\mac C$ of diwidth two. 
Hence $\mac C$ is also a bond carving of $G$, and we claim that it still has diwidth two. If not, then there exists $A\in \mac C$,
containing exactly one of $a,b$, say $a$, such that in the cyclic order of edges in $\delta(A)$, the edges before and after $e$ have 
head in $A$. But that is impossible, since one of these two edges belongs to $P$. This proves~\ref{breakcircular}.\end{proof}
\begin{proof}[Proof of~\ref{widthweight}.]
Let $G$ be a 
1-strong, 2-weak, loopless digraph with a diplanar drawing in the plane, and we assume first that $G$ is weightable. We must show that 
$G$ admits
a bond carving of diwidth two, and we prove this by induction on $|V(G)|$. 

\begin{step}\label{step1c}
We may assume that there is no partition $(A,B)$ of $V(G)$ with $|A|,|B|\ge 2$ such that $G[A], G[B]$ are both 1-weak
and $A$ has change number two and such that 
the digraphs $G_1,G_2$ obtained by squishing $A$ and squishing $B$, respectively,
are weightable.
\end{step}

Suppose that such $A,B$ exist. Then $G_1,G_2$ are both loopless, 1-strong, and 2-weak, and admit diplanar drawings in the plane 
(since $A$ has change number two). Since $|A|,|B|\ge 2$, we can apply the inductive hypothesis to $G_1,G_2$, and deduce that they 
both admit
bond carvings of diwidth two. But then so does $G$, by~\ref{usecut}. This proves~\eqref{step1c}.

\bigskip

Suppose that there are two vertex-disjoint similarly-oriented  cycles in $G$. By~\ref{planarcut},
$G$ admits a cut-curve $F$ with change number two, such that both parts $A,B$ of the corresponding partition have central cycles. 
But then $G[A],G[B]$ 
are 1-weak, and the correspond digraphs $G_1,G_2$ are weightable, by~\ref{goodcut}, contrary to~\eqref{step1c}.
Thus there are no two vertex-disjoint similarly-oriented  cycles in $G$. Hence $G$ admits a diplanar drawing in the plane such that every directed cycle
is clockwise and bounds an open disc containing the origin. But then the result holds by~\ref{breakcircular}. 
This proves the ``only if'' part of the theorem. 

For the ``if'' part of the theorem, assume now that $G$ admits a bond carving $\mac C$ of diwidth two, and we must prove that 
$G$ is weightable.  We proceed by induction on $|V(G)|$. Suppose that there exists $A\in \mac C$ such that $|A|, |B|\ge 2$, where $B=V(G)\setminus A$. 
Let $G_1,G_2$ be obtained by squishing $A$ and squishing $B$, respectively. Then $G_1,G_2$ both have diplanar drawings, both are 1-strong, 2-weak and loopless, and both admit bond carvings of diwidth two.
From the inductive hypothesis both are weightable. But then, from~\ref{buildplanar}, $G$ is weightable. Thus we may assume that 
there is no such $A\in \mac C$. Hence $\mac C$ is the set of all singleton subsets of $V(G)$ and their complements. From the 
maximality condition in the definition of a carving, it follows that $|V(G)|\le 3$. Since $G$ is diplanar it follows that $G$ is weightable. This proves the ``if'' part, and so 
proves~\ref{widthweight}.\end{proof}

Could we extend this further?
If we want a bond carving, $G$ must be 2-weak, because planar digraphs that are not 2-weak do not admit bond carvings.
But we could drop the ``2-weak'' hypothesis if we were willing to weaken the requirement that the corresponding edge-cutsets must
be bonds. 
Instead of requiring $G[A]$ to be 1-weak for each $A\in \mac C$, we could just ask that for each $A\in \mac C$, there is a 
cut-curve with change number two that separates $A$ and $V(G)\setminus A$ in the natural sense. Every diplanar 1-strong weightable digraph admits a carving with this property, but we omit the details.


\section{Nonplanar compositions}\label{sec:nonplanarcon}

We will prove that every 2-strong, 3-weak,  weightable digraph can be built from planar ones by composition operations that preserve being weightable, and in this section we explain the compositions we will use.

For brevity, if $G$ is a digraph and $A$ is a subdigraph or a subset of $V(G)$, and $v$ is a vertex of $G$ not in $A$, we write $v\to A$ to mean that every edge of $G$ between $v$ and $A$ is from $v$ to $A$, and $v\from A$ to mean that every edge of $G$
between $v$ and $A$ is from $A$ to $v$.
If $G$ is a digraph and $Y\subseteq V(G)$, a \emph{$Y$-wing} is a subdigraph $W$ with $Y\subseteq V(W)$ such that every edge of $G$ 
with an end in $V(W)\setminus Y$ belongs to $W$. (Edges with both ends in $Y$ might or might not belong to $W$.) A $Y$-wing $W$ is
\emph{non-trivial} if $V(W)\ne Y$ and \emph{non-separable} if it is non-trivial and $W\setminus Y$ is 1-weak.
Two $Y$-wings $W_1,W_2$
are \emph{internally disjoint} if $V(W_1)\cap V(W_2) = Y$ and $E(W_1)\cap E(W_2) = \emptyset$. If $W$ is a $Y$-wing, 
a \emph{$W$-path} is a dipath of $W$ with distinct ends both in $Y$ such that none of its internal vertices belong to $Y$ (it might have no internal vertices).

From now on, we use the convention in figures that an arrow drawn inside a vertex set indicates the existence of a path between the two attached vertices; for instance, in Figure~\ref{fig:firstcon}, there must be a path from $y_2$ to $y_1$ in $W_1$.

\begin{thm}\label{firstcon}
Let $G$ be a digraph, let $y_1,y_2,y_3$ be distinct, and let $Y=\{y_1,y_2,y_3\}$. Let $W_1,W_2$
be internally disjoint $Y$-wings with union $G$. Suppose that:
\begin{itemize}
\item there is a $W_1$-path from $y_2$ to $y_1$, and $y_1$ is a source of $W_2$, and $y_2$ is a sink of $W_2$;
\item the digraph $G_1$ obtained
from $W_1$ by adding the edges $y_1y_2, y_1y_3,y_3y_2$ (if they are not already present) is 1-strong and weightable; and
\item the digraph $G_2$
obtained from $W_2$ by adding the edges $y_2y_1, y_1y_3,y_3y_2$ (if they are not already present) is 1-strong and weightable
(equivalently, by~\ref{butterfly}, the digraph obtained from $G_2$ by contracting $y_2y_1$ is 1-strong and weightable).
\end{itemize}
See Figure~\ref{fig:firstcon}. Then $G$ is weightable.
\end{thm}

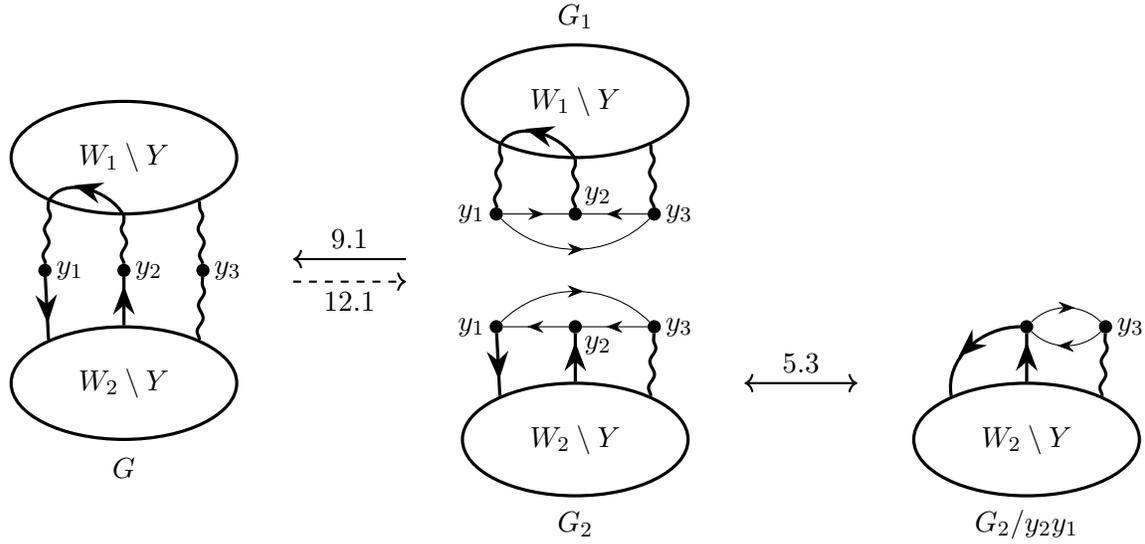
\begin{figure}[htb!]
    \centering
    \begin{tikzpicture}[scale=1.5]
        \tikzset{every label/.style={label distance=-2pt}};
        \vset{W1}{$W_1\setminus Y$}{0,2};
        \vset{W2}{$W_2\setminus Y$}{0,0};

        \vtx[right:$y_1$]{y1}{-0.7, 1};
        \vtx[right:$y_2$]{y2}{0, 1};
        \vtx[right:$y_3$]{y3}{0.7, 1};

        \hdedge{y1}{W2.150};
        \hdedge{W2.90}{y2};
        \sline{W2.30}{y3};

        \sline{y1}{W1.210};
        \sline{y2}{W1.270};
        \sline{y3}{W1.330};
        \hdedge[out=120, in=60]{W1.270}{W1.210};

        \node (G) at (0,-0.75) {$G$};

        \draw[<-, thick] (1.5,1.1) -- (2.5,1.1) node[midway, above] {\ref{firstcon}};
        \draw[->, thick, dashed] (1.5,0.9) -- (2.5,0.9) node[midway, below] {\ref{firstoutcome}};

        \vset{W1}{$W_1\setminus Y$}{4,2.5};
        
        \vtx[left:$y_1$]{y1}{3.3, 1.5};
        \vtx[{[label distance=-3pt]45:$y_2$}]{y2}{4, 1.5};
        \vtx[right:$y_3$]{y3}{4.7, 1.5};
        \sline{y1}{W1.210};
        \sline{y2}{W1.270};
        \sline{y3}{W1.330};
        \dedge[out=315, in=225]{y1}{y3};
        \dedge{y1}{y2};
        \dedge{y3}{y2};
        \hdedge[out=120, in=60]{W1.270}{W1.210};
        \node (G1) at (4, 3.25) {$G_1$};

        \vset{W2}{$W_2\setminus Y$}{4,-0.5};
        
        \vtx[left:$y_1$]{y1}{3.3, 0.5};
        \vtx[{[label distance=-3pt]315:$y_2$}]{y2}{4, 0.5};
        \vtx[right:$y_3$]{y3}{4.7, 0.5};
        \hdedge{y1}{W2.150};
        \hdedge{W2.90}{y2};
        \sline{W2.30}{y3};
        \dedge{y2}{y1};
        \dedge[out=45, in=135]{y1}{y3};
        \dedge{y3}{y2};
        \node (G2) at (4, -1.25) {$G_2$};

        \draw[<->, thick] (5.5,0) -- (6.5,0) node[midway, above] {\ref{butterfly}};

        \vset{W2}{$W_2\setminus Y$}{8,-0.5};
        
        \vtx{w}{8, 0.5};
        \vtx[right:$y_3$]{y3}{8.7, 0.5};
        \hdedge[out=180, in=100]{w}{W2.150};
        \hdedge{W2.90}{w};
        \sline{W2.30}{y3};
        \dedge[out=45, in=135]{w}{y3};
        \dedge[out=225, in=315]{y3}{w};
        \node (G2) at (8, -1.25) {$G_2/y_2y_1$};
    \end{tikzpicture}
    \caption{The first nonplanar construction.}
    \label{fig:firstcon}
\end{figure}

\begin{proof} By~\ref{pickroot}, there is a 0/1-valued weighting $w_1$ of $G_1$ such that $w_1(y_1y_2) = w_1(y_1y_3) = 1$. Since
there is a $W_1$-path from $y_2$ to $y_1$, it follows that $w_1(y_3y_2) = 0$.
By~\ref{pickroot}, there is
a 0/1-valued weighting $w_2$ of $G_2$ such that $w_2(e)=1$ for every edge $e$ with tail $y_1$, and we have $w_2(y_2y_1) = 0$. It follows 
that $w_2(y_3y_2) = 0$ since $y_1\DD y_3\DD y_2\DD y_1$ is a dicycle.
For each $e\in E(G)$, define $w(e) = w_i(e)$ where $i\in \{1,2\}$ and $e\in E(W_i)$ ($i$ exists and is unique
since $W_1\cup W_2=G$ and $W_1,W_2$ are internally disjoint).

For every $W_2$-path $P$ from $y$ to $y'$ with $y,y'\in Y$,
it follows that $(y,y')$ is one of
$(y_1,y_3)$, $(y_3,y_2)$, $(y_1,y_2)$ since $y_1\to A_i$ and $y_3\from A_i$, and so $yy'$ is an edge of $G_1$ that we call the
\emph{$P$-substitute}. We claim that:

\begin{step}\label{step1d}
For every $W_2$-path $P$, $w_2(P) = w_1(yy')$
where $yy'$ is the $P$-substitute.
\end{step}

Let $P$ be from $y$ to $y'$. Then $(y,y')$ is one of
$(y_1,y_3), (y_3,y_2), (y_1,y_2)$.
If $(y,y')=(y_1,y_3)$, then $P$ can be completed to a dicycle of $G_2$ by adding the edges $y_3y_2$
and $y_2y_1$. Since $w_2(y_3y_2), w_2(y_2y_1) = 0$, it follows that $w_2(P) = 1=w_1(y_1y_3)$.
Similarly, if $(y,y')=(y_3,y_2)$ then $w_2(P) = 0=w_1(y_3y_2)$ (by adding $y_2y_1$ and $y_1y_3$), and if $(y,y')= (y_1,y_2)$
then $w_2(P) = 1=w_1(y_1y_2)$ (by adding $y_2y_1$). This proves~\eqref{step1d}.

\bigskip

We claim that $w$ is a weighting of $G$. To see this,
let $C$ be a dicycle of $G$. The $W_2$-paths included in $C$ are pairwise edge-disjoint and include all edges of $C$ not in $W_1$. 
By replacing each such $W_2$-path $P$ by the $P$-substitute, we obtain a dicycle $C'$  of $G_1$ such that $w_1(C') = w(C)$, and consequently $w(C)=1$.
This proves that $w$ is a weighting, and so proves~\ref{firstcon}.\end{proof}

For the second construction, it is easier to break it into two parts. First, we have:
\begin{thm}\label{secondcon1}
Let $G$ be a digraph, let $y_1,y_2,y_3,y_4$ be distinct, and let $Y=\{y_1,\dots, y_4\}$. Let $W_1,W_2$
be internally disjoint $Y$-wings with union $G$. Suppose that:
\begin{itemize}
\item $y_1$ is a source of $W_1$, and $y_2$ is a sink of $W_1$, and $y_3$ is a source of $W_2$, and $y_4$ is a sink of $W_2$;
\item there is a dipath of $W_1$ from $y_1$ to $y_2$, and there is a dipath of $W_2$ from $y_3$ to $y_4$;
\item the digraph $G_1$ obtained
from $W_1$ by adding a new vertex $v_1$ and the edges
\[ v_1y_1, y_2v_1, y_3v_1, v_1y_4, y_2y_1 \]
is 1-strong and weightable; and
\item the digraph $G_2$ 
obtained from $W_2$ by adding a new vertex $v_2$ and the edges
\[ y_1v_2, v_2y_2, v_2y_3, y_4v_2, y_4y_3 \]
is 1-strong and weightable.
\end{itemize}
See Figure~\ref{fig:secondcon1}. Then $G$ is weightable.
\end{thm}

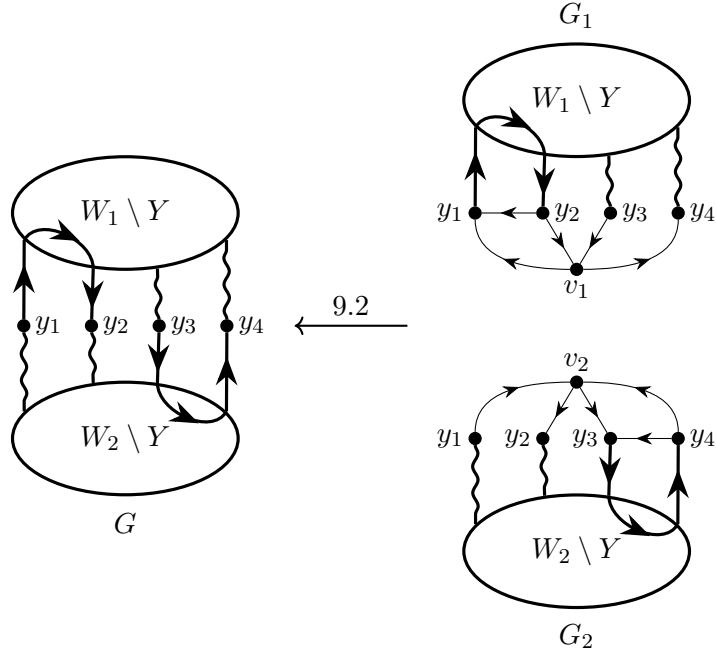
\begin{figure}[htb!]
    \centering
    \begin{tikzpicture}[scale=1.5]
        \tikzset{every label/.style={label distance=-2pt}};
        \vset{W1}{$W_1\setminus Y$}{0,2};
        \vset{W2}{$W_2\setminus Y$}{0,0};

        \vtx[right:$y_1$]{y1}{-0.9, 1};
        \vtx[right:$y_2$]{y2}{-0.3, 1};
        \vtx[right:$y_3$]{y3}{0.3, 1};
        \vtx[right:$y_4$]{y4}{0.9, 1};

        \hdedge{y1}{W1.195};
        \hdedge{W1.240}{y2};
        \sline{y3}{W1.300};
        \sline{y4}{W1.345};
        \hdedge[out=60, in=90]{W1.195}{W1.240};

        \sline{y1}{W2.165};
        \sline{y2}{W2.120};
        \hdedge{y3}{W2.60};
        \hdedge{W2.15}{y4};
        \hdedge[out=270, in=240]{W2.60}{W2.15};

        \node (G) at (0,-0.75) {$G$};

        \draw[<-, thick] (1.5,1) -- (2.5,1) node[midway, above] {\ref{secondcon1}};

        \vset{W1}{$W_1\setminus Y$}{4,3};

        \vtx[left:$y_1$]{y1}{3.1, 2};
        \vtx[right:$y_2$]{y2}{3.7, 2};
        \vtx[right:$y_3$]{y3}{4.3, 2};
        \vtx[right:$y_4$]{y4}{4.9, 2};
        \vtx[below:$v_1$]{v1}{4, 1.5};

        \hdedge{y1}{W1.195};
        \hdedge{W1.240}{y2};
        \sline{y3}{W1.300};
        \sline{y4}{W1.345};
        \hdedge[out=60, in=90]{W1.195}{W1.240};
        \dedge{y2}{y1};
        \dedge[out=180, in=270]{v1}{y1};
        \dedge{y2}{v1};
        \dedge{y3}{v1};
        \dedge[out=0, in=270]{v1}{y4};
        \node (G1) at (4, 3.75) {$G_1$};

        \vset{W2}{$W_2\setminus Y$}{4,-1};

        \vtx[left:$y_1$]{y1}{3.1, 0};
        \vtx[left:$y_2$]{y2}{3.7, 0};
        \vtx[left:$y_3$]{y3}{4.3, 0};
        \vtx[right:$y_4$]{y4}{4.9, 0};
        \vtx[above:$v_2$]{v2}{4, 0.5};

        \sline{y1}{W2.165};
        \sline{y2}{W2.120};
        \hdedge{y3}{W2.60};
        \hdedge{W2.15}{y4};
        \hdedge[out=270, in=240]{W2.60}{W2.15};
        \dedge{y4}{y3};
        \dedge[out=90, in=180]{y1}{v2};
        \dedge{v2}{y2};
        \dedge{v2}{y3};
        \dedge[out=90, in=0]{y4}{v2};
        \node (G2) at (4, -1.75) {$G_2$};
    \end{tikzpicture}
    \caption{The first constituent of the second nonplanar construction.}
    \label{fig:secondcon1}
\end{figure}

\begin{proof} 
By~\ref{pickroot}, there is a 0/1-valued weighting $w_1'$ of $G_1$ such that $w_1'(e)=1$ for $e\in \{v_1y_1,v_1y_4\}$ and
$w_1'(e) = 0$ for $e\in \{y_2v_1,y_3v_1\}$. Since there is a dipath of $W_1$ from $y_1$ to $y_2$, which can be completed to a dicycle of $G_1$ via $y_2\DD y_1$ or via $y_2\DD v_1\DD y_1$, it follows that $w_1'(y_2y_1) = 1$. Since $G_1$ is 1-strong, every edge $e$ of $G_1$ with tail $y_1$ belongs to a dicycle which therefore contains one of $v_1y_1,y_2y_1$, and so $w_1'(e) = 0$.
Consequently, since all in-edges of $y_1$ have weight one, and all out-edges of $y_1$ have weight zero, there is a 0/1-valued
weighting $w_1$ of $G_1$ such that $w_1(v_1y_4)=1$, $w_1(e)=0$ for each of the other three edges incident with $v_1$, $w_1(y_2y_1) = 0$,
and $w_1(e)=1$ for every edge $e$ of $G_1$ with tail $y_1$. Similarly, there is a 0/1-valued
weighting $w_2$ of $G_2$ such that $w_2(y_1v_2)=1$, $w_2(e)=0$ for each of the other three edges incident with $v_2$, $w_2(y_4y_3) = 0$,
and $w_2(e)=1$ for every edge $e$ of $G_2$ with head $y_4$. For each edge $e$ of $G$, let $w(e) = w_i(e)$ where $i\in \{1,2\}$ satisfies
$e\in E(W_i)$. We claim that $w$ is a weighting of $G$. To see this, let $C$ be a dicycle of $G$. We may assume that 
$C$ is not a dicycle of $G_1$ or of $G_2$. Thus either $C\cap W_1$ and $C\cap W_2$ are both dipaths, or
$C\cap W_1$ is the disjoint union of two dipaths and so is $C\cap W_2$. 

Suppose first that $C\cap W_1$ is a dipath $P$ from $y$ to $y'$ with $y,y'\in Y$. Thus $y,y'$ are distinct, and
$C\cap W_2$ is a dipath $Q$ from $y'$ to $y$. Since $y_1$ is a source of $W_1$, it follows that $y'\ne y_1$, and similarly $y\ne y_2$, $y'\ne y_3$ and $y\ne y_4$. So, of the twelve pairs $(y,y')$ where $y,y'\in Y$ are distinct, only four possibilities remain:
the pairs $(y_1,y_3),(y_1,y_2),(y_4,y_3),(y_4,y_2)$. If $(y,y') = (y_1,y_3)$, then $w_1(P) = 1$, since $P$ can be completed to a dicycle of $G_1$ via $y_3\DD v_1\DD y_1$, and $w_2(Q)=0$, since $Q$ can be completed to a dicycle of $G_2$ via
$y_1\DD v_2\DD y_3$, and so $w(C)=1$. The other three cases are can all be handled similarly, and we omit the details. (See the table below.)
\begin{center}
\begin{tabular}{c|c c c}
    $(y,y')$ & $w_1(P)$ & $w_2(Q)$ & $w(C)$ \\
    \hline
    $(y_1,y_3)$ & 1 & 0 & 1 \\
    $(y_1,y_2)$ & 1 & 0 & 1 \\
    $(y_4,y_3)$ & 0 & 1 & 1 \\
    $(y_4,y_2)$ & 0 & 1 & 1
\end{tabular}
\end{center}

Now we assume (for a contradiction) that $C\cap W_1$ is the disjoint union of two dipaths $P_1,P_2$, and so $C\cap W_2$ is also the union of two dipaths $Q_1,Q_2$. Since $y_1$ is a source of $W_1$, neither of $P_1,P_2$ has last vertex $y_1$, and so neither of $Q_1,Q_2$
has first vertex $y_1$. By the same arguments for $y_2,y_3,y_4$, it follows that $P_1,P_3$ are both from $\{y_1,y_4\}$ to $\{y_2,y_3\}$,
and therefore $Q_1,Q_2$ are both from $\{y_2,y_3\}$ to $\{y_1,y_4\}$. Suppose first that $P_1$ is from $y_1$ to $y_3$, and so 
$P_2$ is from $y_4$ to $y_2$. 
We recall that there is a dipath $R$ in $W_1$ from $y_1$ to $y_2$. But then 
\[v_1\DD y_1\DD R\DD y_2\DD v_1,\]
\[v_1\DD y_4\DD P_2\DD y_2\DD y_1\DD P_1\DD y_3\DD v_1\]
are dicycles of $G_1$ that disagree on $\{v_1,y_1,y_2\}$, a contradiction. So neither of $P_1,P_2$ is from $y_1$ to $y_3$;
and similarly, neither of $Q_1,Q_2$ is from $y_3$ to $y_1$, which is impossible. Thus there is no such cycle $C$, and consequently $w$ is
a weighting. This proves~\ref{secondcon1}.\end{proof}

We also need:
\begin{thm}\label{secondcon2}
Let $G$ be a digraph, let $y_1,y_2,y_3,y_4$ be distinct, and let $Y=\{y_1,\dots, y_4\}$. 
Suppose that:
\begin{itemize}
\item $y_1, y_3$ are sources of $G$, and $y_2, y_4$ are sinks of $G$; and
\item the digraph $G_0$ obtained from $G$ by adding the edges $y_1y_3,y_3y_1$ and making the identifications $y_1=y_2$ and $y_3=y_4$
is 1-strong and weightable.
\end{itemize}
Let $H$ be obtained from $G$ by adding two new vertices $v_1,v_2$ and the edges 
\[y_2y_1,y_4y_3, v_1y_1,y_2v_1,y_3v_1,v_1y_4,y_1v_2,v_2y_2,v_2y_3,y_4v_2\]
(see Figure~\ref{fig:secondcon2}) is weightable.
\end{thm}

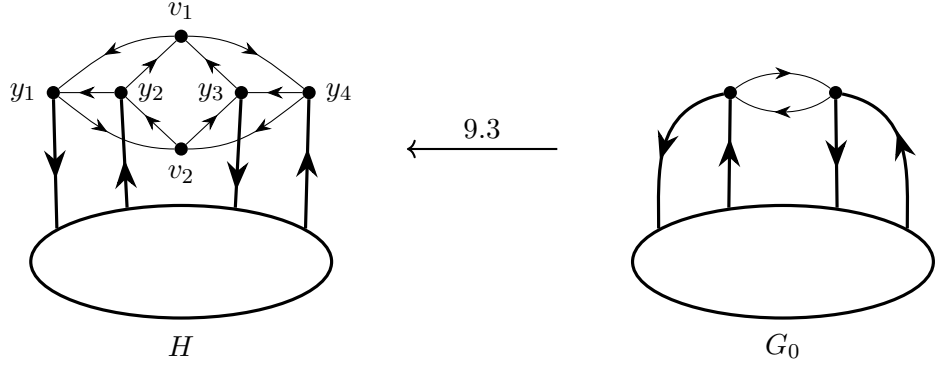
\begin{figure}[htb!]
    \centering
    \begin{tikzpicture}[yscale=1.5,xscale=2]
        \vset{V}{}{0,0};
        \vtx[left:$y_1$]{y1}{-0.85, 1.5};
        \vtx[right:$y_2$]{y2}{-0.4, 1.5};
        \vtx[left:$y_3$]{y3}{0.4, 1.5};
        \vtx[right:$y_4$]{y4}{0.85, 1.5};
        \vtx[above:$v_1$]{v1}{0, 2};
        \vtx[below:$v_2$]{v2}{0, 1};

        \hdedge{y1}{V.165};
        \draw[very thick, midarrow=0.3] (V.135) -- (y2);
        \draw[very thick, midarrow=0.7] (y3) -- (V.45);
        \hdedge{V.15}{y4};

        \dedge{y2}{y1};
        \dedge{y4}{y3};
        \dedge[out=180, in=45]{v1}{y1};
        \dedge{y2}{v1};
        \dedge{y3}{v1};
        \dedge[out=0, in=135]{v1}{y4};
        
        \draw[midarrow=0.4] (y1) to[out=315, in=180] (v2);
        \dedge{v2}{y2};
        \dedge{v2}{y3};
        \draw[midarrow=0.4] (y4) to[out=225, in=0] (v2);
        \node (H) at (0,-0.75) {$H$};

        \draw[<-, thick] (1.5, 1) -- (2.5, 1) node[midway, above] {\ref{secondcon2}};

        \vset{V}{}{4,0};

        \vtx{w1}{3.65,1.5};
        \vtx{w2}{4.35,1.5};

        \hdedge[out=195, in=90]{w1}{V.165};
        \hdedge{V.135}{w1};
        \hdedge{w2}{V.45};
        \hdedge[out=90, in=345]{V.15}{w2};

        \dedge[out=45, in=135]{w1}{w2};
        \dedge[out=225, in=315]{w2}{w1};

        \node (G0) at (4,-0.75) {$G_0$};
    \end{tikzpicture}
    \caption{The second constituent of the second nonplanar construction.}
    \label{fig:secondcon2}
\end{figure}

\begin{proof}
Let $G'$ be obtained from $G$ by adding two new vertices $v_1,v_2$ and the edges 
\[y_1v_2, v_2y_3, v_1y_1,y_3v_1, y_2v_1,y_4v_2.\]
Let $F=\{y_4y_3,v_1y_4,v_2y_2,y_2y_1\}$; thus, $G'$ is obtained from $H$ by deleting the edges in $F$.
The edges $v_1y_1, v_2y_3, y_2v_1,y_4v_2$ are all singular edges of $G'$ (because $y_2, y_4$ are sinks of $G$), and $G_0$
is obtained by contracting these edges. Consequently, $G'$ is 1-strong, and weightable 
by~\ref{butterfly}. By~\ref{pickroot} we may choose a 0/1-valued weighting $w'$ of $G'$ such that $w'(e) = 1$ for 
all edges $e$ with tail $y_1$ and $w'(v_1y_1)=0$. Since $y_1\DD v_2\DD y_3\DD v_1\DD y_1$ is a dicycle of $G'$,
it follows that $w'(e)=0$ for $e\in \{v_2y_3,y_3v_1,v_1y_1\}$. Since $G_0$ is 1-strong,
and $y_3$ is a source and $y_2$ a sink of $G$, there
is a dipath $R$ of $G$ from $y_1$ to $y_4$. Consequently, $w'(y_4v_2) = 0$ (because $R$ can be extended to a dicycle via 
$y_4\DD v_2\DD y_3\DD v_1\DD y_1$). Similarly, $w'(y_2v_1) = 0$. For each edge $e$ of $H$, define $w(e) = w'(e)$ if $e\in E(G')$,
and $w(y_4y_3) = 0$, $w(v_1y_4) = 1$, $w(v_2y_2) = 0$ and $w(y_2y_1) = 0$. We claim that $w$ is a weighting of $H$. 
To show this, suppose that there is a dicycle $C$ of $H$ with $w(C)\ne 1$, and choose it with $E(C)\cap F$ minimal.
Since $w'$ is a weighting of $G'$, it follows that $E(C)\cap F\ne \emptyset$. 

Suppose that $y_4y_3\in E(C)$. If $v_2\notin V(C)$, then we can replace the edge $y_4y_3$ of $C$ by the path $y_4\DD y_2\DD y_3$
and obtain another dicycle $C'$ say, with $w(C') = w(C)\ne 1$, contrary to the minimality of $E(C)\cap F$. Thus,
$v_2\in V(C)$. Hence $C$ is the concatenation of three dipaths: from $y_4$ to $y_3$ (of length one); a dipath $P$, say, from 
$y_3$ to $v_2$; and a dipath $Q$, say, from $v_2$ to $y_4$. Since $y_3\notin V(Q)$, it follows that $y_2$ is the second vertex of $Q$,
and since $y_2, y_4\notin V(P)$, it follows that $y_1$ is the penultimate vertex of $P$. Hence $y_1\notin V(Q)$, and so $v_1$ is the third vertex of $Q$, but then both in-neighbours of $y_1$ belong to $Q$, and yet one of them belongs to $P$, a contradiction. Thus
$y_4y_3\notin E(C)$, and similarly $y_2y_1\notin E(C)$. 

Suppose that $v_1y_4\in E(C)$. Then $y_4v_2\in E(C)$ (because $y_4y_3\notin E(C)$), and since we cannot replace the subpath
$v_1\DD y_4\DD v_2$ in $C$ by $v_1\DD y_1\DD v_2$ (because these two paths have the same total weight and the latter has smaller intersection with $F$) it follows that $y_1\in V(C)$.
But there is no dipath in $H$ from $v_2$ to $v_1$ passing through $y_1$ and not containing $y_4$, a contradiction. 
So $v_1y_4\notin E(C)$, and similarly $v_2y_2\notin E(C)$. This proves that $w$ is a weighting of $H$, and this proves~\ref{secondcon2}.\end{proof}

By combining~\ref{secondcon1} and~\ref{secondcon2}, we obtain what we really wanted:

\begin{thm}\label{secondcon}
Let $y_1,\dots, y_4$ be distinct vertices of a digraph $G$, and let $Y=\{y_1,\dots, y_4\}$. Let $W_1,W_2,W_3$ be pairwise internally disjoint $Y$-wings with union $G$. Suppose that:
\begin{itemize}
\item $y_1$ is a source of $W_1\cup W_3$, and $y_2$ is a sink of $W_1\cup W_3$, and $y_3$ is a source of $W_2\cup W_3$, and $y_4$ is a sink of $W_2\cup W_3$;
\item there is a dipath of $W_1$ from $y_1$ to $y_2$, and there is a dipath of $W_2$ from $y_3$ to $y_4$;
\item the digraph $G_1$ obtained
from $W_1$ by adding a new vertex $v_1$ and the edges
\[ v_1y_1, y_2v_1, y_3v_1, v_1y_4, y_2y_1 \]
is 1-strong and weightable; 
\item the digraph $G_2$
obtained from $W_2$ by adding a new vertex $v_2$ and the edges
\[ y_1v_2, v_2y_2, v_2y_3, y_4v_2, y_4y_3 \]
is 1-strong and weightable; and
\item the digraph $G_0$ obtained from $W_3$ by adding the edges $y_1y_3,y_3y_1$ and making the identifications $y_1=y_2$ and $y_3=y_4$
is 1-strong and weightable.
\end{itemize}
See Figure~\ref{fig:secondcon}. Then $G$ is weightable.
\end{thm}

\begin{figure}[htb!]
    \centering
    \begin{tikzpicture}[scale=1.4]
        \tikzset{every label/.style={label distance=-2pt}}
        \vset{W1}{$W_1\setminus Y$}{-1.5,2.5};
        \vset{W2}{$W_2\setminus Y$}{1.5,2.5};
        \vset{W3}{$W_3\setminus Y$}{0,-0.5};

        \vtx[{[yshift=-3pt]right:$y_1$}]{y1}{-0.9, 0.9};
        \vtx[{[yshift=-3pt]right:$y_2$}]{y2}{-0.3, 0.9};
        \vtx[{[yshift=-3pt]right:$y_3$}]{y3}{0.3, 0.9};
        \vtx[{[yshift=-3pt]right:$y_4$}]{y4}{0.9, 0.9};

        \hdedge{y1}{W1.215};
        \hdedge{W1.270}{y2};
        \sline{y3}{W1.315};
        \sline{y4}{W1.345};
        \hdedge[out=60, in=120]{W1.215}{W1.270};

        \sline{y1}{W2.195};
        \sline{y2}{W2.225};
        \hdedge{y3}{W2.270};
        \hdedge{W2.325}{y4};
        \hdedge[out=60, in=120]{W2.270}{W2.325};

        \hdedge{y1}{W3.165};
        \hdedge{W3.120}{y2};
        \hdedge{y3}{W3.60};
        \hdedge{W3.15}{y4};

        \node (G) at (0,-1.25) {$G$};

        \draw[<-, thick] (2.25,1.1) -- (3.25,1.1) node[midway, above] {\ref{secondcon}};
        \draw[->, thick, dashed] (2.25,0.9) -- (3.25,0.9) node[midway, below] {\ref{secondoutcome}};

        \vset{W1}{$W_1\setminus Y$}{4.5,2.75};
        \vset{W2}{$W_2\setminus Y$}{7.5,2.75};
        \vset{W3}{$W_3\setminus Y$}{6,-0.75};

        \vtx[left:$y_1$]{y1}{3.6, 1.75};
        \vtx[right:$y_2$]{y2}{4.2, 1.75};
        \vtx[right:$y_3$]{y3}{4.8, 1.75};
        \vtx[right:$y_4$]{y4}{5.4, 1.75};
        \vtx[below:$v_1$]{v1}{4.5, 1.25};

        \hdedge{y1}{W1.195};
        \hdedge{W1.240}{y2};
        \sline{y3}{W1.300};
        \sline{y4}{W1.345};
        \hdedge[out=60, in=90]{W1.195}{W1.240};
        \dedge{y2}{y1};
        \dedge[out=180, in=270]{v1}{y1};
        \dedge{y2}{v1};
        \dedge{y3}{v1};
        \dedge[out=0, in=270]{v1}{y4};
        \node (G1) at (4.5, 3.5) {$G_1$};

        \vtx[left:$y_1$]{y1}{6.6, 1.75};
        \vtx[left:$y_2$]{y2}{7.2, 1.75};
        \vtx[left:$y_3$]{y3}{7.8, 1.75};
        \vtx[right:$y_4$]{y4}{8.4, 1.75};
        \vtx[below:$v_2$]{v2}{7.5, 1.25};

        \sline{y1}{W2.195};
        \sline{y2}{W2.240};
        \hdedge{y3}{W2.300};
        \hdedge{W2.345}{y4};
        \hdedge[out=90, in=120]{W2.300}{W2.345};
        \dedge{y4}{y3};
        \dedge[out=270, in=180]{y1}{v2};
        \dedge{v2}{y2};
        \dedge{v2}{y3};
        \dedge[out=270, in=0]{y4}{v2};
        \node (G2) at (7.5, 3.5) {$G_2$};

        \vtx{w1}{5.65,0.5};
        \vtx{w2}{6.35,0.5};

        \hdedge[out=195, in=90]{w1}{W3.165};
        \hdedge{W3.125}{w1};
        \hdedge{w2}{W3.55};
        \hdedge[out=90, in=345]{W3.15}{w2};

        \dedge[out=45, in=135]{w1}{w2};
        \dedge[out=225, in=315]{w2}{w1};

        \node (G0) at (6,-1.5) {$G_0$};
    \end{tikzpicture}
    \caption{The second nonplanar construction.}
    \label{fig:secondcon}
\end{figure}
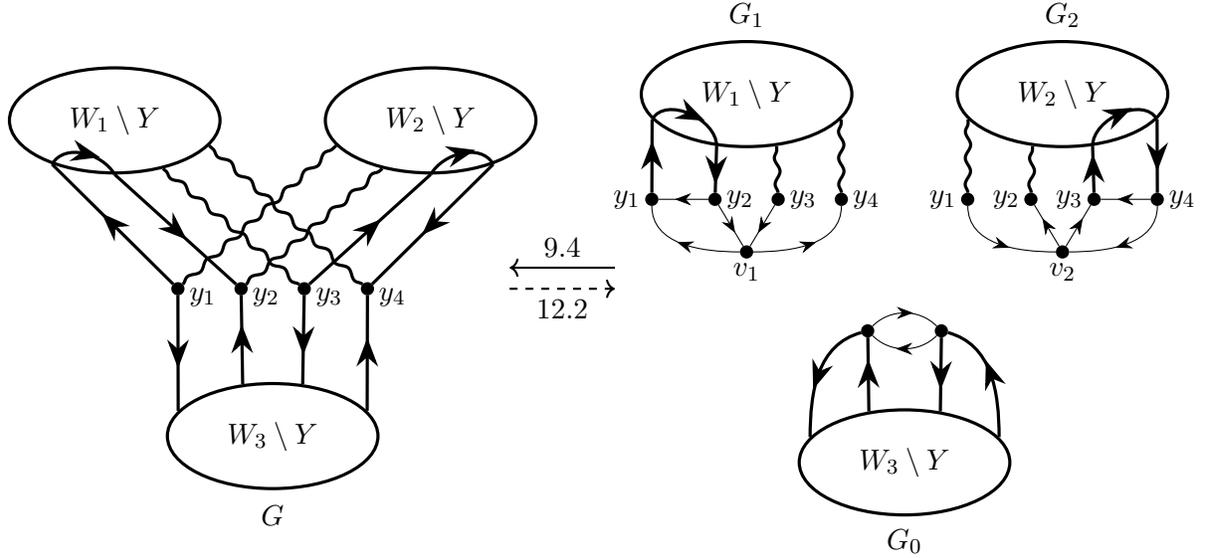

\begin{proof}
Let $W_3^+$ be obtained from $W_3$ as in~\ref{secondcon2}; that is, by adding two new vertices $v_1,v_2$ and the edges
\[y_1v_2, v_2y_3, v_1y_1,y_3v_1, y_2v_1,y_4v_2.\]
Then $W_3^+$ is 1-strong and weightable, by~\ref{secondcon2}. By~\ref{secondcon1} applied to $W_1$ and $W_3^+$, we deduce that 
the digraph obtained from $W_1\cup W_3$ by adding a new vertex $v_1$ and the edges $v_1y_1,y_2v_1, y_3v_1,v_1y_4, y_2y_1$ is 
strong and weightable. By~\ref{secondcon1}, applied to $W_1\cup W_3$  and $W_2$, it follows that $G$ is weightable. This proves~\ref{secondcon}.\end{proof}

Finally, we need:
\begin{thm}\label{thirdcon}
Let $y_1,\dots, y_4$ be distinct vertices of a digraph $G$, and let $Y=\{y_1,\dots, y_4\}$. Let $W_1,W_2$ be internally disjoint
$Y$-wings with union $G$. Suppose that:
\begin{itemize}
\item $y_1,y_3$ are sources of $W_1$ and $y_2,y_4$ are sinks of $W_1$;
\item there are dipaths of $W_1$ from $y_1$ to $y_4$ and from $y_3$ to $y_2$, and 
there are dipaths of $W_2$ from $y_2$ to $y_1$ and from $y_4$ to $y_3$;
\item the digraph $G_1$  obtained from $W_1$ by adding the edges $y_2y_1,y_4y_3,y_1y_3,y_3y_1$ is 1-strong and weightable 
(equivalently, by~\ref{butterfly}, the digraph obtained from $G_1$ by contracting $y_2y_1,y_4y_3$ is 1-strong and weightable);
\item the digraph $G_2$ obtained from $W_2$ by adding two new vertices $v_1,v_2$ and the edges
\[ y_1v_1,y_3v_2,v_1y_4,v_2y_2,y_3y_4,v_1v_2,v_2v_1 \]
is 1-strong and weightable.
\end{itemize}
See Figure~\ref{fig:thirdcon}. Then $G$ is weightable.
\end{thm}

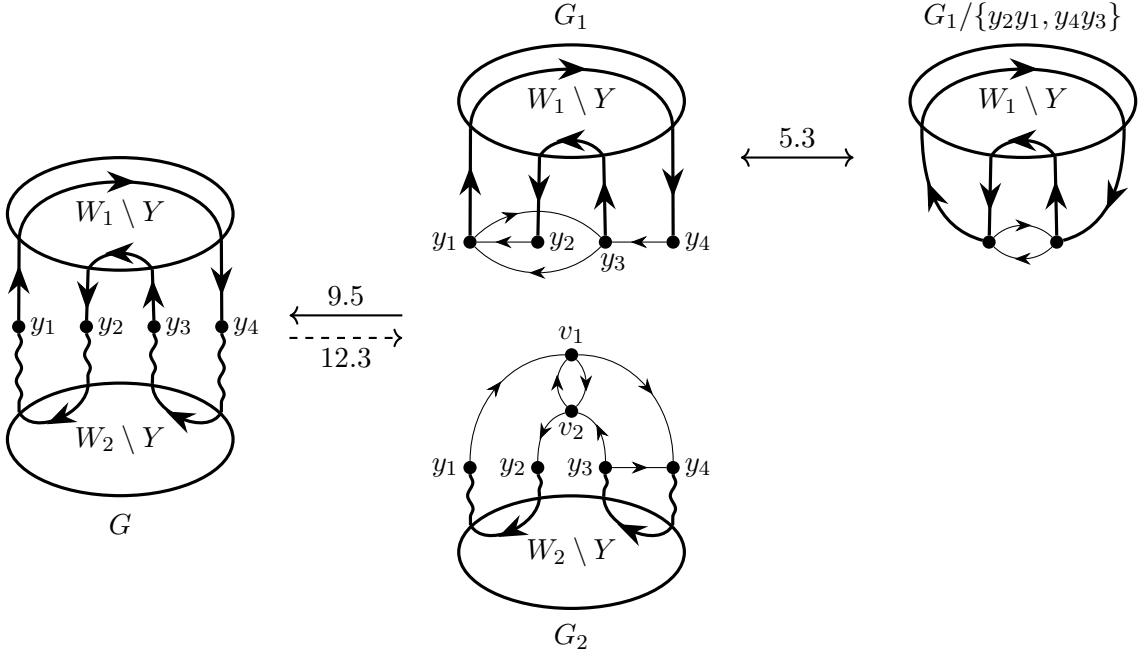
\begin{figure}[htb!]
    \centering
    \begin{tikzpicture}[scale=1.5]
        \tikzset{every label/.style={label distance=-2pt}};
        \vset{W1}{$W_1\setminus Y$}{0,2};
        \vset{W2}{$W_2\setminus Y$}{0,0};

        \vtx[right:$y_1$]{y1}{-0.9, 1};
        \vtx[right:$y_2$]{y2}{-0.3, 1};
        \vtx[right:$y_3$]{y3}{0.3, 1};
        \vtx[right:$y_4$]{y4}{0.9, 1};

        \hdedge{y1}{W1.195};
        \hdedge{W1.240}{y2};
        \hdedge{y3}{W1.300};
        \hdedge{W1.345}{y4};
        \hdedge[out=90, in=90]{W1.195}{W1.345};
        \hdedge[out=120, in=60]{W1.300}{W1.240};

        \sline{y1}{W2.165};
        \sline{y2}{W2.120};
        \sline{y3}{W2.60};
        \sline{y4}{W2.15};
        \hdedge[out=270, in=300]{W2.120}{W2.165};
        \hdedge[out=240, in=270]{W2.15}{W2.60};

        \node (G) at (0,-0.75) {$G$};

        \draw[<-, thick] (1.5,1.1) -- (2.5,1.1) node[midway, above] {\ref{thirdcon}};
        \draw[->, thick, dashed] (1.5,0.9) -- (2.5,0.9) node[midway, below] {\ref{thirdoutcome}};

        \vset{W1}{$W_1\setminus Y$}{4,3};
        
        \vtx[left:$y_1$]{y1}{3.1, 1.75};
        \vtx[right:$y_2$]{y2}{3.7, 1.75};
        \vtx[{[xshift=3pt]below:$y_3$}]{y3}{4.3, 1.75};
        \vtx[right:$y_4$]{y4}{4.9, 1.75};
        \hdedge{y1}{W1.195};
        \draw[very thick, midarrow=0.4] (W1.240) -- (y2);
        \hdedge{y3}{W1.300};
        \hdedge{W1.345}{y4};
        \hdedge[out=90, in=90]{W1.195}{W1.345};
        \hdedge[out=120, in=60]{W1.300}{W1.240};
        \dedge{y2}{y1};
        \dedge{y4}{y3};
        \draw[midarrow=0.3] (y1) to[out=45, in=135] (y3);
        \dedge[out=225, in=315]{y3}{y1};
        \node (G1) at (4, 3.75) {$G_1$};

        \vset{W2}{$W_2\setminus Y$}{4,-1};
        
        \vtx[left:$y_1$]{y1}{3.1, -0.25};
        \vtx[left:$y_2$]{y2}{3.7, -0.25};
        \vtx[left:$y_3$]{y3}{4.3, -0.25};
        \vtx[right:$y_4$]{y4}{4.9, -0.25};
        \vtx[above:$v_1$]{v1}{4, 0.75};
        \vtx[below:$v_2$]{v2}{4, 0.25};
        
        \sline{y1}{W2.165};
        \sline{y2}{W2.120};
        \sline{y3}{W2.60};
        \sline{y4}{W2.15};
        \hdedge[out=270, in=300]{W2.120}{W2.165};
        \hdedge[out=240, in=270]{W2.15}{W2.60};

        \dedge[out=90, in=180]{y1}{v1};
        \dedge[out=90, in=345]{y3}{v2};
        \dedge[out=0, in=90]{v1}{y4};
        \dedge[out=195, in=90]{v2}{y2};
        \dedge{y3}{y4};
        \dedge[out=315, in=45]{v1}{v2};
        \dedge[out=135, in=225]{v2}{v1};
        \node (G2) at (4, -1.75) {$G_2$};

        \draw[<->, thick] (5.5,2.5) -- (6.5,2.5) node[midway, above] {\ref{butterfly}};

        \vset{W1}{$W_1\setminus Y$}{8,3};

        \vtx{w1}{7.7, 1.75};
        \vtx{w2}{8.3, 1.75};
        
        \hdedge[out=165, in=270]{w1}{W1.195};
        \hdedge{W1.240}{w1};
        \hdedge{w2}{W1.300};
        \hdedge[out=270, in=15]{W1.345}{w2};
        \dedge[out=45, in=135]{w1}{w2};
        \dedge[out=225, in=315]{w2}{w1};
        \hdedge[out=90, in=90]{W1.195}{W1.345};
        \hdedge[out=120, in=60]{W1.300}{W1.240};
        \node (G1) at (8, 3.75) {$G_1/\{y_2y_1,y_4y_3\}$};
    \end{tikzpicture}
    \caption{The third nonplanar construction. Note that the pair of paths depicted in $W_1\setminus Y$ need not be disjoint, and likewise for $W_2\setminus Y$.}
    \label{fig:thirdcon}
\end{figure}
\begin{proof}
By~\ref{pickroot}, there is a 0/1-valued weighting $w_1$ of $G_1$ such that $w_1(e) = 1$ for the edges of $G_1$ with tail $y_1$
and $w_1(e) = 0$ for those with head $y_1$. Since there are dipaths of $W_1$ from $y_1$ to $y_4$ and from $y_3$ to $y_2$, it follows 
that $w_1(y_4y_3) = w_1(y_2y_1) = 0$. By~\ref{pickroot}, there is a 0/1-valued weighting $w_2$ of $G_2$ such that 
$w_2(e) = 1$ for the edges with head $v_1$ and $w_2(e) = 0$ for the edges with tail $v_1$. Since 
there are dipaths of $W_2$ from $y_2$ to $y_1$ and from $y_4$ to $y_3$, it follows that $w_2(v_2y_2) = w_2(y_3y_2) = 0$. Let $P$ be a path in $W_2$ from $y_4$ to $y_3$; we have $w_2(P)=0$ since $P$ completes to a cycle in $G_2$ via $y_3\DD v_2\DD v_1\DD y_4$. Therefore $w_2(y_3y_4)=1$, since this edge completes to a cycle via $P$.

For each edge $e\in E(G)$, define $w(e) = w_i(e)$ where $e\in E(G_i)$; note that $w_1$ and $w_2$ agree on any edges in common between $G_1,G_2$ (namely edges with both ends in $Y$). We claim that $w$ is a weighting of $G$. To show this, let $C$
be a dicycle of $G$. We may assume that $C$ is not a cycle of $W_1$ or of $W_2$, so either $C\cap W_1$,  $C\cap W_2$ are both 
paths, or $C\cap W_1$,  $C\cap W_2$ are both the disjoint union of two paths. Suppose first that $C\cap W_1$ is a path $P$ from $y$ 
to $y'$, where $y,y'\in Y$ are distinct, so $C\cap W_2$ is a path $Q$ from $y'$ to $y$. Since $y_1,y_3$ are sources of $W_1$, and $y_2,y_4$ are sinks of $W_1$, it follows that
$(y,y')$ is one of the pairs $(y_1,y_2), (y_1,y_4), (y_3,y_2), (y_3,y_4)$. If $(y,y') = (y_1,y_2)$, then $w_1(P)=0$
(because $P$ can be completed to a dicycle of $G_1$ via $y_3\DD y_2$), and $w_2(Q) = 0$ (because 
$Q$ can be completed via $y_1\DD v_1\DD v_2\DD y_2$), and so $w(C)=1$. The argument is similar in the other three cases and we omit the details. (See the table below.)
\begin{center}
\begin{tabular}{c|c c c}
    $(y,y')$ & $w_1(P)$ & $w_2(Q)$ & $w(C)$ \\
    \hline
    $(y_1,y_2)$ & 1 & 0 & 1 \\
    $(y_1,y_4)$ & 1 & 0 & 1 \\
    $(y_3,y_2)$ & 0 & 1 & 1 \\
    $(y_3,y_4)$ & 1 & 0 & 1
\end{tabular}
\end{center}

Now we assume that $C\cap W_1$,  $C\cap W_2$ are both the disjoint union of two paths. Since $y_1,y_3$ are sources of $W_1$, and $y_2,y_4$ are sinks of $W_1$, the two components of $C\cap W_1$ are both from $\{y_1,y_3\}$ to $\{y_2,y_4\}$, so there are two 
possibilities:
\begin{itemize}
\item $C\cap W_1$ is the disjoint union of $P_1$ from $y_1$ to $y_4$ and $Q_1$ from $y_3$ to $y_2$, and $C\cap W_2$ is the disjoint
union of $P_2$ from $y_2$ to $y_1$ and $Q_2$ from $y_4$ to $y_3$; or
\item $C\cap W_1$ is the disjoint union of $P_1$ from $y_1$ to $y_2$ and $Q_1$ from $y_3$ to $y_4$, and $C\cap W_2$ is the disjoint 
union of $P_2$ from $y_2$ to $y_3$ and $Q_2$ from $y_4$ to $y_1$.
\end{itemize}
In the first case, $w_1(P_1) = 1$ (by completing via $y_4\DD y_3\DD y_1$), and similarly $w_1(Q_1) = 0$, $w_2(P_2) = 0$, and $w_2(Q_2) = 0$, so $w(C) = 1$ as required. In the second case, since $P_2,Q_2$ are disjoint, it follows that 
\[y_2\DD P_2\DD y_3\DD y_4\DD Q_2\DD y_1\DD v_1\DD v_2\DD y_2\]
is a dicycle of $G_2$. But $w_2(P_2) = 1$ (by completing via $y_3\DD v_2\DD y_2$) and 
$w_2(y_1v_1) = 1$, so this second case cannot occur. This proves that $w$ is a weighting, and so proves~\ref{thirdcon}.\end{proof}

\section{Decompositions in the nonplanar case}

If a 2-strong weightable graph is nonplanar, we can deduce from the main theorem of~\cite{pfaffians} that it has a 2-, 3-, or 4-vertex cutset $X$ such that its deletion from $G$ makes at least three weak components, and this will allow us to construct $G$ from 
smaller weightable 1-strong digraphs by the constructions of the previous section.

\begin{figure}[htb!]
\centering
\begin{tikzpicture}[scale=0.75]
\vtx{a1}{-5,1};
\vtx{a2}{-3,1};
\vtx{a3}{-5,-1};
\vtx{a4}{-3,-1};

\vtx{b1}{-1,1};
\vtx{b2}{1,1};
\vtx{b3}{-1,-1};
\vtx{b4}{1,-1};

\vtx{c1}{3,1};
\vtx{c2}{5,1};
\vtx{c3}{3,-1};
\vtx{c4}{5,-1};

\vtx{d1}{-2,2.5};
\vtx{d2}{2,2.5};
\vtx{d3}{-2,-2.5};
\vtx{d4}{2,-2.5};

\foreach \from/\to in {a1/a2,a2/a4,a4/a3,a3/a1, b1/b2,b2/b4,b4/b3,b3/b1, c1/c2,c2/c4,c4/c3,c3/c1, a1/d1,a2/d2,a3/d3,a4/d4, b1/d1,b2/d2,b3/d3,b4/d4, c1/d1,c2/d2,c3/d3,c4/d4}
\draw (\from) -- (\to);

\end{tikzpicture}

\caption{Rotunda.} \label{fig:rotunda}
\end{figure}

The graph ``Rotunda'' is shown in Figure~\ref{fig:rotunda}. Note that Rotunda is bipartite.
There are four vertices of Rotunda (two at the top and two at the bottom of the figure) such that deleting them 
makes a graph with three
components, each a cycle of length four. We call the set of these four vertices the \emph{join} of Rotunda. 
By an \emph{odd subdivision} of Rotunda, we mean a graph $R$ obtained from a copy of Rotunda by replacing each edge by a path of odd length, all
internally disjoint. Such a graph $R$ is bipartite, and the set of four vertices of $R$ that corresponds to the join of Rotunda is called the
\emph{join} of $R$. (It is uniquely defined by $R$.) 

The \emph{Heawood graph} is shown in Figure~\ref{fig:heawood}. For every perfect matching $M$ of the Heawood graph $H$, the collapse of $(H,M)$ is isomorphic to the digraph $F_7$ in Figure~\ref{fig:heawood}. 

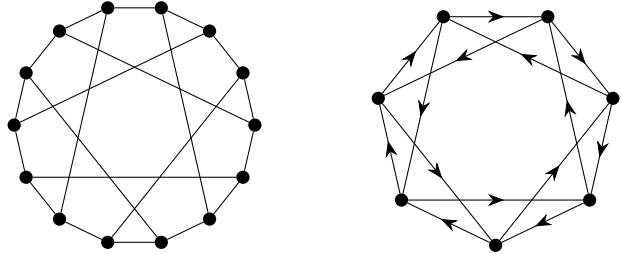
\begin{figure}[htb!]
\centering
\begin{tikzpicture}[auto=left, scale = .8]

\def\r{2}
\def\s{360/14}
\vtx{0}{{\r}, 0};
\vtx{1}{{\r*cos(\s)}, {\r*sin(\s)}};
\vtx{2}{{\r*cos(2*\s)}, {\r*sin(2*\s)}};
\vtx{3}{{\r*cos(3*\s)}, {\r*sin(3*\s)}};
\vtx{4}{{\r*cos(4*\s)}, {\r*sin(4*\s)}};
\vtx{5}{{\r*cos(5*\s)}, {\r*sin(5*\s)}};
\vtx{6}{{\r*cos(6*\s)}, {\r*sin(6*\s)}};
\vtx{7}{{\r*cos(7*\s)}, {\r*sin(7*\s)}};
\vtx{8}{{\r*cos(8*\s)}, {\r*sin(8*\s)}};
\vtx{9}{{\r*cos(9*\s)}, {\r*sin(9*\s)}};
\vtx{10}{{\r*cos(10*\s)}, {\r*sin(10*\s)}};
\vtx{11}{{\r*cos(11*\s)}, {\r*sin(11*\s)}};
\vtx{12}{{\r*cos(12*\s)}, {\r*sin(12*\s)}};
\vtx{13}{{\r*cos(13*\s)}, {\r*sin(13*\s)}};

\foreach \from/\to in {0/1,1/2,2/3,3/4,4/5,5/6,6/7,7/8,8/9,9/10,10/11,11/12,12/13,13/0,0/5,2/7,4/9,6/11,8/13,10/1,12/3}
\draw (\from)--(\to);
\begin{scope}[shift ={(6,0)}]
\def\r{2}
\def\s{360/7}
\vtx{1}{{\r*cos(90+0.5*\s)}, {\r*sin(90+0.5*\s)}};
\vtx{2}{{\r*cos(90+1.5*\s)}, {\r*sin(90+1.5*\s)}};
\vtx{3}{{\r*cos(90+2.5*\s)}, {\r*sin(90+2.5*\s)}};
\vtx{4}{{\r*cos(90+3.5*\s)}, {\r*sin(90+3.5*\s)}};
\vtx{5}{{\r*cos(90+4.5*\s)}, {\r*sin(90+4.5*\s)}};
\vtx{6}{{\r*cos(90+5.5*\s)}, {\r*sin(90+5.5*\s)}};
\vtx{7}{{\r*cos(90+6.5*\s)}, {\r*sin(90+6.5*\s)}};

\foreach \from/\to in {7/6,6/5,5/4,4/3,3/2,2/1,1/7,1/3,2/4,3/5,4/6,5/7,6/1,7/2}
\dedge{\from}{\to};

\end{scope}
\end{tikzpicture}

\caption{The Heawood graph (left) and its unique collapse $F_7$ (right).} \label{fig:heawood}
\end{figure}

The following two results are proved in~\cite{pfaffians}:
\begin{thm}\label{pfaffianthm}
Let $H$ be a Pfaffian brace. Then either $H$ is planar, $H$ is the Heawood graph, or there is a subgraph $R$ of $H$ that is 
an odd subdivision
of Rotunda such that
$H\setminus V(R)$ has a perfect matching.
\end{thm}

\begin{thm}\label{userotunda}
Let $H$ be a Pfaffian brace, and suppose that $R$ is a subgraph of $H$ that is
an odd subdivision
of Rotunda such that 
$H\setminus V(R)$ has a perfect matching. Let $X$ be the join of $R$. Then the three components of $R\setminus X$ are contained in three
distinct components of $H\setminus X$.
\end{thm}

We will deduce from~\ref{pfaffianthm} and~\ref{userotunda} that:
\begin{thm}\label{unevendecomp}
Let $G$ be a 2-strong, 3-weak odd-weightable digraph that is not diplanar. Then either $G$ is $F_7$ (shown in Figure~\ref{fig:heawood}), or
there is a set $Y\subseteq V(G)$ such that one of the following holds (shown in Figure~\ref{fig:unevendecomp}):
\begin{enumerate}
    \item\label{case1} $|Y|=3$, $Y=\{y_1,y_2,y_3\}$ say, and $G\setminus Y$ has the following properties:
    \begin{enumerate}
        \item It has at least three weak components $D_1, D_2,\dots$.
        \item For $i=1,2$, $|V(D_i)|\ge 2$, $y_1\to D_i$, and $y_2\from D_i$.
    \end{enumerate}
    
    \item\label{case2} $|Y|=4$, $Y=\{y_1,y_2,y_3, y_4\}$ say, where $y_4y_1,y_2y_3\notin E(G)$, and $G\setminus Y$ has the following properties:
    \begin{enumerate}
        \item It has at least three weak components $D_1,D_2,\dots$.
        \item $y_1\to D_1$ and $y_2\from D_1$.
        \item $y_3\to D_2$ and $y_4\from D_2$,
        \item For $i>2$, $y_1,y_3\to D_i$ and $y_2,y_4\from D_i$.
    \end{enumerate}
    
    \item\label{case3} $|Y|=4$, $Y=\{y_1,y_2,y_3, y_4\}$ say, and $G\setminus Y$ has the following properties:
    \begin{enumerate}
        \item It has at least two weak components $D_1,D_2,\dots$.
        \item For $i=1,2$, $|V(D_i)|\ge 2$, $y_1, y_3\to D_i$, and $y_2,y_4\from D_i$.
    \end{enumerate}
\end{enumerate}
\end{thm}

\begin{figure}[htb!]
    \centering
    \begin{tikzpicture}[scale=1.3]
        \tikzset{every label/.style={label distance=-2pt}};
        \vset{C1}{$D_1$}{-1.5,2.25};
        \vset{C2}{$D_2$}{1.5,2.25};
        \vset{R}{\footnotesize$\ge 1\text{ component}$}{0,-0.25};

        \vtx[{[yshift=-3pt]right:$y_1$}]{y1}{-0.7,1};
        \vtx[{[yshift=-3pt]right:$y_2$}]{y2}{0,1};
        \vtx[{[yshift=-3pt]right:$y_3$}]{y3}{0.7,1};

        \sline{y1}{R.152};
        \sline{y2}{R.90};
        \sline{y3}{R.28};
        
        \hdedge{y1}{C1.267};
        \draw[very thick, midarrow=0.35] (C1.320) -- (y2);
        \sline{y3}{C1.348};

        \draw[very thick, midarrow=0.75] (y1) -- (C2.192);
        \draw[very thick, midarrow=0.35] (C2.220) -- (y2);
        \sline{y3}{C2.273};

        \node (1) at (0,-1) {Case~\ref{case1}};

        \vset{C1}{$D_1$}{4.5,2.25};
        \vset{C2}{$D_2$}{7.5,2.25};
        \vset{R}{\footnotesize$\ge 1\text{ component}$}{6,-0.25};

        \vtx[{[yshift=-3pt]right:$y_1$}]{y1}{5.1,1};
        \vtx[{[yshift=-3pt]right:$y_2$}]{y2}{5.7,1};
        \vtx[{[yshift=-3pt]right:$y_3$}]{y3}{6.3,1};
        \vtx[{[yshift=-3pt]right:$y_4$}]{y4}{6.9,1};

        \hdedge{y1}{R.165};
        \hdedge{R.122}{y2};
        \hdedge{y3}{R.58};
        \hdedge{R.15}{y4};
        
        \hdedge{y1}{C1.242};
        \draw[very thick, midarrow=0.35] (C1.300) -- (y2);
        \sline{y3}{C1.333};
        \sline{y4}{C1.355};

        \sline{y1}{C2.185};
        \sline{y2}{C2.207};
        \draw[very thick, midarrow=0.65] (y3) -- (C2.240);
        \hdedge{C2.298}{y4};

        \node (2) at (6,-1) {Case~\ref{case2}};

        \vset{R}{\footnotesize$\ge 0\text{ components}$}{3,-0.75};
        \vset{C1}{$D_1$}{1.5,-3.25};
        \vset{C2}{$D_2$}{4.5,-3.25};

        \vtx[{[yshift=3pt]right:$y_1$}]{y1}{2.1,-2};
        \vtx[{[yshift=3pt]right:$y_2$}]{y2}{2.7,-2};
        \vtx[{[yshift=3pt]right:$y_3$}]{y3}{3.3,-2};
        \vtx[{[yshift=3pt]right:$y_4$}]{y4}{3.9,-2};

        \sline{y1}{R.195};
        \sline{y2}{R.238};
        \sline{y3}{R.302};
        \sline{y4}{R.345};
        
        \hdedge{y1}{C1.118};
        \draw[very thick, midarrow=0.35] (C1.60) -- (y2);
        \draw[very thick, midarrow=0.8] (y3) -- (C1.27);
        \draw[very thick, midarrow=0.2] (C1.5) -- (y4);

        \draw[very thick, midarrow=0.8] (y1) -- (C2.175);
        \draw[very thick, midarrow=0.2] (C2.153) -- (y2);
        \draw[very thick, midarrow=0.65] (y3) -- (C2.120);
        \hdedge{C2.62}{y4};

        \node (2) at (3,-4) {Case~\ref{case3}};
    \end{tikzpicture}
    \caption{The three decompositions of~\ref{unevendecomp}. Note that some of the ovals may correspond to more than one weak component or may be empty, as labeled. Also, edges with both ends in $Y$ are not shown, and there are some additional size constraints on certain $D_i$; check the theorem statement for details.}
    \label{fig:unevendecomp}
\end{figure}
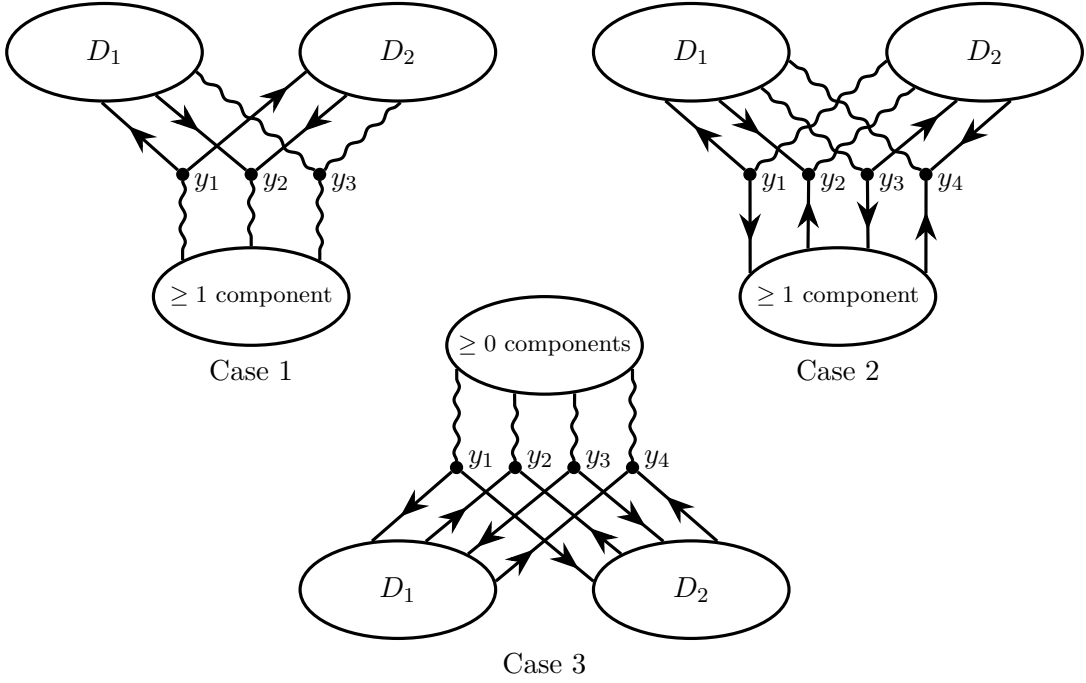

\begin{proof}
Let $(H,M)$ be the bisource for $G$. Since $G$ is 2-strong and odd-weightable, it follows that $H$ is a Pfaffian brace, by theorems 
of~\cite{pfaffians}. Since $G$ is not diplanar, it follows that $H$ is not planar, so we can apply~\ref{pfaffianthm}
and~\ref{userotunda}.
If $H$ is the Heawood graph, then $G$ is the digraph $F_7$ of Figure~\ref{fig:heawood}, so we assume not. Hence by
\ref{pfaffianthm} and~\ref{userotunda}, there is a subgraph $R$ of $H$ that is
an odd subdivision
of Rotunda, with join $X$ say, such that:
\begin{itemize}
\item $H\setminus V(R)$ has a perfect matching, and
\item the three components of $R\setminus X$ are contained in three distinct components of $H\setminus X$.
\end{itemize}
Let $(A,B)$ be a bipartition of $R$. Let $R_1,R_2,R_3$ be the three components of $R\setminus X$.
Since $R$ is an odd subdivision of Rotunda, it follows 
that $|A\cap X|=2$. Let $X=\{x_1,\dots, x_4\}$ where $X\cap A=\{x_1,x_3\}$.
Let $C_1,\dots, C_k$ be the components of $H\setminus X$, and let $R_i\subseteq C_{t_i}$ for $ i = 1,2,3$. We have $k\ge 3$, and by~\ref{userotunda}, we have that $t_1,t_2,t_3$
are all different.

\begin{step}\label{step1e}
$|V(C)\cap A|=|V(C)\cap B|$ for each component $C$ of $H\setminus X$, and therefore the number of edges in $M$ between   
$X\cap A$ and $V(C)\cap B$ equals the number between $X\cap B$ and $V(C)\cap A$.
\end{step}

By relabeling the components of $H\setminus X$ if necessary, we may assume $t_i=i$ for $i=1,2,3$. We may also assume $C\ne C_1,C_2$ 
by the symmetry between $C_1,C_2,C_3$. Choose edges $e_1,e_2$ between $x_1$ and $V(C_1)$ and 
between $x_3$ and $V(C_2)$, respectively. 
Since $H$ is 2-extendible, the matching $\{e_1,e_2\}$ can be extended to a perfect matching of $H$, and consequently
$|V(C)\cap A|\le |V(C)\cap B|$. Similarly, $|V(C)\cap B|\le |V(C)\cap A|$, and so equality holds. This proves the first statement of~\eqref{step1e}, and the
second follows since $M$ is a perfect matching. This proves~\eqref{step1e}.

\bigskip

For $1\le i\le k$, let $M_i$ be the set of edges in $M$ that have both ends in $V(C_i)$. 
We produce $G$ from $H$ by directing the edges of $H$ from $A$ to $B$, and then contracting the edges of $M$; for each edge $e\in M$, let $\phi(e)$ be the vertex of $G$ made by
contracting $e$, and for $N\subseteq M$ let $\phi(N) = \{\phi(e):e\in N\}$. Let $Y$ be the set of vertices $\phi(e)$ such that 
$e\in M$ has an end in $X$. For each $i$, let $D_i=G[\phi(M_i)]$; then the nonnull digraphs among $D_1,\dots,D_k$ are precisely the
weak components of $G\setminus Y$. Since $G$ is 3-weak, $3\le |Y|\le 4$, so at most one edge of $M$ has both ends in $X$. 

The relationship between the matching $M$ and the vertices $X$ determines the outcome of the theorem. The three cases are depicted in Figure~\ref{fig:matching}.

\begin{figure}[htb!]
    \centering
    \begin{tikzpicture}[scale=1.3]
        \tikzset{every label/.style={label distance=-2pt}};
        \vset{C1}{$C_1$}{-1.5,2.25};
        \vset{C2}{$C_2$}{1.5,2.25};
        \vset{R}{$C_j$}{0,-0.25};

        \vtx[above:$x_1$]{x1}{-0.9,1};
        \vtx[above:$x_2$]{x2}{-0.3,1};
        \vtx[above:$x_3$]{x3}{0.3,1};
        \vtx[above:$x_4$]{x4}{0.9,1};

        \draw (x1) -- (R.165);
        \draw (x2) -- (R.122);
        \draw (x3) -- (x4);

        \node (1) at (0,-1) {Case~\ref{case1}};

        \vset{C1}{$C_1$}{4.5,2.25};
        \vset{C2}{$C_2$}{7.5,2.25};
        \vset{R}{}{6,-0.25};

        \vtx[below:$x_1$]{x1}{5.1,1};
        \vtx[below:$x_2$]{x2}{5.7,1};
        \vtx[below:$x_3$]{x3}{6.3,1};
        \vtx[below:$x_4$]{x4}{6.9,1};
        
        \draw (x1) -- (C2.185);
        \draw (x2) -- (C2.207);

        \draw (x3) -- (C1.333);
        \draw (x4) -- (C1.355);

        \node (2) at (6,-1) {Case~\ref{case2}};

        \vset{R}{$C_j$}{3,-0.75};
        \vset{C1}{$C_1$}{1.5,-3.25};
        \vset{C2}{$C_2$}{4.5,-3.25};

        \vtx[below:$x_1$]{x1}{2.1,-2};
        \vtx[below:$x_2$]{x2}{2.7,-2};
        \vtx[below:$x_3$]{x3}{3.3,-2};
        \vtx[below:$x_4$]{x4}{3.9,-2};

        \draw (x1) -- (R.195);
        \draw (x2) -- (R.238);
        \draw (x3) -- (R.302);
        \draw (x4) -- (R.345);

        \node (2) at (3,-4) {Case~\ref{case3}};
    \end{tikzpicture}
    \caption{The cases of $M\subseteq E(H)$ that lead to the three outcomes of~\ref{unevendecomp}. Here we just draw the edges in $M$ and 
the components of $H\setminus X$ that play an important role in the proof. The labels of the elements drawn are consistent with the 
proof of~\ref{unevendecomp}, and their positions are consistent with Figure~\ref{fig:unevendecomp}.}
    \label{fig:matching}
\end{figure}
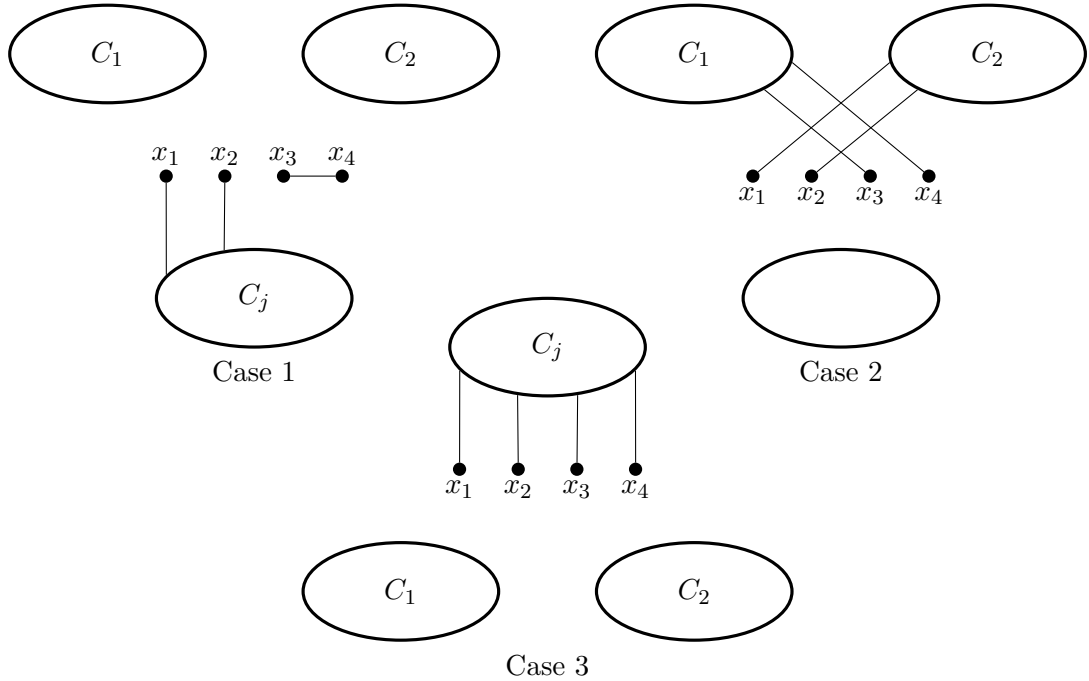

Suppose first that there is an edge of $M$ with both ends in $X$, 
say $x_3x_4\in M$. In this case, we may assume that $t_i=i$ for $i=1,2,3$. Let $e_1,e_2$ be the edges in $M$ incident with $x_1,x_2$ respectively. 
By~\eqref{step1e}, there exists $j\in \{1,\dots, k\}$ such that $e_1,e_2$
both have an end in $V(C_j)$. From the symmetry between $C_1,C_2,C_3$, we may assume that $j\ne 1,2$.
Since $C_1,C_2$ each have at least four vertices, it follows that $D_1,D_2$ each have at least two vertices.
Moreover,
for $i = 1,2$, all edges of $H$ between $V(C_i)$
and an end of $e_1$ are incident with $x_1$ (since the other end of $e_1$ is in $V(C_j)$), and so in $G$,
all edges between $V(D_i)$ and $y_1$ are directed from $y_1$
to $V(D_i)$, that is, $y_1\to D_i$. Similarly, $y_2\from D_i$ for $i = 1,2$. Since $|V(C_3)|\ge 4$, it follows that $|D_3|\ge 1$ (even if $j=3$),
so $G\setminus Y$ has at least three weak components, and therefore case~\ref{case1} theorem holds. 

Thus we may assume that no edge of $M$ has both ends in $X$. For $1\le i\le 4$, let $e_i$ be the edge in $M$ incident with $x_i$, and let $\phi(e_i) = y_i$.

Next, we assume that there exist distinct $j_1,j_2\in \{1,\dots, k\}$ such that $e_3, e_4$ have an end in $V(C_{j_1})$
and $e_1,e_2$ each have an end in $V(C_{j_2})$; and so, by relabeling the components of $H\setminus X$ if necessary, we may assume $j_1=1$ and $j_2=2$. It follows that $y_4y_1\notin E(G)$ (because the end of $e_1$ in $B$ and the end
of $e_4$ in $A$ belong to different components of $H\setminus X$), and similarly $y_2y_3\notin E(G)$. 

If no edge in $M$ joins $X$ and $V(C_i)$, then $y_1,y_3\to D_i$ and $y_2,y_4\from D_i$. This is the case for all $i$ except for $i=1,2$. 
Additionally, we have $y_1\to D_1$, $y_2\from D_1$, $y_3\to D_2$, and $y_4\from D_2$. Also, since $|V(C_{t_i})|\ge 4$ for $i=1,2,3$, the 
corresponding subdigraphs $D_{t_i}$ are nonnull (even if $t_i\in\{j_1,j_2\}$). Therefore case~\ref{case2} of the theorem holds.

Finally, we assume that for some $j\in \{1,\dots, k\}$, $e_1,\dots, e_4$ all have an end in $V(C_j)$. In this case, we may again assume that $t_i=i$ for $i=1,2,3$, and we may further assume that $j\ne 1,2$. Then for $i=1,2$, we have
$y_1, y_3\to D_i$ and $y_2,y_4\from D_i$. Also, $C_1,C_2$ both have at least four vertices, so $D_1,D_2$ both have at least two vertices (possibly $V(D_3)=\emptyset$). Thus case~\ref{case3} of the theorem holds.
This proves~\ref{unevendecomp}.\end{proof}

This result tells us about decompositions in odd-weightable digraphs. Since weightable 
digraphs are odd-weightable,~\ref{unevendecomp} applies also to weightable digraphs, but for weightable digraphs we can refine these 
decompositions so that they become the reverse of constructions, as we show in the next few sections.

\section{Path lemmas}

We need some lemmas about dipaths. We recall that if $P$ is a dipath and $u,v\in V(P)$, and $u$ is earlier than $v$ in $P$, then 
$P[u,v]$ denotes the subpath of $P$ from $u$ to $v$. 
\begin{thm}\label{pathsmeet}
Let $G$ be a digraph, and let $a_1,a_2,b_1,b_2\in V(G)$ be distinct. Suppose that:
\begin{itemize}
\item $|V(G)|\ge 5$;
\item no edge has head $a_1$ or $a_2$ and no edge has tail $b_1$ or $b_2$;
\item  for each vertex $v\in V(G)\setminus \{b_1,b_2\}$ and $i = 1,2$, there is a dipath from $v$ to $b_i$; and
\item  for every vertex $v\in V(G)\setminus \{a_1,a_2\}$ and $i = 1,2$, there is a dipath from $a_i$ to $v$.
\end{itemize}
Then there are two dipaths from $\{a_1,a_2\}$ to $\{b_1,b_2\}$ that are not vertex-disjoint, such that the ends of these paths are 
all distinct.
\end{thm}
\begin{proof} We begin with:

\begin{step}\label{step1f}
There is a dipath from $a_1$ to $b_1$ with length at least two.
\end{step}

Since $|V(G)|\ge 5$, there is a vertex $v$ different from $a_1,a_2,b_1,b_2$; by the third and fourth bullets, there is a dipath 
from $a_1$ to $v$ and from $v$ to $b_1$. Therefore the union of these paths includes
a dipath from $a_1$ to $b_1$ of length at least two. This proves~\eqref{step1f}.

\bigskip

Let us say a \emph{fork} means a triple $(P_1,P_2,Q_1)$, where $P_1,P_2,Q_1$ are dipaths, pairwise vertex-disjoint except 
that they have a common end $v\ne a_1,a_2,b_1,b_2$, and 
$P_1$ is from $a_1$ to $v$, $P_2$ is from $a_2$ to $v$, and $Q_1$ is from $v$ to $b_1$. 

\begin{step}\label{step2f}
There is a fork.
\end{step}

By~\eqref{step1f}, there is a dipath from $a_1$ to $b_1$ with non-null interior. By the fourth bullet of the theorem, there is a 
dipath from $a_2$ to the interior of $P_1$; choose a minimal such path $P_2$, and let $v$ be the end of $P_2$ in the interior of $R$. Let $P_1,Q_1$ be the subpaths of $R$ from $a_1$ to $v$ and from $v$ to $b_1$, respectively. Then $(P_1,P_2,Q_1)$ is a fork. This
proves~\eqref{step2f}.

\bigskip
Let us choose a fork $(P_1,P_2,Q)$ with $Q$ maximal, and let $v$ be the common end of the three paths. 

\begin{step}\label{step3f}
There is no dipath between $V(P_1)$, $V(P_2)$ that is vertex-disjoint from $V(Q_1)$.
\end{step}

Suppose there is a such a path from $V(P_1)$ to $V(P_2)$, say, and let
$R$ be a minimal such path, with ends $r_i\in V(P_i)$ for $i = 1,2$. Then
\[(P_1[a_1,r_1]\cup R, P_2[a_2,r_2], P_2[r_2,v]\cup Q_2)\]
is a fork, contrary to the maximality of $Q_1$. This proves~\eqref{step3f}.

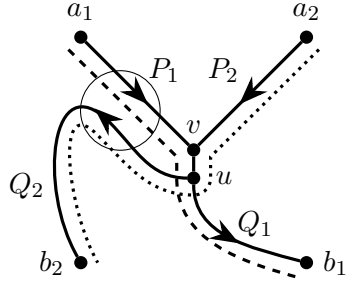
\begin{figure}[htb!]
    \centering
    \begin{tikzpicture}[scale=1.5]
        \vtx[above:$a_1$]{a1}{-1,2};
        \vtx[above:$a_2$]{a2}{1,2};
        \vtx[left:$b_2$]{b2}{-1,0};
        \vtx[right:$b_1$]{b1}{1,0};
        \vtx[above:$v$]{v}{0,1};
        \vtx[{[label distance=2pt]right:$u$}]{u}{0,0.75};

        \node[circle, draw, minimum size=2.75em] (C) at (-0.65, 1.35) {};

        \draw[very thick, midarrow] (a1) -- (v) node[midway, above right] {$P_1$};
        \draw[very thick, midarrow] (a2) -- (v) node[midway, above left] {$P_2$};
        \draw[very thick, midarrow=0.35] (u) to[out=180, in=315] (-0.75, 1.25) to[out=135, in=120, looseness=1.5] node[pos=0.7, left] {$Q_2$} (b2);

        \draw[very thick, midarrow] (v) -- (u) to[out=270, in=165] node[pos=0.7, above] {$Q_1$} (b1);
        
        \draw[very thick, dashed] ([shift={(-0.11,-0.11)}]a1.center) -- ([shift={(-0.15, -0.06)}]v.center) -- ([shift={(-0.15, 0)}]u.center) to[out=270, in=165] ([shift={(-0.04, -0.14)}]b1.center);

        \draw[very thick, dotted] ([shift={(0.11,-0.11)}]a2.center) -- ([shift={(0.15, -0.06)}]v.center) -- ([shift={(0.15, 0)}]u.center) to[out=270, in=0] ([shift={(0, -0.15)}]u.center) to[out=180, in=315] (-0.86, 1.14) to[out=135, in=120, looseness=1.2] ([shift={(0.15, 0)}]b2.center);
    \end{tikzpicture}
    \caption{The end of the proof of~\ref{pathsmeet}. Inside the circle, the paths $P_1$ and $Q_2$ may intersect in an arbitrary way; none of the other paths may intersect. The dashed and dotted paths are the ones we desired.}
    \label{fig:pathsmeet}
\end{figure}
There is a dipath from $V(Q_1)$ to $b_2$; let $Q_2$ be a minimal such path, with ends $u\in V(Q_1)$ and $b_2$.
Thus, $u$
is the only vertex of $Q_2$ in $V(Q_1)$, and $u\ne b_1$. By~\eqref{step3f}, $Q_2$ does not meet both $V(P_1)\setminus \{v\}$ and
$V(P_2)\setminus \{v\}$, and so, exchanging $a_1,a_2$ if necessary, we may assume that
$V(Q_2)\cap V(P_2)\subseteq \{u\}\cap \{v\}$. But then the dipaths
$P_2\cup Q_1[v,u]\cup Q_2$ and
$P_1\cup Q_1$ have nonempty intersection and satisfy the theorem. (See Figure~\ref{fig:pathsmeet}.)
This proves~\ref{pathsmeet}.\end{proof}

Let $G$ be a digraph, let $a_1,a_2,b_1,b_2\in V(G)$, and for $i = 1,2$ let $P_i$ be a dipath from $a_i$ to $b_i$, chosen such that $P_1,P
_2$ intersect, and subject to that their union is minimal. We say that $P_1,P_2$ are in \emph{bubble form}.

\begin{thm}\label{bubble}
Let $G,a_1,a_2,b_1,b_2, P_1,P_2$ be as above, where $P_1,P_2$ are in bubble form.
Then $P_1\cap P_2$ is the disjoint union of dipaths $Q_1,\dots, Q_k$ for some $k\ge 1$, such that $Q_1,\dots, Q_k$ are in this order in $P_1$, and in the reverse order in $P_2$. (See Figure~\ref{fig:bubble}.)
\end{thm}

\begin{figure}[htb!]
    \centering
    \begin{tikzpicture}[scale=1.5]
        \vtx[above:$a_1$]{a1}{-4.5,3};
        \vtx[above:$a_2$]{a2}{-3,3};
        \vtx[below:$b_2$]{b2}{-4.5,0};
        \vtx[below:$b_1$]{b1}{-3,0};
        \vtx{c1}{-4.5,2};
        \vtx{c2}{-3,2};
        \vtx{d1}{-4.5,1};
        \vtx{d2}{-3,1};

        \foreach \from/\to in {a1/c1, c1/d1, d1/b2, a2/c2, c2/d2, d2/b1}
        \hdedge{\from}{\to};

        \draw[very thick, midarrow=0.25] (d1) -- (c2);
        \draw[very thick, midarrow=0.25] (d2) -- (c1);

        \node[left] (P1) at (-4.6,2.75) {$P_1$};
        \node[right] (P2) at (-2.9,2.75) {$P_2$};
        \node[left] (Q1) at (-4.8,1.5) {$Q_1$};
        \node[right] (Q2) at (-2.7,1.5) {$Q_2$};
        \draw ([xshift=-2pt]Q1.east) -- (-4.5, 1.5);
        \draw ([xshift=2pt]Q2.west) -- (-3, 1.5);

        \draw[dashed, very thick] ([shift={(-0.15, 0)}]a1.center) -- ([shift={(-0.15, 0)}]d1.center) to[out=270, in=207] ([shift={(0.067, -0.134)}]d1.center) -- ([shift={(-0.15, -0.243)}]c2.center) -- ([shift={(-0.15, 0)}]b1.center);

        \draw[dotted, very thick] ([shift={(0.15, 0)}]a2.center) -- ([shift={(0.15, 0)}]d2.center) to[out=270, in=333] ([shift={(-0.067, -0.134)}]d2.center) -- ([shift={(0.15, -0.243)}]c1.center) -- ([shift={(0.15, 0)}]b2.center);
    
        \vtx[above:$a_1$]{a1}{-1,3};
        \vtx[above:$a_2$]{a2}{1,3};
        \vtx[below:$b_2$]{b2}{-1,0};
        \vtx[below:$b_1$]{b1}{1,0};
        \vtx{c1}{-1,2};
        \vtx{c2}{0,2};
        \vtx{c3}{1,2};
        \vtx{d1}{-1,1};
        \vtx{d2}{0,1};
        \vtx{d3}{1,1};

        \foreach \from/\to in {a1/c1, c1/d1, d1/b2, c2/c1, d1/d2, d2/c2, c2/c3, d3/d2, a2/c3, c3/d3, d3/b1}
        \hdedge{\from}{\to};

        \node[left] (P1) at (-1.1,2.75) {$P_1$};
        \node[right] (P2) at (1.1,2.75) {$P_2$};
        \node[right] (Q1) at (-0.7,1.5) {$Q_1$};
        \node[right] (Q2) at (0.3,1.5) {$Q_2$};
        \node[right] (Q3) at (1.3,1.5) {$Q_3$};
        \draw ([xshift=2pt]Q1.west) -- (-1, 1.5);
        \draw ([xshift=2pt]Q2.west) -- (0, 1.5);
        \draw ([xshift=2pt]Q3.west) -- (1, 1.5);

        \draw[dashed, very thick] ([shift={(-0.15, 0)}]a1.center) -- ([shift={(-0.15, 0)}]d1.center) to[out=270, in=180] ([shift={(0, -0.15)}]d1.center) -- ([shift={(0, -0.15)}]d2.center) to[out=0, in=270] ([shift={(0.15, 0)}]d2.center) -- ([shift={(0.15, -0.15)}]c2.center) -- ([shift={(-0.15, -0.15)}]c3.center) -- ([shift={(-0.15, 0)}]b1.center);

        \draw[dotted, very thick] ([shift={(0.15, 0)}]a2.center) -- ([shift={(0.15, 0)}]d3.center) to[out=270, in=0] ([shift={(0, -0.15)}]d3.center) -- ([shift={(0, -0.15)}]d2.center) to[out=180, in=270] ([shift={(-0.15, 0)}]d2.center) -- ([shift={(-0.15, -0.15)}]c2.center) -- ([shift={(0.15, -0.15)}]c1.center) -- ([shift={(0.15, 0)}]b2.center);
    \end{tikzpicture}
    \caption{Two examples of paths $P_1$ (dashed) and $P_2$ (dotted) in bubble form as described in~\ref{bubble}, here with $k=2$ (left) or 3 (right) paths in $P_1\cap P_2$.}
    \label{fig:bubble}
\end{figure}
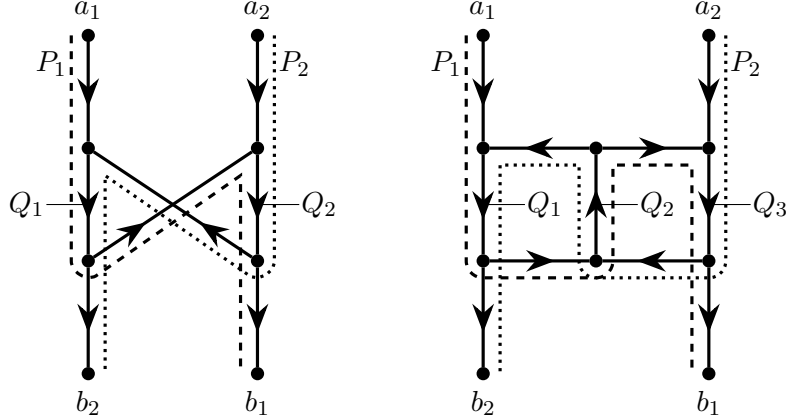

\begin{proof} We claim that:

\begin{step}\label{step1g}
If $u,v\in V(P_1\cap P_2)$, then either $u$ is before $v$ in one of $P_1,P_2$ and after $v$ in the other,
or the subpaths of $P_1,P_2$ joining $u,v$ are equal.
\end{step}

To see this, suppose that $u$ is before $v$ in both $P_1,P_2$.
There is a dipath $P_2'$ from $a_2$ to $b_2$, 
included in $P_2[a_2,u]\cup P_1[u,v]\cup P_2[v, b_2]$. 
Since $P_2'\subseteq P_1\cup P_2$, the minimality of $P_1\cup P_2$ implies that $P_2 \subseteq P_1\cup P_2'$, and in 
particular every edge of $P_2[u,v]$ belongs to $E(P_1\cup P_2')$ and hence to $E(P_1)$ (because it cannot belong to $P_2'$ 
unless it is in $P_1[u,v]$).  So $P_2[u,v]$ is a subpath of $P_1$, and therefore $P_2[u,v]=P_1[u,v]$. This proves~\eqref{step1g}. 

\bigskip

Certainly $P_1\cap P_2$ is the disjoint union of some number of dipaths $Q_1,\dots, Q_k$; let us 
number them in their order they appear in $P_1$. By~\eqref{step1g}, these subpaths appear in $P_2$ in reverse order. This proves~\ref{bubble}.\end{proof}

A similar statement appears as Lemma~4.3 in~\cite{2strong}. If $P_1,P_2$ are in bubble form, we call the number of components of $P_1\cap P_2$ their \emph{intersection number}, and we say
$P_1,P_2$ \emph{make a bubble} if they have intersection number two.

\section{Refining the outcomes of~\ref{unevendecomp}}

Now we will start to convert the outcomes of~\ref{unevendecomp} to reversible decompositions for weightable digraphs. We begin with the the first
outcome of~\ref{unevendecomp}. 
Two dipaths are \emph{internally disjoint} if every vertex in their intersection is an end of each of them.
\begin{thm}\label{firstoutcome}
Let $y_1,y_2,y_3$ be distinct vertices of a 2-strong, 3-weak, weightable digraph $G$, and let $Y=\{y_1,y_2,y_3\}$.  Suppose that
$G\setminus Y$ has at least three weak components $D_i$, and for $i=1,2$,
$y_1\to D_i$ and $y_2\from D_i$. Then $G$ is nonplanar and can be built from two smaller weightable digraphs by an application of the construction of~\ref{firstcon}.
\end{thm}
\begin{proof}
Since $G$ is 3-weak, each component of $G\setminus Y$ has an edge to each of $y_1,y_2,y_3$, so the underlying graph $G^-$ has a $K_{3,3}$ minor and is not planar. In this proof, let $A=D_1$ and $B=D_2$. 
Let $C=G\setminus V(A\cup B)$. For distinct $y_i,y_j\in Y$, we mean by $A[i,j]$ a dipath from $y_i$ to $y_j$ with all internal 
vertices in $V(A)$, and the same for $B$. By $C[i,j]$ we mean a dipath of $C$ from $y_i$ to $y_j$ with no internal vertex in $Y$.

\begin{step}\label{step1h}
$C[2,1], C[3,2]$ exist.
\end{step}

Since $G$ is 2-strong, there are two internally disjoint dipaths from $\{y_2,y_3\}$ to $y_1$, and since $y_1\to A\cup B$, it follows
that each of these dipaths has no vertex in $V(A\cup B)$. 
This proves~\eqref{step1h}.

\begin{step}\label{step2h}
$A[1,3], A[3,2], B[1,3], B[3,2]$ exist.
\end{step}

Choose $v\in V(A)$. Since $G$ is 2-strong, there are two internally disjoint dipaths of $G$ from $v$ to $Y$, and since 
$y_1\to A$, all their vertices belong to $V(A)\cup \{y_2,y_3\}$. Similarly, there are dipaths from $y_1,y_3$ to $v$ such that all their
vertices belong to $\{y_1,y_3\}\cup V(A)$. Consequently, $A[1,3], A[3,2]$ exist, and the same holds for $B$. This proves~\eqref{step2h}.

\begin{step}\label{step3h}
$B[1,3], B[3,2]$ are internally disjoint, for every choice of $B[1,3], B[3,2]$.
\end{step}

Suppose not, and choose
$b\in V(B)$ that belongs to both paths. Then the dicycles
$B[1,3]\cup C[3,1]$ and $B[3,2]\cup C[2,1]\cup A[1,3]$ disagree on $\{b,y_3,y_1\}$, contrary to~\ref{triple}.
This proves~\eqref{step3h}.

\bigskip
Let $W_1=G\setminus V(A)$, and let $W_2$ be the subdigraph obtained from $A$ by adding $Y$ and all edges of $G$ between $Y, V(A)$.
Thus, $W_1,W_2$
are internally disjoint $Y$-wings with union $G$.
We claim:
\begin{itemize}
\item There is a $W_1$-path from $y_2$ to $y_1$. This is true because 
$C[2,1]$
is such a path.
\item $y_1$ is a source of $W_2$ and $y_2$ is a sink of $W_2$. This is true because $y_1\to A$ and $y_2\from A$.
\item The digraph $G_1$ obtained
from $W_1$ by adding the edges $y_1y_2, y_1y_3,y_3y_2$ is weightable and 1-strong. This is true because
$G_1$ is clearly 1-strong, and it is weightable by~\ref{butterfly} and~\eqref{step3h} since it can be obtained as a butterfly minor from $G$ by contracting singular edges of
$W_1\cup A[1,2]\cup B[1,3]\cup B[3,2]$. 
\item The digraph $G_2$
obtained from $W_2$ by adding the edges $y_2y_1, y_1y_3,y_3y_2$ is weightable and 1-strong.
This is true because $G_2$ is clearly 1-strong, and it is weightable by~\ref{butterfly} and~\eqref{step3h} since it can be obtained as a butterfly minor from $G$ by contracting singular edges 
of $W_2\cup C[2,1]\cup B[1,3]\cup B[3,2]$.
\end{itemize}
This proves~\ref{firstoutcome}.\end{proof}

Now we handle the second outcome of~\ref{unevendecomp}:
\begin{thm}\label{secondoutcome}
Let $y_1,\dots, y_4$ be distinct vertices of a 2-strong, 3-weak, weightable digraph $G$, and let 
$Y=\{y_1,\dots, y_4\}$. Suppose that $G\setminus Y$ has at least three weak components $D_i$ such that the four assumptions of~\ref{unevendecomp} case~\ref{case2} hold. Then $G$ is nonplanar and can be constructed from three smaller weightable digraphs by the construction of~\ref{secondcon}.
\end{thm}
\begin{proof} For each weak component $D$ of $G\setminus Y$ and
all distinct $i,j\in \{1,\dots, 4\}$,  $D[i,j]$ denotes
a dipath of $G$ from $y_i$ to $y_j$ with all internal vertices in $V(D)$, and $D^+$ is the digraph formed by adding to $D$ the vertices of $Y$ and all edges between $D$ and $Y$. Let $A=D_1$, $B=D_2$, and $C=D_3$.

\begin{step}\label{step1i}
For each $i>2$, $D_i[1,2],D_i[1,4],D_i[3,2], D_i[3,4]$ exist, and every choice of $D_i[3,2]$ is vertex-disjoint from every choice of $D_i[1,4]$.
\end{step}

It suffices to prove this for $i=3$ (so $D_i=C$). Choose $v\in V(C)$. Since $G$ is 2-strong, there are two dipaths from $v$ to $Y$, vertex-disjoint except for $v$, and since
$y_1,y_3\to C$, it follows that these paths are from $v$ to $y_2,y_4$ respectively. Similarly, there are dipaths from $y_1,y_3$
to $v$, and consequently $C[1,2],C[1,4],C[3,2], C[3,4]$ exist.

Suppose that there is a vertex $c$ in some $C[3,2]$ and in some $C[1,4]$. There are two vertex-disjoint dipaths 
from $y_4$ to $y_1$, say $P,Q$. Since each weak component $D$ of $G\setminus Y$ has either $y_1\to D$ or $y_4\from D$ (or both), each of $P$ and $Q$ contains one of $y_2,y_3$, so we may assume that $P$ contains 
$y_2$ and $Q$ contains $y_3$. Since $y_4\from D_i$, $y_1\to D_i$, and $y_2\notin V(Q)$, it follows that 
$V(Q)\cap V(C)=\emptyset$. But then $C[1,4]\cup Q$ is a dicycle containing $c,y_3,y_1$ in this order.

Similarly, there are two internally disjoint paths $P',Q'$ from $y_2$ to $y_3$.
These each must contain one of $p_1,p_4$; we may assume $Q'$ contains $p_1$. Then $V(Q')\cap V(C)=\emptyset$ and $C[3,2]\cup Q'$ contains $c,y_1,y_3$ in this order, so $\{c,y_1,y_3\}$ is a bad triple, contradicting~\ref{triple}. 
This proves~\eqref{step1i}. 

\begin{step}\label{step2i}
In $G$, every dipath from $y_1$ to $y_3$ intersects every dipath from $y_4$ to $y_2$.
\end{step}

Suppose that $P,Q$ are dipaths from $y_1$ to $y_3$ and from $y_4$ to $y_2$, respectively, that are vertex-disjoint.
In particular, no internal vertex of $P$ or $Q$ belongs to $Y$. Hence, either $P$ has length one or $P$ is a path of $A^+$. Similarly, either $Q$ has length one or
it is a path of $A^+$. There is a dipath $R$ from 
$y_2$ to $y_1$ in 
$G\setminus \{y_4\}$. Either $R$ has length one, or its second vertex belongs to $V(B)$; if the latter, then $R$ is a path of $B^+$ since $y_3\to D_1$. In particular, $R$ is internally disjoint from $P,Q$.

\begin{figure}[htb!]
    \centering
    \begin{tikzpicture}[scale=2]
        \tikzset{every label/.style={label distance=-2pt}};
        \smallvset{B}{}{-2, 0.8};
        \vset{A}{}{0, 1};
        \vset{C}{}{0,-1};
        
        \node[above] at (B.north) {$B$};
        \node[above] at (A.north) {$A$};
        \node[below] at (C.south) {$C$};

        \vtx[right:$y_1$]{y1}{-0.9,0};
        \vtx[right:$y_2$]{y2}{-0.3,0};
        \vtx[right:$y_3$]{y3}{0.3,0};
        \vtx[right:$y_4$]{y4}{0.9,0};

        \vtx[above:$b$]{b}{0, -0.8};
        \vtx[below:$a$]{a}{0, -1.2};

        \draw[very thick, midarrow=0.6] (y1) to[out=270, in=210] node[pos=0.4, right] {$L$} (a);
        \hdedge[out=30, in=330]{a}{b};
        \draw[very thick, midarrow=0.7] (b) to[out=150, in=270] (y2);

        \draw[very thick, midarrow=0.3] (y3) to[out=270, in=30] (b);
        \hdedge[out=210, in=150]{b}{a};
        \draw[very thick, midarrow=0.4] (a) to[out=330, in=270] node[pos=0.6, left] {$M$} (y4);

        \draw[very thick, midarrow] (y2) to[out=135, in=30] (B.center) node[above] {$R$} to[out=210, in=150, looseness=0.6] (y1);
        \draw[very thick, midarrow] (y1) to[out=90, in=180] (-0.5, 1) node[above] {$P$} to[out=0, in=90] (y3);
        \draw[very thick, midarrow] (y4) to[out=90, in=0] (0.5, 1) node[above] {$Q$} to[out=180, in=90] (y2);
    \end{tikzpicture}
    \caption{End of the proof of~\eqref{step2i}. Note that $P,Q,R$ could each be a single edge.}
    \label{fig:step2i}
\end{figure}

By~\ref{pathsmeet}, for $i>2$ there are two dipaths $L,M$ in $C^+$ from $\{y_1,y_3\}$ to $\{y_2,y_4\}$ that are not vertex-disjoint, and such 
that their ends are all distinct. By~\eqref{step1i}, it must be that $L$ is from $y_1$ to $y_2$ and $M$ is from $y_3$ to $y_4$.
If $L,M$ have intersection number one, let $u\in L\cap M$; then $L[y_1,u]\cup M[u,y_4]$ and $M[y_3,u]\cup L[u, y_2]$ are $C[1,4]$ and $C[3,2]$ paths that intersect, contradicting~\eqref{step1i}. Thus $L,M$ do not have intersection number one, so there exist $a,b$ such that $y_1,a,b,y_2$ are in order in $L$ and
$y_3,b,a,y_4$ are in order in $M$. But then $R\cup L$ and $P\cup M\cup Q\cup R$ are both dicycles, and they disagree
on $\{a,b,y_1\}$, a contradiction. (See Figure~\ref{fig:step2i}.) This proves~\eqref{step2i}.

\bigskip
Consequently, $A[1,3], A[4,2]$ exist and intersect, and therefore $A[1,2],A[4,3]$ exist. Similarly,
$B[3,1], B[2,4]$ exist and intersect, and $B[3,4], B[2,1]$ exist. Thus, each vertex in $Y$ has a neighbour in $V(B)$ and a neighbour 
in $V(A)$, and at least three vertices in $Y$ have a neighbour in $V(C)$ since 
$G$ is 3-weak. Thus the underlying graph $G^-$ contains a $K_{3,3}$ minor, so $G$ is nonplanar.

\begin{step}\label{step3i}
There do not exist distinct vertices $u,v$ in both $A[1,3], A[4,2]$ such that $y_1,u,v,y_3$ are in order in $A[1,3]$ and 
$y_4,v,u,y_2$ are in order in 
$A[4,2]$. Consequently, if $A[1,3], A[4,2]$ are in bubble form, they have intersection number one, and the same for $B[3,1], A[2,4]$.
\end{step}

Suppose that $u,v$ belong to both $A[1,3], A[4,2]$ such that $y_1,u,v,y_3$ are in order in $A[1,3]$ and 
$y_4,v,u,y_2$ are in order in 
$A[4,2]$. Choose $w$ in both $B[3,1], B[2,4]$. Then $A[1,3]\cup B[3,1]$ and $A[4,2]\cup B[2,4]$ are dicycles that disagree on $\{u,v,w\}$, a contradiction. This proves~\eqref{step3i}.

\begin{step}\label{step4i}
There is a choice of the three dipaths $A[1,3], A[4,2], A[4,3]$ such that the intersection of each pair of them is a path and the intersection of all three is null. Similarly there is a choice of $B[3,1], B[2,4], B[2,1]$ with the same property.
\end{step}

If $y_4y_3\in E(G)$ then the claim is true by~\eqref{step3i}, so we assume that $y_4y_3\notin E(G)$.
There are two internally disjoint dipaths $P,Q$ in $G$ from $y_4$ to $y_3$, both of length at least two. By~\eqref{step2i}, the second vertex of
$P$ is not equal to $y_2$, and so it belongs to $V(A)$. Similarly, the penultimate vertex of $P$ belongs to $V(A)$.
Suppose that $P$ has some internal vertex not in $V(A)$. Since $y_1\to A$, $P\cap A$ contains two disjoint subpaths, one from 
$y_4$ to $y_2$, and the other from $y_1$
to $y_3$, contrary to~\eqref{step2i}. Hence all vertices of $P$ belong to $V(A)\cup Y$, and so $y_1,y_2\notin V(P)$ (because $y_1\to A$ and $y_2\from A$, and so neither is an interior vertex of a dipath of $A^+$). 
Thus $P$ is a choice for $A[4,3]$, and so is $Q$.

Let $R$ be a choice of 
$A[1,3]$ and $S$ a choice of $A[4,2]$, and, let us choose $R,S$ with $P\cup Q\cup R\cup S$ minimal. 
Let $R_1$ be the 
minimal path of $R$ from $y_1$ to $V(P\cup Q)$, and we may assume that $R_1$ has ends $y_1$ and $r\in V(P)$. (Possibly $r=y_3$,
but if not then $R_1$ is disjoint from $Q$.) Similarly, let $S_1$ be the minimal subpath of $S$ from $V(P\cup Q)$ to $y_2$, with ends 
$s,y_2$. Thus $s$ belongs to one of $P,Q$.

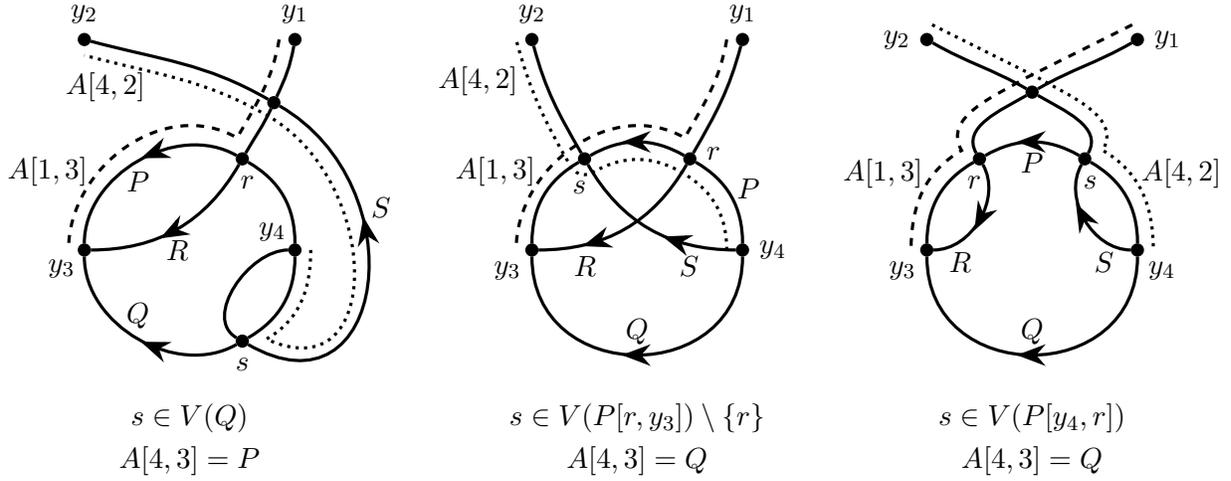
\begin{figure}[htb!]
    \centering
    \begin{tikzpicture}[scale=1.4]
        \vtx[above:$y_2$]{y2}{-1,2};
        \vtx[above:$y_1$]{y1}{1,2};
        \vtx[{[label distance=-3pt]below left:$y_3$}]{y3}{-1,0};
        \vtx[{[label distance=-3pt]above left:$y_4$}]{y4}{1,0};
        \vtx[{[xshift=2pt]below:$r$}]{r}{0.5,0.866};
        \vtx[below:$s$]{s}{0.5,-0.866};
        \vtx{t}{0.8, 1.4};

        \draw[very thick, midarrow=0.6] (y4) to[out=90, in=330] (r) to[out=150, in=90] node[midway, below] {$P$} (y3);
        \draw[very thick, midarrow=0.6] (y4) to[out=270, in=30] (s) to[out=210, in=270] node[midway, above] {$Q$} (y3);

        \draw[very thick, midarrow=0.7] (y1) to[out=260, in=65] (t) to[out=245, in=60] (r) to[out=240, in=0] node[midway, below] {$R$} (y3);
        \draw[very thick, midarrow] (y4) to[out=180, in=150] (s) to[out=330, in=330, looseness=1.7] node[pos=0.7, right] {$S$} (t) to[out=150, in=345] (y2);

        \draw[very thick, dashed] ([shift={(-0.15,0)}]y1.center) to[out=260, in=60] ([shift={(-0.05,0.2)}]r.center) to[out=150, in=90, looseness=1.05] node[pos=0.65, left] {$A[1,3]$} ([shift={(-0.15,0)}]y3.center);

        \draw[very thick, dotted] ([shift={(0.15,0)}]y4.center) to[out=270, in=30] ([shift={(0.23,0.02)}]s.center) to[out=335, in=330, looseness=1.5] ([shift={(-0.05,-0.13)}]t.center) to[out=150, in=345] node[pos=0.9, below] {$A[4,2]$} ([shift={(0,-0.15)}]y2.center);

        \node at (0, -1.6) {$s\in V(Q)$};
        \node at (0, -2) {$A[4,3]=P$};

        \vtx[above:$y_2$]{y2}{3.25,2};
        \vtx[above:$y_1$]{y1}{5.25,2};
        \vtx[{[label distance=-2pt]below left:$y_3$}]{y3}{3.25,0};
        \vtx[right:$y_4$]{y4}{5.25,0};
        \vtx[{[yshift=2pt]right:$r$}]{r}{4.75,0.866};
        \vtx[{[label distance=2pt, xshift=-2pt]below:$s$}]{s}{3.75,0.866};

        \draw[very thick, midarrow] (y4) to[out=90, in=330] node[pos=0.6, right] {$P$} (r) to[out=150, in=30] (s) to[out=210, in=90] (y3);
        \draw[very thick, midarrow] (y4) to[out=270, in=270, looseness=1.6] node[midway, above] {$Q$} (y3);

        \draw[very thick, midarrow=0.8] (y1) to[out=255, in=60] (r) to[out=240, in=0] node[pos=0.75, below] {$R$} (y3);
        \draw[very thick, midarrow=0.2] (y4) to[out=180, in=300] node[pos=0.25, below] {$S$} (s) to[out=120, in=285] (y2);

        \draw[very thick, dashed] ([shift={(-0.15,0)}]y1.center) to[out=260, in=60] ([shift={(-0.05,0.2)}]r.center) to[out=150, in=90, looseness=1.05] node[pos=0.65, left] {$A[1,3]$} ([shift={(-0.15,0)}]y3.center);

        \draw[very thick, dotted] ([shift={(-0.15,0)}]y4.center) to[out=90, in=330] ([shift={(-0.075,-0.13)}]r.center) to[out=150, in=30] ([shift={(0.075,-0.13)}]s.center) to[out=210, in=300] ([shift={(-0.13,-0.075)}]s.center) to[out=120, in=285] node[pos=0.7, left] {$A[4,2]$} ([shift={(-0.15,0)}]y2.center);

        \node at (4.25, -1.6) {$s\in V(P[r,y_3])\setminus\{r\}$};
        \node at (4.25, -2) {$A[4,3]=Q$};

        \vtx[left:$y_2$]{y2}{7,2};
        \vtx[right:$y_1$]{y1}{9,2};
        \vtx[{[label distance=-2pt]below left:$y_3$}]{y3}{7,0};
        \vtx[{[label distance=-2pt]below right:$y_4$}]{y4}{9,0};
        \vtx[{[xshift=2pt]below:$s$}]{s}{8.5,0.866};
        \vtx[{[xshift=-2pt]below:$r$}]{r}{7.5,0.866};
        \vtx{t}{8, 1.5};

        \draw[very thick, midarrow] (y4) to[out=90, in=330] (s) to[out=150, in=30] node[midway, below] {$P$} (r) to[out=210, in=90] (y3);
        \draw[very thick, midarrow] (y4) to[out=270, in=270, looseness=1.6] node[midway, above] {$Q$} (y3);

        \draw[very thick, midarrow=0.8] (y1) to[out=210, in=30] (t) to[out=210, in=120] (r) to[out=300, in=0] node[pos=0.75, below] {$R$} (y3);
        \draw[very thick, midarrow=0.2] (y4) to[out=180, in=240] node[pos=0.25, below] {$S$} (s) to[out=60, in=330] (t) to[out=150, in=330] (y2);

        \draw[very thick, dashed] ([shift={(-0.05,0.14)}]y1.center) -- ([shift={(-0.05,0.14)}]t.center) to[out=210, in=120] ([shift={(-0.17,0.08)}]r.center) to[out=210, in=90] node[pos=0.3, left] {$A[1,3]$} ([shift={(-0.15,0)}]y3.center);

        \draw[very thick, dotted] ([shift={(0.05,0.14)}]y2.center) -- ([shift={(0.05,0.14)}]t.center) to[out=330, in=60] ([shift={(0.17,0.08)}]s.center) to[out=330, in=90] node[pos=0.3, right] {$A[4,2]$} ([shift={(0.15,0)}]y4.center);

        \node at (8, -1.6) {$s\in V(P[y_4,r])$};
        \node at (8, -2) {$A[4,3]=Q$};
    \end{tikzpicture}
    \caption{The three cases in the proof of~\eqref{step4i}.}
    \label{fig:step4i}
\end{figure}

Suppose first that $s\in V(Q)$. By~\eqref{step3i} applied to $R_1\cup P[r,y_3]$ and $Q[y_4,s]\cup S_1$,
it follows that $R_1,S_1$ have intersection number one, and by the minimality of $P\cup Q\cup R\cup S$ it follows that 
$R_1\cap S_1$ is a path. But then the claim holds, taking
\begin{align*}
    A[1,3] &= R_1\cup P[r,y_3], \\
    A[4,2] &= Q[y_4,s]\cup S_1, \\
    A[4,3] &= P.
\end{align*}

Thus we may assume that $s\notin V(Q)$, and so $s\in V(P)$. Next suppose that $s$ belongs to $P[r,y_3]$ and $s\ne r$. Then from~\eqref{step3i} applied to $R_1\cup P[r,y_3]$ and $S$,
it follows that $R_1,S_1$ are disjoint, and so the claim holds, taking
\begin{align*}
    A[1,3] &= R_1\cup P[r,y_3], \\
    A[4,2] &= P[y_4,s]\cup S_1, \\
    A[4,3] &= Q.
\end{align*}

So we may 
assume that $s$ belongs to $P[y_4,r]$.  By~\eqref{step3i} applied to $R_1\cup P[r,y_3]$ and $P[y_4,s]\cup S_1$, we have that $R_1,S_1$ have intersection number one, and so 
by the minimality of $P\cup Q\cup R\cup S$ it follows that
$R_1\cap S_1$ is a path. But then taking
\begin{align*}
    A[1,3] &= R_1\cup P[r,y_3], \\
    A[4,2] &= P[y_4,s]\cup S_1, \\
    A[4,3] &= Q
\end{align*}
again satisfies the claim. (See Figure~\ref{fig:step4i}.) This proves~\eqref{step4i}.

\bigskip
Let $W_1$ be the $Y$-wing obtained from $A$ by adding $Y$, all edges between $Y$ and $V(A)$, and the edge $y_4y_3$ if it exists.
Let $W_2$ be the $Y$-wing obtained from $B$ by adding $Y$, all edges between $Y$ and $V(B)$, and the edge $y_2y_1$ if it exists.
Let $W_3$ be the $Y$-wing such that $W_1,W_2,W_3$ are pairwise internally disjoint and have union $G$.
Then we claim:
\begin{itemize}
\item $y_1$ is a source of $W_1\cup W_3$, and $y_2$ is a sink of $W_1\cup W_3$, and $y_3$ is a source of $W_2\cup W_3$, and $y_4$ is 
a sink of $W_2\cup W_3$. To see this, it is clear that $y_1$ is a source of $W_1$, but we must check that it is a source of $W_3$, and similarly 
we must check the other three statements for $W_3$. For each $i>2$, we have $y_1,y_3\to D_i$ and $y_2,y_4\from D_i$, so 
we only need to check the edges of $W_3$ with both ends in $Y$. Let $yy'\in E(G)$, where $y,y'\in Y$. We must show that either 
$(y,y') = (y_4,y_3)$
(in which case this edge is included in $W_1$, not $W_3$), or $(y,y') = (y_2,y_1)$ (similarly), or $y\ne y_2,y_4$ and $y'\ne y_1,y_3$. Suppose that $y\in \{y_2,y_4\}$.
By~\eqref{step2i} and since $A[1,3], B[3,1]$ exist, it follows that $y'\notin \{y_2,y_4\}$. But there are no edges $y_4y_1$ or $y_2y_3$, by hypothesis, so we may assume that $y\notin \{y_2,y_4\}$, and similarly
$y'\notin \{y_1,y_3\}$, as required. 
\item There is a dipath of $W_1$ from $y_1$ to $y_2$, and there is a dipath of $W_2$ from $y_3$ to $y_4$. This is true since $A[1,2]$
and $B[3,4]$ exist. 
\item The digraph $G_1$ obtained
from $W_1$ by adding a new vertex $v_1$ and the edges
\[ v_1y_1,y_2v_1, y_3v_1,v_1y_4, y_2y_1 \]
is 1-strong and weightable. This is true by~\ref{butterfly} since this digraph is a butterfly minor of $G$, obtained from $W_1\cup B[3,1]\cup B[2,4]\cup B[2,1]$ by contracting singular edges, where
$B[3,1], B[2,4],\allowbreak B[2,1]$ are as in~\eqref{step4i}.
\item The digraph $G_2$
obtained from $W_2$ by adding a new vertex $v_2$ and the edges
\[ y_1v_2,v_2y_2, v_2y_3, y_4v_2, y_4y_3 \]
is 1-strong and weightable. 
This is true for the same reason as the previous bullet.
\item The digraph $G_0$ obtained from $W_3$ by adding the edges $y_1y_3,y_3y_1$ and making the identifications $y_1=y_2$ and $y_3=y_4$
is 1-strong and weightable. This is true by~\ref{butterfly} since this digraph is a butterfly minor of $G$, obtained by contracting singular edges of 
$W_3\cup A[1,3]\cup A[4,3]\cup B[3,1]\cup B[2,1]$, where $A[1,3], A[4,3], B[3,1], B[2,1]$ are as in~\eqref{step4i}.
\end{itemize}
This proves~\ref{secondoutcome}.\end{proof}

Finally, we handle the third outcome:
\begin{thm}\label{thirdoutcome}
Let $y_1,\dots, y_4$ be distinct vertices of a 2-strong, 3-weak, weightable digraph $G$, and let $Y=\{y_1,\dots, y_4\}$. Suppose that $G\setminus Y$ has at least two weak components $D_i$ such that $y_1, y_3\to D_i$ and $y_2,y_4\from D_i$ for $i=1,2$. Then $G$ is nonplanar, and
$G$ can be built from two smaller weightable digraphs by the construction of~\ref{thirdcon}.
\end{thm}
\begin{proof}
Let $A=D_1$, $B=D_2$, and $C=G\setminus V(A\cup B)$. 
For distinct $y_i,y_j\in Y$, $A[i,j]$ means a dipath from $y_i$ to $y_j$ with all internal vertices in $V(A)$, and $B[i,j]$ is defined analogously. By $C[i,j]$ we mean a dipath of $C$ from $y_i$ to $y_j$ with no internal vertex in $Y$.

\begin{step}\label{step1j}
We may assume that there are choices of $C[2,1], C[4,3]$ that are vertex-disjoint.
\end{step}

Since $G$ is 2-strong, there are two vertex-disjoint dipaths from $\{y_2,y_4\}$ to $\{y_1,y_3\}$, and by exchanging $y_1,y_3$ if necessary,
we may assume that there are disjoint dipaths from $y_2$ to $y_1$ and from $y_4$ to $y_3$. Since neither path has any internal vertex 
in $Y$, and $y_1,y_3\to A\cup B$, it follows that both paths are paths of $C$. This proves~\eqref{step1j}.

\bigskip
Since $V(A)\ne \emptyset$ and $G$ is 2-strong, it follows as usual that $A[1,2],A[1,4],A[3,2],A[3,4]$ exist, and the same for $B$.

\begin{step}\label{step2j}
$A[1,4], A[3,2]$ are vertex-disjoint for every choice of $A[1,4], A[3,2]$, and the same for $B$.
\end{step}

Suppose there is a vertex $c$ in both  $A[1,4], A[3,2]$. Then, choosing $C[2,1], C[4,3]$ as in~\eqref{step1j},
\begin{align*}
&A[1,4]\cup C[4,3]\cup B[3,2]\cup C[2,1],\\
&A[3,2]\cup C[2,1]\cup B[1,4]\cup C[4,3]
\end{align*}
are dicycles that disagree on $\{y_1,y_3,c\}$, contradicting~\ref{triple}. This proves~\eqref{step2j}.

\begin{step}\label{step3j}
There are choices of $A[1,2], A[3,4]$ that intersect, and the same for $B$.
\end{step}

This is immediate from~\eqref{step2j} and~\ref{pathsmeet}.

\begin{step}\label{step4j}
$C[4,1], C[2,3]$ intersect for every choice of $C[4,1], C[2,3]$, and hence $G$ is nonplanar.
\end{step}

Otherwise, choosing $A[1,2]$ and $A[3,4]$ as in~\eqref{step3j} and $c$ in both paths, 
\begin{align*}
&A[3,4]\cup C[4,1]\cup B[1,2]\cup C[2,3],\\
&A[1,2]\cup C[2,3]\cup B[3,4]\cup C[4,1]
\end{align*}
are dicycles that disagree on $\{y_1,y_3,c\}$, a contradiction. This proves the first claim. Since the common vertices of 
$C[4,1], C[2,3]$ are not in $Y$, the internal vertices of $C[4,1], C[2,3]$ all belong to a weak component $D$ of $G\setminus Y$ and
each of $y_1,\dots, y_4$ has a neighbour in $V(D)$. Since each of $y_1,\dots, y_4$ also has a neighbour in $V(A)$ and a neighbour in $V(B)$, the underlying graph
$G^-$ has a $K_{3,3}$ minor, so $G$ is nonplanar. This proves~\eqref{step4j}.

\begin{step}\label{step5j}
There is a choice of $A[1,2]$ and $A[3,4]$ that make a bubble, and the same for $B$.
\end{step}

By~\eqref{step3j}, we may choose $P=A[1,2]$ and $Q=A[3,4]$ in bubble form. If they have intersection number one, choose $v\in P\cap Q$. Then $P[y_1,v]\cup Q[v,y_4]$ and $Q[y_3,v]\cup P[v,y_2]$ are $A[1,4],A[3,2]$ paths that intersect, contradicting~\eqref{step2j}. So $P\cap Q$ consists
of at least two disjoint paths $R_1,\dots, R_k$. 
Suppose that $k\ge 3$, and choose $p_i\in V(R_i)$ for $i = 1,2,3$. Then, by~\ref{bubble}, the dicycles 
$P\cup C[2,1]$ and $Q\cup C[4,3]$ disagree on $\{p_1,p_2,p_3\}$, a contradiction. Thus $k=2$, and this proves~\eqref{step5j}.

\begin{step}\label{step6j}
$C[2,1], C[4,3]$ are vertex-disjoint for every choice of $C[2,1], C[4,3]$.
\end{step}

Suppose some vertex $c$ belongs to both $C[2,1], C[4,3]$, and choose $A[1,2],A[3,4]$  as in\eqref{step5j}. Choose $p_1,p_2$ in both
$A[1,2],A[3,4]$ such that $y_1,p_1,p_2,y_2$ are in order in $A[1,2]$ and $y_3,p_2,p_1,y_4$ are in order in $A[3,4]$. Then 
$A[1,2]\cup C[2,1]$, $A[3,4]\cup C[4,3]$ are dicycles that disagree on $\{p_1,p_2,c\}$, a contradiction. This proves~\eqref{step6j}.

\begin{step}\label{step7j}
There is a choice of $C[4,1], C[2,3]$ that make a bubble.
\end{step}

From~\eqref{step4j} we may choose $C[4,1], C[2,3]$ in bubble form. By~\eqref{step6j}, their intersection is not one path, so it is the disjoint union of at 
least two. If it is the disjoint union of at least three paths, choose vertices $c_1,c_2,c_3$ from three of these paths; 
then $C[4,1]\cup A[1,4]$, $C[2,3]\cup A[3,2]$ disagree on $\{c_1,c_2,c_3\}$, a contradiction. So $C[4,1], C[2,3]$ make a bubble. This proves~\eqref{step7j}. 

\bigskip

Let $W_1$ be the $Y$-wing obtained from $A\cup B$ by adding $Y$ and all edges between $V(A\cup B)$ and $Y$, and let $W_2=C$. Thus,
$W_1,W_2$ are internally disjoint $Y$-wings with union $G$. 
We claim:
\begin{itemize}
\item $y_1,y_3$ are sources of $W_1$ and $y_2,y_4$ are sinks of $W_1$. This is true because $y_1,y_3\to A\cup B$ and $y_2,y_4\from A\cup B$.
\item There are dipaths of $W_1$ from $y_1$ to $y_4$ and from $y_3$ to $y_2$, and
there are dipaths of $W_2$ from $y_2$ to $y_1$ and from $y_4$ to $y_3$. This is true because $A[1,4], A[3,2], C[2,1],C[4,3]$ exist.
\item The digraph $G_1$  obtained from $W_1$ by adding the edges $y_2y_1,y_4y_3, y_1y_3,y_3y_1$ is 1-strong and weightable. This is true by~\ref{butterfly} since
$G_1$ is a butterfly minor of $G$, obtained by contracting singular edges of $W_1\cup C[4,1]\cup C[2,3]$, where $C[4,1], C[2,3]$ are chosen as in~\eqref{step7j}.
\item The digraph $G_2$ obtained from $W_2$ by adding two new vertices $v_1,v_2$ and the edges
\[y_1v_1,y_3v_2, v_1y_4, v_2y_3, y_3y_4\] 
is 1-strong and weightable. This is true by~\ref{butterfly} since $G_2$ is a butterfly minor of $G$, obtained by contracting singular edges of $W_2\cup A[1,2]\cup A[3,4]\cup B[3,4]$, where $A[1,2], A[3,4]$ are as in~\eqref{step5j}.
\end{itemize}
This proves~\ref{thirdoutcome}.\end{proof}

We deduce:
\begin{thm}\label{diplanar}
Every 2-strong 3-weak weightable planar digraph is diplanar.
\end{thm}
\begin{proof}
We assume that $G$ is  2-strong, 3-weak, weightable and not diplanar, and we need to show that $G$ is not planar. Since 
$G$ is weightable, $G\ne F_7$, and it follows from~\ref{unevendecomp} that $G$ satisfies one of the three
outcomes of that theorem, and so $G$ is nonplanar by \ref{firstoutcome}, \ref{secondoutcome}, or~\ref{thirdoutcome}. This proves~\ref{diplanar}.\end{proof}

And our second main theorem:
\begin{thm}\label{mainthm2}
Every 1-strong weightable digraph can be built by means of the  constructions of Sections~\ref{sec:easy} and~\ref{sec:nonplanarcon} from 1-strong diplanar digraphs.
\end{thm}
\begin{proof}
Let $G$ be a weightable digraph; we prove that $G$ can be so constructed by induction on $|V(G)|$. We can assume that $G$ is simple.
We may assume that $G$ is 2-strong and 3-weak, because otherwise
it can be built from smaller weightable digraphs by the constructions of Sections~\ref{sec:easy}. If $G$ is planar, then it is
diplanar by~\ref{diplanar} and we are done. Otherwise, by~\ref{unevendecomp}
and \ref{firstoutcome}, \ref{secondoutcome}, and~\ref{thirdoutcome}, $G$ can be built by the constructions of Section~\ref{sec:nonplanarcon} from smaller 1-strong digraphs, and the result follows from the inductive hypothesis. This proves~\ref{mainthm2}.\end{proof}

Finally, let us derive an algorithm to test whether a digraph $G$ is weightable. Let $n=|V(G)|+|E(G)|$; we will give an algorithm with running time
$O(n^6)$. Given $G$, we can search for one of the decompositions listed in \ref{summary} (or deduce that $G$ is circular)
in time $O(n^5)$. The slowest case is testing for the second and third nonplanar decompositions, which can be done by testing each set of 
four vertices for the necessary properties to be the set $Y$ (this test can be done in linear time in $E$, for a total running time of $O(n^5)$).
If we find a decomposition, continue recursively, until either we have constructed $G$ from circular digraphs (and it is therefore weightable),
or one of the parts does not admit any further decomposition and yet is not circular (when $G$ is not weightable).
Then the induction adds one more factor of $n$ to the total running time, since in the worst case every step of the decomposition splits 
the graph into one piece of size $n-c_1$ and one or two pieces of size $c_2$.

\section*{Acknowledgements}
Thanks to Stephen Bartell, Maria Chudnovsky, Julien Codsi, Tung Nguyen, Alex Scott and Raphael Steiner, who at various times
discussed with us the question answered in this paper.

\end{document}